\documentclass[11pt, reqno]{amsart}

\usepackage{amssymb}
\usepackage{amsmath}
\usepackage{amsthm}
\usepackage{graphicx}
\usepackage{braket}
\usepackage{pdfpages}
\usepackage{graphicx}
\usepackage{caption}
\usepackage{subcaption}
\usepackage{mathrsfs} 
\usepackage{tikz-cd} 
\usepackage{bm}
\usepackage{enumitem}
\usepackage{mathtools}
\usepackage{pgfplots}
\usepackage{import}
\usepackage{xifthen}
\usepackage{transparent}
\usepackage{mathtools}

\usepackage[utf8]{inputenc}
\usepackage[english]{babel}

\everymath{\displaystyle}

\usepackage{xargs}                      
\usepackage[colorinlistoftodos,prependcaption,textsize=tiny]{todonotes}
\newcommandx{\unsure}[2][1=]{\todo[linecolor=red,backgroundcolor=red!25,bordercolor=red,#1]{#2}}
\newcommandx{\change}[2][1=]{\todo[linecolor=blue,backgroundcolor=blue!25,bordercolor=blue,#1]{#2}}
\newcommandx{\info}[2][1=]{\todo[linecolor=OliveGreen,backgroundcolor=OliveGreen!25,bordercolor=OliveGreen,#1]{#2}}
\newcommandx{\improvement}[2][1=]{\todo[linecolor=Plum,backgroundcolor=Plum!25,bordercolor=Plum,#1]{#2}}

\usepackage{hyperref}
\usepackage{cleveref}

\usepackage{lmodern,babel,adjustbox,booktabs,multirow}

\makeatletter
\newcommand{\setword}[2]{
  \phantomsection
  #1\def\@currentlabel{\unexpanded{#1}}\label{#2}%
}

\crefformat{footnote}{#2\footnotemark[#1]#3}

\numberwithin{equation}{section}
\theoremstyle{plain}
\newtheorem{theorem}{Theorem}[section]

\theoremstyle{theorem}
\newtheorem{prop}[theorem]{Proposition}
\newtheorem{conj}[theorem]{Conjecture}
\newtheorem{lem}[theorem]{Lemma}
\newtheorem{cor}[theorem]{Corollary}

\newtheorem{example}[theorem]{Example}

\newtheorem*{question*}{Question}

\theoremstyle{plain}

\newtheorem{maintheorem}{Theorem}

\theoremstyle{definition}
\newtheorem{defn}[theorem]{Definition}
\newtheorem{rmk}[theorem]{Remark}

\newcommand{\full}{{\operatorname{full}}}

\newcommand{\R}{\mathbb{R}}
\newcommand{\C}{\mathbb{C}}
\newcommand{\B}{\mathbb{B}}

\newcommand{\Q}{\mathbb{Q}}
\newcommand{\Z}{\mathbb{Z}}
\newcommand{\Hyp}{\mathbb{H}}

\newcommand{\N}{\mathbb{N}}
\newcommand{\D}{\mathbb{D}}

\newcommand{\RT}{\mathscr{T}}
\newcommand{\RV}{\mathscr{V}}

\newcommand{\fm}{\text{fm}}

\DeclareMathOperator{\dist}{dist}
\DeclareMathOperator{\Rat}{Rat}

\DeclareMathOperator{\PSL}{PSL}
\DeclareMathOperator{\M}{\mathcal{M}}

\DeclareMathOperator{\Isom}{Isom}

\DeclareMathOperator{\CP}{CP}

\DeclareMathOperator{\Aut}{Aut}
\DeclareMathOperator{\Int}{Int}
\DeclareMathOperator{\chull}{Cvx\, Hull}

\DeclareMathOperator{\proj}{proj}
\DeclareMathOperator{\Mod}{mod}

\numberwithin{figure}{section}

\newcommand{\hide}[1]{}

\newcommand{\wC}{\widehat{\mathbb{C}}}

\newcommand{\upbullet}[1]{\overset{\bullet}{ #1}{}}

\renewcommand{\sp}{{\ \ }}

\DeclareFontFamily{U}{mathb}{\hyphenchar\font45}
\DeclareFontShape{U}{mathb}{m}{n}{ <5> <6> <7> <8> <9> <10> gen * mathb <10.95> mathb10 <12> <14.4> <17.28> <20.74> <24.88> mathb12 }{}
\DeclareSymbolFont{mathb}{U}{mathb}{m}{n}
\DeclareFontSubstitution{U}{mathb}{m}{n}
\DeclareMathSymbol{\selfmap}{3}{mathb}{"FD}

\newcommand{\Disk}{\mathbb{D}}

\newcommand{\Fam}{{\mathcal F}}
\newcommand{\Width}{{\mathcal W}}

\newcommand\ZZ{\mathcal{Z}}

\newcommand\RR{\mathcal{R}}

\newcommand\XX{\mathcal{X}}

\newcommand\intr{\operatorname{int}}

\newcommand{\bM}{{\mathbf M}}
\newcommand{\bbv}{{\mathbf v}}

\newcommand{\wZ}{{\widehat Z}}

\newcommand{\egm}{{\operatorname{egm}}}

\newcommand{\bK}{{\mathbf K}}

\newcommand{\bbm}{{\mathbf m}}

\newcommand{\bbw}{{\mathbf w}}

\newcommand{\qq}{\mathfrak{q}}

\newcommand{\length}{{\mathfrak l}}

\newcommand\inn{{\operatorname{inn}}}

\title[Sierpinski Hyperbolic Components of Disjoint Type are Bounded]{Sierpinski Carpet Hyperbolic Components of Disjoint Type are Bounded}
\begin{document}
\begin{author}[D.~Dudko]{Dzimitry Dudko}
\address{Institute for Mathematical Sciences, Stony Brook University, 100 Nicolls Rd, Stony Brook, NY 11794-3660, USA}
\email{dzmitry.dudko@stonybrook.edu}
\end{author}

\begin{author}[Y.~Luo]{Yusheng Luo}
\address{Department of Mathematics, Cornell University, 212 Garden Ave, Ithaca, NY 14853, USA}
\email{yusheng.s.luo@gmail.com}
\end{author}

\begin{abstract} We establish certain uniform \emph{a priori} bounds for hyperbolic components of disjoint type. As an application, we will prove that Sierpinski carpet hyperbolic components of disjoint type are bounded. Furthermore, we show that for each map $f$ on the closure of such a hyperbolic component, there exists a quadratic-like restriction around every non-repelling periodic point. Extensions of these results to non-Sierpinski configurations are underway. As a prototype example, we describe the post-critical set of any map on the boundary of the hyperbolic component of $z^2$.
\end{abstract}

\maketitle

\setcounter{tocdepth}{1}
\tableofcontents

\section{Introduction}
A rational map $f: \widehat\C \longrightarrow \widehat\C$ is called {\em hyperbolic} if every critical point of $f$ converges to an attracting periodic cycle under iteration.
For our purposes, it is convenient to mark all the fixed points, and consider the fixed point marked rational maps $\Rat_{d, \fm}$ and the corresponding moduli space $\M_{d,\fm} = \Rat_{d, \fm}/ \PSL_2(\C)$ (see \S \ref{subsec:mhc}).
The set of conjugacy classes of hyperbolic maps form an open and conjecturally dense subset of $\M_{d, \fm}$, and a connected component is called a {\em (marked) hyperbolic component}.

Let $\mathcal{H} \subseteq \M_{d, \fm}$ be a hyperbolic component. 
As $[f]$ varies in $\mathcal{H}$, the topological dynamics on the Julia set $J_f$ remains constant, but the geometry of $J_f$ varies.
We say $\mathcal{H}$ is a {\em Sierpinski carpet hyperbolic component} if the Julia set of any map $[f]\in \mathcal{H}$ is a Sierpinski carpet, and it is of {\em disjoint type} if for any map $[f]\in \mathcal{H}$, all critical points of $[f]$ are in pairwise different periodic cycles of Fatou components.
Equivalently, $\mathcal{H}$ is of disjoint type if any map $[f]\in\mathcal{H}$ has exactly $2d-2$ attracting periodic cycles.
For disjoint-type hyperbolic components, the multipliers of attracting cycles provide natural identification:
\begin{equation}
    \label{eq:H is D}
    \boldsymbol{\rho}: \ \mathcal{H}\ \overset{\simeq}{\longrightarrow}\  \Disk^{2d-2}
\end{equation}

Motivated by Thurston's compactness theorem for acylindrical hyperbolic 3-manifold, McMullen conjectured in the early 1990s (see \cite{McM95}) that
\begin{conj}\label{conj:main}
A Sierpinski carpet hyperbolic component $\mathcal{H}$ is bounded in $\M_{d, \fm}$.
\end{conj}

Despite many attempts throughout the decades, the conjecture remains wide open beyond degree $2$ and has been one of the bottlenecks in developing the realization and rigidity theory for rational maps. In this paper, we will prove it in the disjoint type case:
\begin{maintheorem}[Parameter Space]\label{thm:A}
A Sierpinski carpet hyperbolic component $\mathcal{H} \subseteq \M_{d, \fm}$ of disjoint type is bounded in $\M_d$. Moreover,~\eqref{eq:H is D} naturally extends to
\begin{equation}
\label{eq:thm:A}
     \boldsymbol{\rho}: \ \overline{\mathcal{H}}\ \overset{\simeq}{\longrightarrow}\ \overline{\Disk^{2d-2}}.
\end{equation}
\noindent In particular, $\partial{\mathcal H}$ is locally connected. 
\end{maintheorem}

\begin{figure}[ht]
  \centering
  \includegraphics[width=0.8\textwidth]{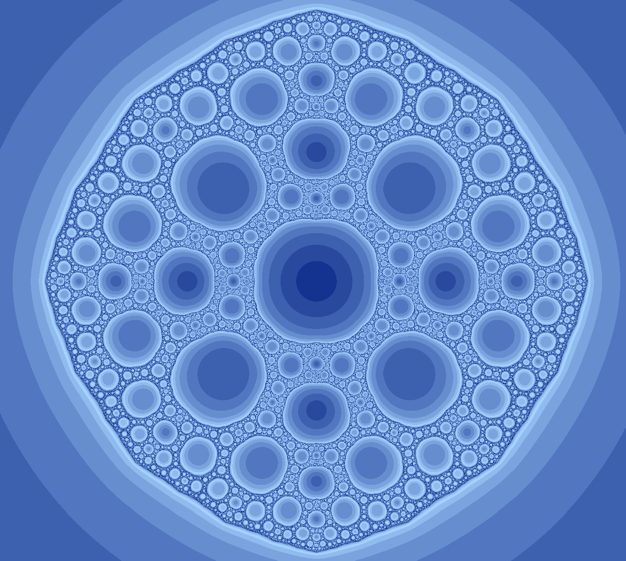}
  \caption{The Julia set of a Sierpinski carpet hyperbolic rational map.}
  \label{fig:SC}
\end{figure}

Our approach also gives uniform bound of the dynamics of maps on $\overline{\mathcal{H}}$:
\begin{maintheorem}[Dynamical Plane]\label{thm:B}
Let $\mathcal{H}$ be a Sierpinski hyperbolic component of disjoint type. 
There exists a constant $\varepsilon>0$ such that for any map $[f] \in \overline{\mathcal{H}}$ and any non-repelling periodic point $x$ of periodic $p$, there exists a quadratic-like restriction $f^p : U \longrightarrow V$, with $x \in U \subseteq V$ and $\Mod(V - U) \geq \varepsilon$. 

Applying the Douady-Hubbard straightening theorem to all quadratic-like restrictions around non-repelling cycles, we obtain the dynamically meaningful refinement of~\eqref{eq:thm:A}: 
\begin{equation}
\label{eq:thm:B}
     \mathcal R_{\operatorname{ql}}: \ \overline{\mathcal{H}}\ \ \overset{\simeq}{\longrightarrow}\ 
 \overline{\Delta}^{2d-2}\ \subset  {\operatorname{Mand}^{2d-2}},
\end{equation}
where $\Delta$ is the main hyperbolic component of the Mandelbrot set.
\end{maintheorem}

 Theorems~\ref{thm:A} and~\ref{thm:B} are applications of Theorem~\ref{thm:cd} which provides uniform geometric \emph{a priori} bounds. These geometric bounds control the postcritical sets of maps on the boundaries of disjoint-type hyperbolic components $\mathcal H$ and depend only on a certain combinatorial ``pulled-off'' constant $N_{Siegel}([f])$. This constant is uniformly bounded in the Sierpinski case. In the follow-up paper~\cite{DLuo}, we will relate $N_{Siegel}([f])$ to the combinatorial distance between $f$ and the ``obstructed boundary'' $\partial^\infty_\Q \mathcal{H}$ (see \S\ref{subsec:boundaryHyerbolicDisjointType}); this would lead to Conjecture~\ref{conj:gdt}. The prototype example of Conjecture \ref{conj:gdt} is Theorem \ref{thm:hypz2} describing the postcritical sets of maps on the boundary of the hyperbolic component of $z\mapsto z^2$.

Upon extensions of pseudo-Siegel bounds beyond the unicritical regime, the methods of this paper might be applicable to all hyperbolic components in the space of rational maps.

The introduction below is organized as follows. In~\S\ref{ss:HistBackgr}, we recall the historical motivation of Conjecture~\ref{conj:main} as a counterpart to Thurston's boundedness theorem of $3$-manifolds. Conjectures~\ref{conj:main} and~\ref{conj:gdt} are also natural continuations of the realization program of topological branched coverings. We state Theorem~\ref{thm:cd} in \S\ref{ss:a priori bounds} after appropriate preparation; see in~\S\ref{sss:RenormTheory} for a brief historical background of the renormalization theory and~\S\ref{sss:EGM:Degen} for a precompactness condition. In~\S\ref{ss:A_B:from:C}, we will sketch the proof of Theorems \ref{thm:A} and \ref{thm:B} assuming Theorem~\ref{thm:cd}. In~\S\ref{subsec:boundaryHyerbolicDisjointType}, we formulate Conjecture~\ref{conj:gdt} describing the obstructed part of the boundary of a general disjoint type hyperbolic component. In~\S\ref{ss:intro:z^2}, we state and prove Theorem~\ref{thm:hypz2} describing the postcritical sets of bineutral quadratics -- this is a prototype and a motivating example for Conjecture~\ref{conj:gdt}. Finally, in~\S\ref{ss:prf:Thm:cd}, we outline the proof of Theorem~\ref{thm:cd}. In the proof, a linear relation is established between the global and local degenerations specified in Definition~\ref{dfn:bounds} and in \S\ref{subsec:sfc}.

\subsection{Historical background}\label{ss:HistBackgr}
Thurston's hyperbolization theorem is one of the most important development in the study of 3-manifolds.
The tools developed along the theorem has revolutionized the theory of Kleinian groups.
In the proof of the hyperbolization theorem,
two boundedness theorems, the {\em double limit theorem} and the {\em Thurston's compactness theorem for acylindrical manifolds}, 
are the key ingredients (see \cite{Kap10, Thu86}).
Based on the Sullivan's dictionary, these two boundedness theorems have natural analogues for rational maps. 
Since convex cocompact acylindrical Kleinian groups have Sierpinski carpet limit sets, Conjecture \ref{conj:main} is the analogue of the Thurston's compactness theorem.

For Kleinian groups, the proof for both boundedness theorems is by contradiction and can be break down into two steps:
\begin{enumerate}[label=\text{(\Alph*)},font=\normalfont,leftmargin=*]
\item \label{step:1}(Geometric part): constructing limiting isometric group actions on $\R$-trees with no global fixed point for any degenerating sequences of Kleinian groups (see \cite{MorganShalen84, Bestvina88, Paulin88});
\item \label{step:2} (Combinatorial/topological part): analyzing possible limiting group actions to get topological decompositions of the underlying 3-manifold (see Rips' theory \cite{Kap10} and Skora's duality theorem \cite{Sko96}).
\end{enumerate}
The contradiction for Thurston's compactness theorem is that acylindrical 3-manifolds do not admit such decomposition constructed in Step~\ref{step:2}. The story is similar for the double limit theorem, except the contradiction comes from geometric constraints of the laminations.

It is already suggested in \cite{McM95} that a similar strategy might work for rational maps.
There have been many constructions of limiting dynamics on trees for degenerations of rational maps (see \cite{McM09, Kiwi15, Luo19a, Luo19b, FG24}). These constructions complete Step~\ref{step:1} for rational maps. In this analogy, Step~\ref{step:2} becomes essential for the boundedness results of rational maps.

The boundary of a hyperbolic component $\mathcal{H}$ consists of two types of points: {\em geometrically finite maps} and {\em geometrically infinite maps}.
To carry out Step~\ref{step:2}, we consider these two cases separately.

Geometrically finite maps are the first to be understood and are essentially determined by a finite set of data (somewhat similar to PCF maps). In \cite{Luo21b}, the second author showed that in a Sierpinski hyperbolic component $\mathcal{H}$ (of any type) a `geometrically finite degenerations' $[f_n]$ always land at geometrically finite parameter $[f_\infty]\in \partial \mathcal{H}$; in particular, $[f_n]$ does not diverge in $\M_d$.  The main step is the following finiteness statement for the dynamics in the limiting tree.
\begin{enumerate}[label=\text{(\alph*)},start=2,font=\normalfont,leftmargin=*]
\item\label{item:b:intro} There is a finite `core' in the limiting tree if $[f_n]$ diverges.
This finiteness induces a decomposition of the rational maps in $\mathcal{H}$ by some {\em limiting Thurston obstruction}.
\end{enumerate}
Similar to Kleinian groups, the contradiction is that Sierpinski carpet Julia set would prevent the existence of such limiting obstructions.

Geometrically infinite maps are more mysterious.
Conjecturally, they all arise as limits of geometrically finite maps.
To study such maps, some uniform bound is usually needed (see \S~\ref{ss:a priori bounds}).
In this paper, we use bounds from renormalization theory developed in \cite{DL22,DLL25} to prove such a uniform bound for a special class of geometrically infinite maps, called {\em eventually-golden-mean maps} (see \S \ref{subsec:egmssm}).
The uniform bound allows us to obtain a limiting map on a finite tree of Riemann spheres.
Similar as in the geometrically finite case, the finiteness allows us construct a decomposition of the rational map in terms of limiting Thurston obstructions, and we obtain a contradiction here (see Theorem \ref{thm:cd} and \S~\ref{ss:prf:Thm:cd}).

Our results are related to the {\em Thurston's realization problem}.
Given a topological branched covering of the sphere $S^2$, Thurston's realization problem asks when it is equivalent to a rational map.
Thurston gives a negative criterion to answer the question for {\em post-critically finite} branched coverings \cite{DH93}. Recently, Dylan Thurston gives a positive criterion for the realization problem for post-critically finite maps \cite{Thu20} (with non-vacuous Fatou dynamics). For geometrically finite maps, 
Thurston's realization problem has been studied extensively (see \cite{DH93, CJS04, CJ11, CT11, CT18}).
Usually, the realization problems for geometrically finite maps are studied by deformation of hyperbolic maps.
Various elementary deformations such as pinching and spinning were constructed and studied (see \cite{Makienko00, TanLei02, HT04, PT04, CT18}).
These operations are generalized in \cite{Luo21a, Luo21b}.
From this perspective, Theorem \ref{thm:A} and Theorem \ref{thm:B} can be interpreted as a Thurston's realization theorem for geometrically infinite maps.
Our method, perhaps for the first time, combines two theories in complex dynamics: bounds from the Thurston theory and the bounds from the renormalization theory; see~\S\ref{sss:RenormTheory} for details regarding the latter theory.

There have been many previous studies to understand deformations of rational maps and related boundedness problems.
Results on unboundedness of hyperbolic components were obtained in \cite{Makienko00,TanLei02}. 
In \cite{Epstein00}, Epstein used algebraic and analytic methods to give the first general boundedness result of hyperbolic component of disjoint type in the quadratic case.
This was later generalized in the {\em bi-critical} setting by Nie and Pilgrim in \cite{NP19}.
Related boundedness results in the degree $4$ Newton family was also proved in \cite{NP20}.
To the best of our knowledge, Theorem \ref{thm:A} is the first time a boundedness theorem of an entire hyperbolic component in $\mathcal{M}_d$ is proved in degree $d\geq 3$.

\begin{rmk}
\label{rem:Khan:MSRI}
In Spring 2022, J. Kahn simultaneously presented an independent approach to the boundedness of Sierpinski hyperbolic components of all types (not necessarily disjoint); see his MSRI-talks at~\cite{Kah22}. To the best of our knowledge, the approaches in~\cite{Kah22} and in our paper are somewhat orthogonal and have different scopes.

The approach in~\cite{Kah22} relies on estimating the degeneration around small Julia sets. This naturally leads to the Sierpinski condition so that small Julia sets are disjoint.

The approach in the current paper relies on estimating the degeneration around small postcritical sets. They are within Siegel disks, while the associated ``small Julia sets'' can have complicated mating-type interactions, cf.,~\S\ref{ss:intro:z^2}. This naturally leads to the disjoint-type condition (see Theorem~\ref{thm:cd} and Remark~\ref{rem:N_Siegel}) where pseudo-Siegel bounds~\cite{DL22,DLL25} are available.
\end{rmk}

\subsection{\emph{A priori} bounds}\label{ss:a priori bounds} As we have already indicated, Theorems~\ref{thm:A} and~\ref{thm:B} are applications of Theorem~\ref{thm:cd} which establishes \emph{a priori} bounds (a precompactness condition) to control the postcritical set of maps on the boundary of a hyperbolic component $ \mathcal H$ so that the bounds depend only on the combinatorial pulled-off constant $N_{Siegel}([f])$.

We will state Theorem~\ref{thm:cd} in \S~\ref{sss:ThmC} after appropriate preparation. In~\S\ref{sss:RenormTheory}, we briefly sketch the evolution of key renormalization ideas and techniques used in the paper. In~\S\ref{sss:EGM}, we introduce the class of eventually-golden-mean maps on $\partial \mathcal H$. A precompactness condition for such a class of maps is stated in~\S\ref{sss:EGM:Degen}; see Definition~\ref{dfn:bounds} and Theorem~\ref{thm:ubd}. An important combinatorial invariant, called {\em the pulled-off constant} $N_{Siegel}([f])$ is introduced in~\S\ref{sss:constant:Pulloff}, on which our \emph{a priori} bounds in Theorem~\ref{thm:cd} depend; see \S\ref{sss:ThmC}.

\subsubsection{Renormalization theory}\label{sss:RenormTheory}
The notion of \emph{a prori} comes from the theory of differential equations. Following the work of Sullivan in the late 1980s -- early 1990s, this notion is widely used in Complex Dynamics as an important initial step for the development of respective Renormalization theory. In particular, it is believed that the MLC conjecture (the Mandelbrot set is locally connected) should follow from an appropriate form of \emph{a priori} bounds for all quadratic polynomials, see~\cite[\S1.5]{DL22} for details. 

In the mid-2000s, Kahn \cite{Kah06} introduced a new machinery, now collectively known as the \emph{Near-Degenerate Regime}, to produce ``quadratic-like'' \emph{a priori} bounds for quadratic polynomials. Together with Lyubich, they set up fundamental tools, such as the Quasi-Additivity Law and the Covering Lemma~\cite{KL09}, and attained substantial progress in the primitive case of the MLC conjecture \cite{KL08, KL09c}. Other applications include the extension of Yoccoz’s results to higher degrees~\cite{KLUnicr, AKLS,KvS09,CDKvS}.

The Near-Degenerate machinery developed in the 2000s addresses the \emph{positive-entropy} case: the ``anti-Molecule'' condition in \cite{KL09c} is equivalent to positivity (i.e., $\ge \varepsilon>0$) of the ``core-entropy'' of the associated combinatorics. It has been known since the 1990s that quadratic-like bounds fail for near-Molecule combinatorics where various near-rotation effects kicks-off. In particular, quadratic-like bounds fail for \emph{non-JLC parameters} where the Julia sets are \emph{not} locally connected. 

A remarkable progress in understanding near-neutral Complex Dynamics (and non-JLC phenomenon) came in the 2000s with the development of the Inou-Shishikura~\cite{IS} near-parabolic renormalization theory in the perturbative regime; some of its most-spectacular applications are the construction of positive area Julia sets~\cite{BC12,AL22}; see~\cite[\S1.3]{DL22} for more discussion.

Renormalization Theory of neutral Siegel maps, also initiated by physicists, was mathematically designed by McMullen in \cite{McMSiegel} in the mid-1990s. In the late-2010s, the hyperbolicity of Siegel Renormalization of bounded type was established in~\cite{DLS20} in the framework of Pacman Renormalization~\cite{DLS20}, providing tools to study near-Siegel maps~\cite{DL23a}. We will use the theory~\cite{DLS20,DL23a} in Theorem~\ref{thm:sctssm} to relate Siegel maps on the boundary $\partial_\egm \mathcal H$ with the center of $\mathcal H$.

In~\cite{DL22}, \emph{pseodu-Siegel a priori bounds} were established for \emph{all} neutral quadratic polynomials by introducing the near-degenerate regime to the zero-netropy setting. Roughly, almost-invariant pseudo-Siegel disks $\wZ^m$ ``hide'' non-JLC issues, and near-degenerate tools are used to establish \emph{a priori} bounds for $\partial \wZ^m$.

\subsubsection{Eventually-golden-mean maps} \label{sss:EGM}
In the remainder of this section, let $\mathcal{H}$ be a hyperbolic component of disjoint type, which may or may not be Sierpinski.
Our strategy is to uniformly control the geometry of the Julia set for some special maps on $\partial \mathcal{H}$, called {\em eventually-golden-mean maps}.

An irrational number $\theta \in (0,1)$ is said to be {\em eventually-golden-mean} if it has a continuous fraction expansion $\theta= [0; a_1,..., a_m, ...]$ with $a_m = 1$ for all large $m$.
A map $[f] \in \partial \mathcal{H}$ is called {\em eventually-golden-mean map} if the multiplier for any of its indifferent periodic cycle is of the form $e^{2\pi i \theta}$, where $\theta$ is eventually-golden-mean. In particular, every critical point of $[f]$ is either in an attracting basin or on the boundary of a Siegel disk.

\subsubsection{Degenerations of compact Riemann surfaces}
To discuss our bound on the geometry, let $X$ be a compact Riemann surface with boundaries.
Let $\gamma$ be a non-peripheral arc connecting $\partial X$, and let $\Gamma_\gamma$ be the family of arcs isotopic to $\gamma$.
We define the {\em degeneration} $\mathcal{W}(\gamma)$ for $\gamma$ of $X$ as the extremal width of the family $\Gamma_\gamma$.
The {\em arc degeneration} for $X$ is
$$
\mathcal{W}_{arc}(X) = \sum_{\gamma: \mathcal{W}(\gamma) \geq 2} \mathcal{W}(\gamma).
$$

Similarly, let $\alpha$ be a homotopically non-trivial simple closed curve, and let $\Gamma_\alpha$ be the family of simple closed curves isotopic to $\alpha$.
We define the {\em degeneration} $\mathcal{W}(\alpha)$ for $\alpha$ of $X$ as the extremal width of $\Gamma_\alpha$.
We define the {\em loop degeneration} for $X$ as
$$
\mathcal{W}_{loop}(X) = \sum_{\alpha: \mathcal{W}(\alpha) \geq 2} \mathcal{W}(\alpha).
$$

By losing $\leq 2$ extremal width of the family, we may assume $\Gamma_\gamma$ and $\Gamma_\alpha$ are laminations.
We remark that since wide families do not cross, both $\mathcal{W}_{arc}(X)$ and $\mathcal{W}_{loop}(X)$ are in fact finite sums.

\subsubsection{Degenerations of eventually-golden-mean maps}\label{sss:EGM:Degen}
Let $[f] \in \partial \mathcal{H}$ be an eventually-golden-mean map.
Denote the list of Siegel disks and attracting Fatou components of $f$ by 
$$
Z_{1,0},... Z_{1,p_1-1}, Z_{2, 0},..., Z_{k_1, p_{k_1}-1} \text{ and } D_{1,0},... D_{1,q_1-1}, D_{2, 0},..., D_{k_2, q_{k_2}-1}.
$$

In \S\ref{sss:ValDomain} and \S\ref{subsec:psd}, we will define the corresponding {\em pseudo-Siegel disks} $\widehat{Z}_{i,j}$ and {\em valuable-attracting domain} $\widehat{D}_{i,j}$.
Some important properties are 
\begin{itemize}
\item $\widehat{Z}_{i,j}$ and $\widehat{D}_{i,j}$ are closed disks, with $\overline{Z_{i,j}} \subseteq \widehat{Z}_{i,j}$ and $\widehat{D}_{i,j} \subseteq D_{i,j}$;
\item $f$ is injective on $\widehat{Z}_{i,j}$ and $\widehat{D}_{i,j}$ is forward invariant under $f$;
\item $\widehat{Z}_{i,j}$ and $\widehat{D}_{i,j}$ contain the critical, post-critical points and the non-repelling periodic point in $\overline{Z_{i,j}}$ and $D_{i,j}$ respectively.
\end{itemize}
To quantify the arc and loop degeneration of the map $[f]$ and provide a precompactness condition, we introduce the following definition.
\begin{defn}
\label{dfn:bounds}
    We say $[f]$ has degeneration bounded by $K$ if there exist 
    \begin{itemize}
    \item $K$-quasiconformal pseudo-Siegel disks $\widehat{Z}_{i,j}$,
    \item $K$-quasiconformal valuable-attracting domains $\widehat{D}_{i,j}$,
\end{itemize}
so that the {\em pseudo-core surface} of $[f]$ (see \S \ref{sec:psd} for more discussions)
$$
\widehat X_f := \widehat\C - \bigcup \Int(\widehat{D}_{i,j}) - \bigcup \Int(\widehat{Z}_{i,j}) \text{ satisfies}
$$
\begin{itemize}
\item  
$\mathcal{W}_{arc}(\widehat X_f) \leq K$; \text{ and }
\item 
$\mathcal{W}_{loop}(\widehat X_f) \leq K$.
\end{itemize}
\end{defn}

We will prove the following boundedness theorems for eventually-golden-mean maps with uniformly bounded degenerations.
\begin{theorem}[Precompactness Condition]\label{thm:ubd}
    Let $[f_n] \in \partial \mathcal{H}$ be a sequence of eventually-golden-mean maps.
    Suppose that $[f_n]$ has degeneration bounded by $K$.
    Then after possibly passing to a subsequence, $[f_n] \to [f] \in \M_{d, \fm}$, and $[f]$ has $2d-2$ non-repelling cycles.
\end{theorem}

To apply Theorem~\ref{thm:ubd}, it is thus essential to have a uniform degeneration bound for eventually-golden-mean maps, which is stated in Theorem~\ref{thm:cd}. We remark that $K$-quasiconformality of valuable-attracting domains $\widehat{D}_{i,j}$ is controlled by the associated multpliers of attracting cycles. On the other hand, $K$-quasiconformality of pseudo-Siegel disks $\widehat{Z}_{i,j}$ is controlled by local degeneration parameters $\mathcal W^+_\lambda(I)$ defined in~\S\ref{subsec:sfc}. A fundamental property used in the proof of Theorem~\ref{thm:cd} is a linear relation between $\mathcal W^+_\lambda(I)$ and $\mathcal{W}_{arc}(\widehat X_f)$, see~\eqref{eq:local:K_f} in~\S\ref{ss:prf:Thm:cd}

\subsubsection{The pulled-off constant}\label{sss:constant:Pulloff}
We now introduce an important combinatorial constant that controls the degenerations of eventually-golden-mean maps.

Two arcs $\gamma_1$ and $\gamma_2$ are said to 
\begin{itemize}
    \item {\em intersect essentially} if for any arcs $(\widetilde{\gamma}_i)_{i\in \{1,2\}}$ homotopic to $(\gamma_i)_{i\in \{1,2\}}$, $\widetilde{\gamma}_1$ intersects $\widetilde{\gamma}_2$; and
    \item {\em intersect laminally} if for any arcs $(\widetilde{\gamma}_i)_{i\in \{1,2\}}$ homotopic and disjoint to $(\gamma_i)_{i\in \{1,2\}}$, $\widetilde{\gamma}_1$ intersects $\widetilde{\gamma}_2$.
\end{itemize}

To justify the notations, let $\mathcal L_1,\mathcal L_2$ be two laminations consisting of homotopic curves. If $\gamma_1\in \mathcal L_1$ intersects laminally $\gamma_2\in \mathcal L_2$, then every curve in $\mathcal L_1$ intersects every curve in $\mathcal L_2$.

A family of arcs $\gamma_i$ are said to be {\em essentially disjoint} (or {\em laminally disjoint}) if no pairs in the family intersect essentially (or laminally).
An essentially (or laminally) disjoint pull back of a map $f$ is an essentially (or laminally) sequence of arcs $\gamma_0,..., \gamma_n$ so that $f: \gamma_{i+1} \longrightarrow \gamma_{i}$ is a homeomorphism for each $i = 0,..., n-1$.

Let $[f] \in \partial \mathcal{H}$ be an eventually-golden-mean map.
Let $\gamma \subseteq \widehat \C - \bigcup \Int(\widehat D_{i,j}) - \bigcup Z_{i,j}$ be a non-peripheral arc connecting boundaries of Siegel disks.
The {\em pulled-off constant} $N(\gamma)$ for $\gamma$ is the smallest number $n$ so that for any laminally disjoint pull back sequence
$\gamma_0 = \gamma, \gamma_1, ..., \gamma_{n}$,
at least one end point of $\partial{\gamma_n}$ is not on the boundary of a periodic Siegel disk.

Similarly, let $[f_{pcf}] \in \mathcal{H}$ be the post-critically finite center of $\mathcal{H}$, and let $P(f_{pcf})$ be its post-critical set.
Let $\gamma$ be a non-peripheral arc in $\widehat \C - P(f_{pcf})$ that connects points in $P(f_{pcf})$.
Its pulled-back constant $N(\gamma)$ is the smallest number $n$ so that for any essentially disjoint pull back sequence
$\gamma_0 = \gamma, \gamma_1, ..., \gamma_{n}$,
at least one end point of $\partial{\gamma_n}$ is not in $P(f_{pcf})$.

\begin{defn}[Pulled-off constant]\label{defn:poc}
Let $[f] \in \partial \mathcal{H}$ be an eventually-golden-mean map.
    The {\em (Siegel) pulled-off constant} for $[f]$ is
$$
N_{Siegel}
([f]) := \sup_{\gamma} N(\gamma)
$$ 
where the supreme is over all non-peripheral arcs connecting boundaries of Siegel disks.

Let $[f_{pcf}] \in \mathcal{H}$ be the post-critically finite center of $\mathcal{H}$.
The {\em pulled-off constant} for $[f_{pcf}]$ is
$$
N([f_{pcf}]) := \sup_{\gamma} N(\gamma)
$$ 
where the supreme is over all non-peripheral arcs connecting the points in the post-critical set.
\end{defn}
In \S \ref{sec:pc}, we will prove that
\begin{itemize}
    \item $N([f_{pcf}]) < \infty$ if and only if $\mathcal{H}$ is Sierpinski.
    \item $N_{Siegel}([f]) \leq N([f_{pcf}])$ for any eventually-golden-mean map $[f] \in \partial \mathcal{H}$, where $\mathcal{H}$ is a Sierpinski hyperbolic component.
\end{itemize}
By combining the above two statements, we see that $N_{Siegel}([f])$ is uniformly bounded if $\mathcal{H}$ is a Sierpinski.

\subsubsection{Key technical theorem}\label{sss:ThmC} The following technical theorem gives the uniform bound for eventually-golden-mean maps, and is the key in our argument.
\begin{maintheorem}[\emph{A priori} bounds]\label{thm:cd}
Let $\mathcal{H}$ be a hyperbolic component of disjoint type, and let $[f] \in \partial \mathcal{H}$ be an eventually-golden-mean map.
Then $[f]$ has degeneration bounded by $K$, where $K$ depends on 
\begin{enumerate}
    \item the hyperbolic component $\mathcal{H}$,
    \item\label{ass:2:thmcd} the pulled-off constant $N_{Siegel}([f])$, and
    \item the multipliers of the attracting cycles of $f$.
\end{enumerate}
\end{maintheorem}

\noindent For applications, Assumption~\eqref{ass:2:thmcd} is the main one in Theorem~\ref{thm:cd}. Therefore, the condition ``$[f_n]$ has degeneration bounded by $K$" in Theorem~\ref{thm:ubd} can be replaced (for practical purposes) with ``$N_{Siegel}([f_n])\le M$" for some $M$.

The current paper develops a mechanism to apply pseudo-Siegel bounds from~\cite{DL22} to various classes of neutral rational maps.  A key challenge is to relate global degeneration parameters of a rational map (i.e.,~$\mathcal{W}_{arc},\mathcal{W}_{loop}$ in Definition~\ref{dfn:bounds}) with local degeneration parameters of respective pseudo-Siegel disks (i.e., $\mathcal{W}_{\lambda}^+(I),\mathcal{W}_{\lambda}(I)$ in \S\ref{subsec:sfc}). We emphasize that such a relation should be linear (see~\eqref{eq:local:K_f}) to make use of global dynamical properties of rational maps as an input for \emph{a priori} bounds, see also~\cite[\S7.4]{DLL25}.

When first announced in 2022, the original implementation of the proof of Theorem~\ref{thm:cd} involved justifying a variant of pseudo-Siegel bounds directly in the dynamical plane of $f$, obtained by adapting arguments from~\cite{DL22}. In~\cite{DLL25}, we developed a $\psi^\bullet$ quadratic-like variant of pseudo-Siegel bounds, allowing us to apply these bounds directly.

\begin{rmk}
\label{rem:N_Siegel}
We remark that if $\mathcal{H}$ is Sierpinski, then 
$$
N_{Siegel}([f]) \leq N([f_{pcf}]) < \infty.
$$
Thus, in this case, the constant $K$ is independent of the pulled-off constant.
We also remark that the constant $K$ is independent of the indifferent multipliers of the map $[f]$. This crucial fact allows us to take the limit of those eventually-golden-mean maps.
\end{rmk}

\subsection{Sketch of the proof of Theorems \ref{thm:A} and \ref{thm:B} from Theorem~\ref{thm:cd}}\label{ss:A_B:from:C} To discuss how Theorem \ref{thm:cd} allows us to prove Theorems \ref{thm:A} and \ref{thm:B}, we start with the following decomposition of the boundary $\partial \mathcal{H}$.
\begin{defn}
    Let $\mathcal{H}$ be a hyperbolic component of disjoint type. The boundary 
$$
\partial \mathcal{H} = \partial_{\operatorname{reg}}\mathcal{H} \sqcup \partial_{\operatorname{exc}}\mathcal{H}\subseteq \mathcal{M}_{d, \fm}
$$ 
splits into the \emph{regular} and \emph{exceptional} parts, where $[f] \in \partial_{\operatorname{reg}}\mathcal{H}$ if $[f]$ has exactly $2d-2$ non-repelling periodic cycles, and $[f] \in \partial_{\operatorname{exc}}\mathcal{H}$ otherwise: at least two non-repelling periodic cycles of $[f]$ collide. 
\end{defn}
We remark that by transversality,  we have that (see Proposition \ref{prop:ts})
\begin{itemize}
    \item the natural extension of the multiplier map ~\eqref{eq:H is D} is an embedding on the regular boundary $\boldsymbol{\rho}: \partial_{\operatorname{reg}}\mathcal{H} \xhookrightarrow{} \partial\D^{2d-2}$;
    \item $\boldsymbol{\rho}(\partial_{\operatorname{reg}}\mathcal{H}) \cap \boldsymbol{\rho}(\partial_{\operatorname{exc}}\mathcal{H}) = \emptyset$.
\end{itemize}
Thus, to prove Theorem \ref{thm:A}, it suffices to show that if $\mathcal{H}$ is Sierpinski, then $\boldsymbol{\rho}(\partial_{\operatorname{reg}}\mathcal{H}) =\partial\D^{2d-2}$.

Denote the boundary of eventually-golden-mean maps and geometrically finite maps by $\partial_{\operatorname{egm}}\mathcal{H}$ and $\partial_{\operatorname{\Q}}\mathcal{H}$ respectively.
If $\mathcal{H}$ is Sierpinski, then $\boldsymbol{\rho}(\partial_{\operatorname{\Q}}\mathcal{H})$ is dense in $\partial \D^{2d-2}$ (see \cite{CT18} or \cite{Luo21b}).
This allows us to show that $\boldsymbol{\rho}(\partial_{\operatorname{egm}}\mathcal{H})$ is dense in $\partial\D^{2d-2}$ (see Proposition \ref{prop:rl}).

Given any multiplier profile $\rho = (\rho_1,..., \rho_{2d-2}) \in \partial \D^{2d-2}$, by Theorem \ref{thm:ubd} and Theorem \ref{thm:cd}, we can construct a convergent sequence of eventually-golden-mean maps $[f_n] \to [f]$ with $\boldsymbol{\rho}([f]) = \rho$.
Since $[f]$ has $2d-2$ non-repelling cycles, $[f] \in \partial_{\operatorname{reg}}\mathcal{H}$.
Thus, $\boldsymbol{\rho}(\partial_{\operatorname{reg}}\mathcal{H}) = \partial\D^{2d-2}$, and Theorem \ref{thm:A} follows.

To prove Theorem \ref{thm:B}, we first construct a semiconjugacy between a map $[f]\in \overline{\mathcal{H}}$, and a topological model $\bar{f}: S^2 \longrightarrow S^2$ which is the quotient map of the post-critical finite map $[f_{pcf}] \in \mathcal{H}$ by collapsing Fatou components.
This allows us to show there exists a quadratic-like restriction near every non-repelling periodic point.
The uniform bound of the modulus in Theorem \ref{thm:B} then follows from the conclusion of Theorem \ref{thm:A} that $\overline{\mathcal{H}}$ is compact; see \S~\ref{ss:proofB} for more details.

\subsection{Boundaries of hyperbolic components of disjoint type}\label{subsec:boundaryHyerbolicDisjointType}
A general hyperbolic components of disjoint type $\mathcal{H}$ may not be bounded in $\mathcal{M}_{d,\fm}$.
We give the following definition to parameterize the boundary at infinity.
\begin{defn}
    Let $\mathcal{H}$ be a hyperbolic component of disjoint type.
    We define the {\em obstructed boundary} 
\begin{multline*}
 \partial^\infty \mathcal{H} =\{\rho \in \partial \D^{2d-2} :\\
 \text{ $\exists [f_n] \in \mathcal{H}$ with }\text{$[f_n] \to \infty$ in $\mathcal{M}_{d, \fm}$ and $\boldsymbol{\rho}([f_n]) \to \rho$} \}.
\end{multline*}

    The {\em rational obstructed boundary} 
    $$
    \partial^\infty_\Q \mathcal{H} = \partial^\infty \mathcal{H} \cap \partial_\Q \D^{2d-2}
    $$ 
    where $\partial_\Q\D^{2d-2}$ consists of tuples $(\rho_1,..., \rho_{2d-2})$ so that all indifferent multipliers are rational, i.e. of the form $e^{2\pi i p/q}$.
\end{defn}
We remark that the rational obstructed boundary $\partial^\infty_\Q \mathcal{H}$ can be identified as obstructed geometrically finite maps on the boundary of $\mathcal{H}$, and can be effectively computed.
We formulate the following conjecture.
\begin{conj}\label{conj:gdt}
Let $\mathcal{H}$ be a hyperbolic component of disjoint type.
Then
    $$
    \partial^\infty \mathcal{H} = \overline{\partial^\infty_\Q \mathcal{H}}.
    $$
In particular, the natural extension of the multiplier map ~\eqref{eq:H is D} gives an identification for the regular boundary:
$$
\partial_{\operatorname{reg}}\mathcal{H} \cong \partial \D^{2d-2} - \overline{\partial^\infty_\Q \mathcal{H}} - \boldsymbol{\rho}(\partial_{\operatorname{exc}}\mathcal{H}).
$$

An appropriate version of Theorem~\ref{thm:B} is applicable to maps on $\partial_{\operatorname{reg}}\mathcal{H}$, where the Douady-Hubbard straightening theorem is replaced by the sector renormalization.
\end{conj}

\subsection*{Remarks about Conjecture~\ref{conj:gdt}, Theorem~\ref{thm:cd}, and  $\boldsymbol{N_{Siegel}([f])}$} An important combinatorial input for Theorem \ref{thm:A} is that if $\mathcal{H}$ is Sierpinski, the pulled-off constant $N_{Siegel}([f])$ is uniformly bounded for any eventually-golden-mean map $[f] \in \partial \mathcal{H}$. 

For general hyperbolic component of disjoint type, $N_{Siegel}([f])$ may be unbounded. To deduce Conjecture~\ref{conj:gdt} from Theorem~\ref{thm:cd} following the lines of~\S\ref{ss:A_B:from:C}, we need a combinatorial bound on the pull-off constant $N_{Siegel}([f])$ in terms of the combinatorial distance between $\boldsymbol{\rho}([f])$ and $\overline{\partial^\infty_\Q \mathcal{H}}$. In the sequel~\cite{DLuo}, we will provide such a bound for $N_{Siegel}([f])$ and prove Conjecture \ref{conj:gdt}. See Theorem~\ref{thm:hypz2} and its proof for illustration.

\subsection{The example of $\mathcal{H}_{z^2}$}\label{ss:intro:z^2}
To illustrate Conjecture~\ref{conj:gdt} and the subtlety about the exceptional boundary, consider the hyperbolic component $\mathcal{H}_{z^2}$ in the moduli space of quadratic rational maps that contains $z^2$.

The (marked) moduli space $\mathcal{M}_{2, \fm}$ of quadratic rational maps can be parameterized by the multipliers of the three marked fixed points $(\rho_1, \rho_2, \rho_3)$, with the restriction (from the holomorphic index formula) $\rho_1\rho_2\rho_3 - (\rho_1 + \rho_2 + \rho_3) + 2 = 0$ (see \cite{Mil93}).
In this coordinate,
$$
\mathcal{H}_{z^2} = \{(\rho_1, \rho_2, \rho_3): |\rho_1|, |\rho_2| < 1, \rho_3 = \frac{2-\rho_1-\rho_2}{1-\rho_1\rho_2}\} \cong \D^2.
$$
A simple computation shows that 
$$
\partial^\infty \mathcal{H}_{z^2} = \{(e^{2\pi i t}, e^{-2\pi i t})\} = \overline{\partial^\infty_\Q \mathcal{H}_{z^2}} = \overline{\{(e^{2\pi i p/q}, e^{-2\pi i p/q})\}}.
$$
Note that that when $\rho_1 = \rho_2 = 1$, $\rho_3$ can be an arbitrary number.
Thus, it is easy to see that the exceptional boundary contains infinitely many maps and fibers over $(1,1)$, i.e.,
$\boldsymbol{\rho}(\partial_{\operatorname{exc}}\mathcal{H}_{z^2}) = \{(1, 1)\}$.
Hence, rigidity fails on the exceptional boundary.
Depending on how the multipliers converge to $(1, 1)$, the corresponding sequence $[f_n]$ can be either bounded or unbounded in $\mathcal{M}_{2, \fm}$.

In the case of $\partial_\text{reg}\mathcal{H}_{z^2}$, the pulled-off constant $N_{Siegel}([f])$ can be explicitly bounded in terms of the difference of rotational speeds. Therefore, applying Theorem \ref{thm:cd}, we obtain:
\begin{maintheorem}[Bi-neutral quadratics]\label{thm:hypz2}
Consider $f\in \partial_\text{reg} \mathcal H_{z^2}$ with two neutral fixed points. Let $c_1,c_2$ be its critical points. Consider the associated postcritical sets $P_1, P_2$. Then 
\begin{itemize}
    \item $f\mid P_1$ and $f\mid P_2$ have degree $1$;
    \item $P_1, P_2$ are separated by some annulus whose modulus depends only on the distance of $\rho_1, \overline{\rho_2} \in S^1$.
\end{itemize}
\end{maintheorem}
\begin{proof}
Consider an eventually-golden-mean map $f$ on the boundary $\partial\mathcal{H}_{z^2}$ with two neutral fixed points. Then the combinatorial distance $\dist(\rho_1,\overline{\rho_2})$ is the difference of the rotational speeds at the two neutral fixed points. Therefore, $N_{Siegel}([f])=O\left(\frac{1}{\dist(\rho_1,\overline{\rho_2})}\right)$.  By taking the limit and applying Theorem \ref{thm:cd}, the corollary follows.
\end{proof}

If $\rho_1, \rho_2$ are eventually-golden-mean, then~\cite{YZ01} provides a convenient description of the dynamical plane of $f\in \partial_\text{reg} \mathcal H_{z^2}$ as the mating of the associated Siegel polynomials.

We believe that, using the sectorial bounds~\cite{DL:sector bounds}, one can introduce a puzzle-type partition for $f$ and then refine Theorem~\ref{thm:hypz2} into a mating-type result between the associated quadratic polynomials with appropriate accommodation to maps with non-locally connected Julia sets, i.e., to Cremer and hairy-Siegel polynomials. (Combinatorially, non-locally connected parts of polynomials do not interact during the mating.) This is beyond the goals of the current paper, where we prefer to keep Theorem~\ref{thm:hypz2} and its proof as simple as possible for illustration purposes.

\subsection{Outline of the proof of Theorem \ref{thm:cd}} \label{ss:prf:Thm:cd}
The proof of Theorem \ref{thm:cd} breaks up into 2 steps.
In the first step, we construct $K$-quasiconformal disks, and show there are no arc degenerations.
In the second step, we show there are no loop degenerations.
The proof for both steps are summarized as follows.

\subsection*{Step 1: no arc degeneration.} For an attracting Fatou component $D_{i,j}$, we define its \emph{valuable domain} $\widehat{D}_{i,j} \subseteq D_{i,j}$ to be the subdisk of $D_{i,j}$ bounded by the equipotential through the unique critical value of the first return map, see~\S\ref{sss:ValDomain}. We fix the multipliers of all attracting cycles; then, the modulus of the annulus $D_{i,j} - \widehat{D}_{i,j}$ is bounded from below by Lemma~\ref{lem:kad}.

Denote the pseudo-core surface (see \S \ref{sec:psd}) of $[f]$ by 
$$
X_f := \widehat \C - \bigcup \Int(\widehat D_{i,j}) - \bigcup Z_{i,j} \text{ and } K_f\coloneqq \mathcal{W}_{arc}(X_f).
$$ We will argue by contradiction and suppose that $K_f$ can be arbitrarily large.

For a Siegel disk $Z_{i,j}$, the dynamics of its first return map $f_{i,j}\colon \partial Z_{i,j}\selfmap$ is conjugate to some rigid irrational rotation on the circle. The conjugacy gives a combinatorial coordinate on $f_{i,j}\colon \partial Z_{i,j}\selfmap$.
The renormalization of the irrational rotation gives a level structure on $\partial Z_{i,j}$: a \emph{level $m$ combinatorial interval} is of the form $J=\big[x, f^{q_{m+1}}_{i,j}(x)\big]\subset \partial Z_{i,j}$, where $f^{q_{m+1}}_{i,j}(x)$ is the closest (level $m$) return of $x$, see~\S\ref{subsec:ci}. 
We denote the combinatorial length of a level $m$ combinatorial interval by $\length_m := |J|$.
Note that $\length_m$ satisfies 
\[  \frac{0.5}{q_{m+1}}<\length_m < \frac{1}{q_{m+1}}.\]

\subsubsection{Non-uniform construction of pseudo-Siegel disks.} \label{sss:into:wX}
In Theorem \ref{thm:cps}, using renormalization theory~\cite{DL22,DLL25} for $\psi^\bullet$-ql maps, we will construct a collection of pseudo-Siegel disks $\widehat Z_{i,j}\supset Z_{i,j}$ whose degenerations are bounded in terms of $K_f$.
Roughly, we will show that each Siegel disk $Z_{i,j}$ is contained in a pseudo-Siegel disk $\widehat Z_{i,j}\supset Z_{i,j}$ such that 
\begin{enumerate}
    \item $\widehat Z_{i,j}$ is a $M = M(K_f)$-quasiconformal disk;
    \item \label{case3:intro} for every interval $J\subseteq \partial \widehat Z_{i,j}$ (``grounded'' rel $\widehat Z_{i,j}$) with $\length_{m+1}< |J| \leq \length_m$, we have
    \begin{enumerate}
        \item\label{est:2a:intro} $\mathcal{W}^{+, np}(J) = O(K\length_m + 1)$; and 
        \item $\mathcal{W}^{+, per}_{\lambda}(J) = O(\sqrt{K\length_{m}} + 1)$.
    \end{enumerate}
\end{enumerate}
Here $\mathcal{W}^{+, np}(J)$ is the extremal width of the family of non-peripheral arcs starting at $J$, and $\mathcal{W}_\lambda^{+, per}(J)$ is the extremal width of the family of peripheral arcs connecting the interval $J \subseteq \partial \widehat Z_{i,j}$ to $\partial \widehat X_f - {\lambda J}$ in $\widehat X_f := \widehat\C - \bigcup \Int(\widehat{D}_{i,j}) - \bigcup \Int(\widehat{Z}_{i,j})$.

Important properties of pseudo-Siegel disks are discussed in~\S\ref{subsec:psd} and \S\ref{subsec:stability} (see, in particular,~\S\ref{sss:compart:wZ}).

\subsubsection{Pulled-off Argument and Localization.}\label{sss:Localization} 
Let $N_f := N_{Siegel}([f])$ be the pulled-off constant.
We show that any wide families of non-peripheral arcs in $\widehat X_f$ must intersect some strictly periodic psuedo-Siegel disks of pre-period $\leq N_f$ (see Lemma \ref{lem:pull}).
It allows us to localize the degeneration (see Theorem \ref{thm:ll}). More precisely, we show that for every $\epsilon> 0$, if the arc degeneration satisfies $K_f\gg_{\epsilon, N_f, \chi(X_f)} \ 1$, where $\chi(X_f)$ is the Euler characteristic (i.e., complexity), then there exists some interval $I$ on some periodic pseudo-Siegel disk $\widehat Z'$ so that
\begin{equation}
\label{eq:local:K_f}
\Width^{+, np}(I) + \Width^{+,per}_\lambda (I) \ge K_f/A\ \  \text{ and }\  \  |I|<\epsilon,    
\end{equation}
for some constant $A\equiv A\big(N_f, \chi(X_f)\big) > 1$ independent of $\epsilon$.
We may assume $\length_{m+1} < |I| \leq \length_m$.

\subsubsection{Calibration Lemma.} \label{sss:Calibration} 
Finally, we apply Theorem \ref{thm:calls} and show that we can find an interval $J \subseteq \partial \widehat Z'$  (grounded rel $\widehat Z'$) such that
\begin{equation}
\label{eq:CL:intro}
\Width^{+, np} (J) \geq K_f/C \ \ \text{and} \ \ |J| \leq \length_{m+1} \leq |I| <\epsilon   
\end{equation}
for some constant $C\equiv C\big(N_f,\chi(X_f)\big) > A > 1$ independent of $\epsilon$. 

We refer to~\eqref{eq:CL:intro} as \emph{Combinatorial Localization} of the degeneration; it is essentially a combination of the localization (Theorem~\ref{thm:ll}) and calibration (Theorem \ref{thm:calls}).

By choosing $\epsilon$ sufficiently small, we obtain from Property~\eqref{case3:intro} and Estimate~\eqref{eq:CL:intro} that   
$$
K_f/C \leq \Width^{+, np}(J)=O(\length_{m+1} K_f+1) = O(\epsilon K_f +1),
$$
which is a contradiction.

\begin{rmk}\label{rem:to:CL}
    We can summarize the argument in Step $1$ as follows. Theorem \ref{thm:cps} stated in \S~\ref{sss:into:wX} says that the arc degeneration $\mathcal{W}_{arc}(X_f)$ of $X_f$ near Siegel disks $Z_{i,j}$ are uniformly distributed along $\partial Z_{i,j}$. On the other hand, Theorem \ref{thm:ll} stated in \S~\ref{sss:Localization}  says that a substantial part of $\mathcal{W}_{arc}(X_f)$ can be localized on a small interval of some $\partial Z_{i,j}$.

    The incompatibility of these two facts almost leads to a contradiction. We note, however, that the estimate in~\eqref{est:2a:intro} is not sufficient to rule out degeneration on the ``special transition scale'' (compare with Remark~\ref{rem:thm:info}). 
    
    A potential degeneration on the special transition scale is handled in Theorem \ref{thm:calls} as explained in \S~\ref{sss:Calibration}. Combinatorially, such degeneration obeys certain invariance constraints of $f\mid X_f$ (see Figure~\ref{fig:ShL}). This leads to Combinatorial Localization~\eqref{eq:CL:intro}. Therefore, we obtain a contradiction by producing a bigger than $K_f$ degeneration.
\end{rmk}

\subsection*{Step 2: no loop degeneration.} We follow the lines of~\cite{Luo21b} (compare with Item~\ref{item:b:intro} in~\S\ref{ss:HistBackgr}) to show that big loop degeneration implies the existence of a Thurston obstruction for realized maps.

\subsubsection{Limiting map on a finite tree}
We will argue by contradiction. Suppose there exists a sequence of eventually golden-mean maps $f_n \in \partial \mathcal{H}$ with $\mathcal{W}_{loop}(\widehat X_f) \to \infty$.
We prove that, after passing to a subsequence if necessary, $f_n$ converges to a non-trivial map on a finite tree of Riemann spheres (see Theorem \ref{thm:gc2}).

\subsubsection{Duality to multi-curves}
We show that this limiting finite tree is ``dual'' to some multi-curves in the complement of periodic Fatou components of $f_n$ for all sufficiently large $n$ (see Proposition \ref{prop:dss}).
This step crucially uses the fact that the arc degeneration is uniformly bounded.

\subsubsection{Limiting Thurston obstruction}
The dynamics on the tree is recorded by a Markov matrix $M$ and a degree matrix $D$.
We show that there exists a non-negative vector $\vec{v} \neq \vec{0}$ with $M\vec v = D \vec v$.
Since the limiting tree is dual to multi-curves, for all sufficiently large $n$, we show that $D^{-1}M$ is no bigger than the Thurston matrix for the corresponding multi-curves of $f_n$.
So the spectral radius of the Thurston's matrix is greater or equal to $1$ (see Proposition \ref{prop:srg1}).
This is a contradiction, and Theorem \ref{thm:cd} follows.

\subsection*{Structure of the paper}
In \S \ref{sec:pan}, we give preparations and introduce some notations.
Four main ingredients in proving uniformly bounded arc degeneration are introduced in \S \ref{sec:cps}, \S \ref{sec:pc}, \S \ref{sec:ltd} and \S \ref{sec:cld}, and these ingredients are assembled in \S \ref{sec:utd}.
The uniformly bounded loop degeneration and Theorem \ref{thm:A} is proved in \S \ref{sec:uthind}.
Theorem \ref{thm:cd} is proved combining Theorem \ref{thm:utd} and Theorem \ref{thm:uthind}.
Finally, Theorem \ref{thm:B} is proved in \S \ref{sec:pc}.

\subsection*{Notations}
In this paper, we will usually fix a hyperbolic component. By a {\em universal constant}, we mean a constant that depends, potentially, only on the hyperbolic component.

We use $A = O(1)$ to mean there exists a universal constant $K$ so that 
$A \leq K$.
More generally, $A = O_{x}(1)$ means that there exists a constant $K_x$ depending on $x$ so that $A \leq K_x$.
Similarly, we use $A \succeq B$ and $A \succeq_x B$ to mean $B/A = O(1)$ and $B/A = O_x(1)$ respectively.

\subsection{Acknowledgement} The first author is partially supported by the NSF grant DMS $2055532$. The second author is partially supported by NSF Grant DMS $2349929$.

We thank Jeremy Kahn, Jan Kiwi, Curt McMullen, Mikhail Lyubich  for many insightful discussions over the years.

The results of the paper were first announced in Spring 2022 during the MSRI semester program ``Complex Dynamics: from special families to natural generalizations in one and several variables''.

\section{Background on hyperbolic components}\label{sec:pan}
In this section, we summarize some background facts on hyperbolic components, and introduce the notion of {\em eventually-golden-mean maps} on the boundary of a hyperbolic component in \S \ref{subsec:egmssm}.

\subsection{Marked hyperbolic component}\label{subsec:mhc}
Following the terminology in \cite{Milnor12}, a {\em fixed point marked rational map} $(f; z_0, z_1,...., z_d)$ is a rational map $f: \widehat\C \longrightarrow \widehat\C$ of degree $d\geq 2$, together with an ordered list of its $d+1$ (not necessarily distinct) fixed points $z_j$.
Let $\Rat_{d,\fm}$ be the space of all fixed point marked rational maps of degree $d$.
The group of M\"obius transformation $\PSL_2(\C)$ acts naturally on $\Rat_{d,\fm}$ and we define the {\em marked moduli space} 
$$
\M_{d,\fm} = \Rat_{d, \fm}/ \PSL_2(\C).
$$
The space of marked hyperbolic maps are open in $\M_{d,\fm}$, and a component is called a (marked) hyperbolic component.
To avoid complicated notations, we shall use $[f]$ to denote an element in $\M_{d,\fm}$, and simply refer to it as a (marked) map.
We remark that as maps vary in a hyperbolic component $\mathcal{H}$, the topology of the Julia sets remains constant.

\begin{defn}
\label{dfn:HH_fm}
Let $\mathcal{H} \subseteq \M_{d,\fm}$ be a hyperbolic component.
\begin{itemize}
    \item It is of {\em disjoint type} if for any map $[f] \in \mathcal{H}$, each grand orbit of a Fatou component of $[f]$ contains a unique critical orbit. 
    \item It is a {\em Sierpinski carpet hyperbolic component} if the Julia set of any map $[f] \in \mathcal{H}$ is homeomorphic to a Sierpinski carpet.
\end{itemize}
\end{defn}

We remark that $\mathcal{H}$ is a finite branched covering of $\mathcal{H}$.
We choose to work with $\mathcal{H}$ as the markings allows us to have a nice parameterization as follows.

Let $\mathcal{H}$ be a hyperbolic component of disjoint type.
There are exactly $2d-2$ attracting periodic cycles for a map $[f] \in \mathcal{H}$.
Let $\mathcal{C}_1, ..., \mathcal{C}_{2d-2}$ be a list of attracting periodic cycles and let $\rho_1,..., \rho_{2d-2}$ be the corresponding multipliers.
The marking of the fixed points allows us to consistently label these attracting periodic cycles throughout $\mathcal{H}$ (see \cite[Theorem 9.3]{Milnor12}), and $\mathcal{H}$ is parameterized by the {\em multiplier profile}, i.e. the multipliers of these $2d-2$ attracting periodic cycles
$$
\boldsymbol{\rho}: \mathcal{H} \overset{\simeq}{\longrightarrow} \D^{2d-2} =  \D_1 \times ... \times \D_{2d-2}.
$$

\subsection{Transversality for multipliers}
Recall that the boundary 
$$
\partial \mathcal{H} = \partial_{\operatorname{reg}}\mathcal{H} \sqcup \partial_{\operatorname{exc}}\mathcal{H}\subseteq \mathcal{M}_{d, \fm}
$$ 
splits into the \emph{regular} and \emph{exceptional} parts, where $[f] \in \partial_{\operatorname{reg}}\mathcal{H}$ if $[f]$ has exactly $2d-2$ non-repelling periodic cycles, and $[f] \in \partial_{\operatorname{exc}}\mathcal{H}$ otherwise.

Let $[f] \in \partial_{\operatorname{reg}}\mathcal{H}$.
Let $x$ be a non-repelling periodic point of $f$ with period $p$.
Suppose $[f_n] \in \mathcal{H}$ with $f_n \to f$, and $x_n \to x$ be a sequence of non-repelling periodic points of $f_n$.
We classify the non-repelling periodic point $x$ into three categories:
\begin{itemize}
    \item Type\setword{(1)}{Type:1}: The multiplier of $[f]$ at $x$ is not $1$ and $x_n$ has period $p$;
    \item Type\setword{(2)}{Type:2}: The multiplier of $[f]$ at $x$ is not $1$ and $x_n$ has period $\nu p$, with $\nu \geq 2$;
    \item Type\setword{(3)}{Type:3}: The multiplier of $[f]$ at $x$ is $1$.
\end{itemize}

\begin{prop}\label{prop:ts}
    The multiplier map extends to an embedding on the regular boundary $\boldsymbol{\rho}: \partial_{\operatorname{reg}}\mathcal{H} \xhookrightarrow{} \partial\D^{2d-2}$ and $\boldsymbol{\rho}(\partial_{\operatorname{reg}}\mathcal{H}) \cap \boldsymbol{\rho}(\partial_{\operatorname{exc}}\mathcal{H}) = \emptyset$.

    \noindent
    Here $\boldsymbol{\rho}(\partial_{\operatorname{reg}}\mathcal{H})$ or $\boldsymbol{\rho}(\partial_{\operatorname{exc}}\mathcal{H})$ are understood as the accumulation set of $\boldsymbol{\rho}([f_n])$ as $[f_n] \to \partial_{\operatorname{reg}}\mathcal{H}$ or $[f_n] \to \partial_{\operatorname{exc}}\mathcal{H}$ respectively.
\end{prop}
\begin{proof}
    Let $[f] \in \partial_{\operatorname{reg}}\mathcal{H}$.
    Let us first suppose that every non-repelling periodic point of $f$ is Type~\ref{Type:1}.
    Let $\mathcal{C}_i$ be a list of non-repelling periodic cycles of $[f]$.
    Then by implicit function theorem, the cycles $\mathcal{C}_i$ move holomorphically on a neighborhood of $f$.
    Thus, we can define a holomorphic map 
    $(\rho_1,..., \rho_{2d-2}): U \longrightarrow \C^{2d-2}$ on a neighborhood $U$ of $[f]$,
    where $\rho_i$ is the multiplier of the cycle $\mathcal{C}_i$.
    By transversality of the multipliers (see \cite[Theorem 6]{Lev10} or \cite{Eps09}), $(\rho_1,..., \rho_{2d-2})$ gives a local parameterization of the moduli space $\mathcal{M}_{d, \fm}$.
    Therefore, $\boldsymbol{\rho}$ extends to an embedding near $[f]$, and $\boldsymbol{\rho}^{-1}(\boldsymbol{\rho}([f])) = \{[f]\}$.
    
    If there are Type~\ref{Type:2} or Type~\ref{Type:3} non-repelling periodic points, the argument is similar, but we need to pass to a branched cover (see \cite[\S 9]{Milnor12}).
    We consider the space of {\em $n$-periodic marked rational maps} consisting of 
    $$
    (f; x_0,..., x_{d^n}) \in \Rat_d\times \widehat\C^{d^n-1},
    $$ 
    where $f$ is a rational map of degree $d$ together with an ordered list of its $d^n+1$ (not necessarily distinct) periodic points dividing $n$.
    Since the iteration map $\Rat_d \longrightarrow \Rat_{d^n}$ is a local immersion (see \cite[Proposition 4.1]{Ye15}), by pulling back the local charts for $\Rat_{d^n, \fm}$ (which is a smooth manifold by \cite[Lemma~9.2]{Milnor12}), we have local holomorphic charts near any $n$-periodic marked rational map. Since $f$ has at least $3$ distinct fixed points, we obtain local chart on the corresponding quotient marked moduli space (see \cite[\S 9]{Milnor12}).
    Note that the forgetful map from $n$-periodic marked rational maps to fixed point marked rational maps is a branched covering.

    If $x$ is a Type~\ref{Type:2} point, then two or more periodic points in the same periodic cycle of $f_n$ collide in the limit, as we assume $[f] \in \partial_{\operatorname{reg}}\mathcal{H}$.
    Denote the period $\nu p$ and $p$ cycle of $[f]$ by $\mathcal{C}$ and $\widetilde{\mathcal{C}}$ respectively.
    By marking these periodic points, we may assume the $\widetilde{\mathcal{C}}$ and $\mathcal{C}$ move holomorphically on this branched cover, and their multipliers $\rho$ and $\widetilde{\rho}$ are holomorphic functions.

    Similarly, if $x$ is a Type~\ref{Type:3} point, then there is a period $p$ repelling point $\tilde{x}_n$ of $f_n$ with $\tilde{x}_n \to x$.
    Denote these two cycles of $[f]$ by $\mathcal{C}$ and $\widetilde{\mathcal{C}}$.
    By marking these periodic points, we may assume that $\mathcal{C}$ and $\widetilde{\mathcal{C}}$ move holomorphically, and their multipliers $\rho$ and $\widetilde{\rho}$ are holomorphic functions.

    In this way, there exists a neighborhood $U$ of $[f]$ and a branched cover $\widetilde{U}$ of $U$ so that the multipliers map $(\rho_1,..., \rho_k, \rho_{k+1}, \widetilde{\rho}_{k+1},..., \rho_{2d-2}, \widetilde{\rho}_{2d-2})$ is a holomorphic map on $\widetilde{U}$, where $k$ is the number of Type~\ref{Type:1} cycles.
    By transversality of the multipliers (see \cite[Theorem 6]{Lev10} or \cite{Eps09}) and restrict the domain if necessary, the map is a finite branched covering onto its image, and the image of $\widetilde{U}$ under the restricted multiplier map $(\rho_1,..., \rho_k, \rho_{k+1}, \rho_{k+2}, ..., \rho_{2d-2})$ is open.

    Since $f$ has at least $3$ distinct fixed points, $\Aut(f) = \{id\}$, where $\Aut(f)$ is the automorphism group of the fixed point marked rational map $f$.
    Therefore, the fiber of the branched cover $\widetilde{U} \longrightarrow U$ consists only of the same map with (potentially) a different marking on the periodic points.
    By fixing a marking of the attracting periodic points in $\mathcal{H}$, we obtain a homeomorphic lift $\widetilde{V} \subseteq \widetilde{U}$ of $V = \mathcal{H}\cap U$.
    The branched cover $\widetilde{U} \longrightarrow U$ is injective on $\overline{\widetilde{V}}$, and hence a homeomorphism between $\overline{\widetilde{V}}$ and $\overline{V}$.
    By lifting the map $\boldsymbol{\rho}$ from $V$ to $\widetilde{V}$, we see that $\boldsymbol{\rho}$ extends to an embedding of $\partial{\mathcal{H}}$ near $[f]$, and $\boldsymbol{\rho}^{-1}(\boldsymbol{\rho}([f])) = \{[f]\}$.
    The proposition now follows.
    \end{proof}

\subsection{Eventually-golden-mean maps}\label{subsec:egmssm}
In this subsection, we will fix a hyperbolic component $\mathcal{H}$ of disjoint type.

Let $\theta \in (0,1)$ be an irrational number, with continued fraction expansion
$$
\theta =  [0; a_1,..., a_m,...] = \cfrac{1}{a_1+\cfrac{1}{a_2+\cfrac{1}{ a_3+\cdots}}}.
$$
We say $\theta$ is of {\em bounded type} if
$$
\sup \{a_m\} < \infty.
$$
More generally, we say $\theta$ is {\em Brjuno} if
$$
\sum\frac{\log \mathfrak{q}_{m+1}}{\mathfrak{q}_n} < \infty.
$$
Note that if $\theta$ is of bounded type, then $\theta$ is Brjuno.

These arithmetic properties of irrational numbers are relevant to holomorphic dynamics.
It is well-known that if $f$ is a holomorphic map defined on $0 \in U$, with $f(0) = 0$ and $f'(0) = e^{2\pi i \theta}$ with $\theta$ being Brjuno, then $f$ is conjugate to the rigid rotation $z \mapsto e^{2\pi i \theta}z$ in a neighborhood of $0$.
If $f$ is a globally defined, then this neighborhood is part of a Siegel disk for $f$.

If $f$ is a rational map with a fixed point of multiplier $e^{2\pi i \theta}$ with $\theta$ of bounded type, then the corresponding Siegel disk has quasi-circle boundary which passes through at least one critical point \cite{Zha11}.

Let $\theta = [0; a_1,..., a_m, ...]$.
We say it is {\em eventually-golden-mean} if there exists $m_\theta$ so that $a_n = 1$ for all $n \geq m_\theta$.
Note that in this case, $\theta$ is automatically of bounded type.

Let $\mathcal{H}$ be a hyperbolic component of disjoint type, and $[f] \in \partial{\mathcal{H}}$.
Then some attracting periodic cycles must become indifferent.
By following the deformations for the corresponding periodic cycles, its multiplier profile $(\rho_1,..., \rho_{2d-2}) = (\rho_1([f]),..., \rho_{2d-2}([f]))$ lies on the boundary 
$$
(\rho_1,..., \rho_{2d-2}) \in \partial\D^{2d-2}.
$$ 

\begin{defn}
We say a boundary parameter  $(\rho_1,..., \rho_{2d-2}) \in \partial\D^{2d-2}$ is
\begin{itemize}
\item {\em rational} if each $\rho_j$ is either in $\D$ or $\rho_j \in S^1$ and is rational;
\item {\em irrational} if each $\rho_j$ is either in $\D$ or $\rho_j \in S^1$ and is irrational;
\item {\em eventually-golden-mean} if each $\rho_j$ is either in $\D$ or $\rho_j \in S^1$ and is eventually-golden-mean.
\end{itemize}
We also say  $(\rho_1,..., \rho_{2d-2})$  is {\em realizable} if there exists $[f] \in \partial{\mathcal{H}} \subseteq \M_{d,\fm}$ with multiplier profile $(\rho_1,..., \rho_{2d-2})$.
\end{defn}

Let $\mathfrak{S} \subseteq \partial \D^{2d-2}$ be the set of realizable eventually-golden-mean boundary parameter.
In this paper, we will focus on the following maps:
\begin{defn}
\label{defn:EGM maps}
A map $[f] \in \partial \mathcal{H} \subseteq \mathcal{M}_{d,\fm}$ is called an {\em eventually-golden-mean map} if its multiplier profile $(\rho_1,..., \rho_{2d-2}) \in \mathfrak{S}$. We denote by $\partial_\egm \mathcal H\simeq \mathfrak{S}$ the set of all such maps in $\partial \mathcal H$.
\end{defn}

We remark that since eventually-golden-mean irrational numbers are of bounded type, any non-repelling cycles of a eventually-golden-mean map $[f]$ are contained either in (super-)attracting Fatou components or Siegel disks.
Moreover, since its multiplier profile is on the boundary $\partial \D^{2d-2}$, there is at least one cycle of Siegel disk for $[f]$.

\begin{defn}\label{defn:strong}
Let $(\rho_1,..., \rho_{2d-2}) \in  \partial \D^{2d-2}$.
A sequence $(\rho_{1,n},..., \rho_{2d-2,n})\in  \partial \D^{2d-2}$ is said to converge to $(\rho_1,..., \rho_{2d-2})$ {\em strongly}, denoted by 
$$
(\rho_{1,n},..., \rho_{2d-2,n}) \to_s (\rho_1,..., \rho_{2d-2}),
$$
if $\rho_{j,n} \to \rho_j$ for all $j$, and $\rho_{j,n} = \rho_j$ when $\vert \rho_{j} \vert < 1$.
\end{defn}

If we further assume that the Julia set is a Sierpinski carpet, then we have the following density result for eventually-golden-mean maps.
\begin{prop}\label{prop:rl}
Let $\mathcal{H}$ be a Sierpinski carpet hyperbolic component of disjoint type.
The set $\mathfrak{S}$ is dense in $\partial \D^{2d-2}$.

Moreover, for any $(\rho_1,..., \rho_{2d-2}) \in \partial \D^{2d-2}$, there exists a sequence 
$$
(\rho_{1,n},..., \rho_{2d-2,n}) \in \mathfrak{S}
$$ 
converging to $(\rho_1,..., \rho_{2d-2})$ strongly.
\end{prop}
\begin{proof}
It follows from the pinching deformation in \cite{CT18} (see also \cite{Luo21b}) that all rational boundary points are realizable.
Let $[f] \in \partial \mathcal{H}$ with rational multiplier profile $(\rho_1,..., \rho_{2d-2})$.
Since $\mathcal{H}$ is Sierpinski, no non-repelling periodic points collide.
We may assume $\rho_i \neq 1$ for all $i$, as other wise, we can pass to a branched cover as in Proposition \ref{prop:ts}.
Hence, we can locally parameterized the periodic cycles analytically.
Thus, there exists a neighborhood $U \subseteq \M_{d,\fm}$ of $[f]$ so that the multipliers 
$$
\boldsymbol{\rho}(t):=(\rho_1(t),..., \rho_{2d-2}(t))
$$ 
is an analytic function on $t\in U$.
By transversality of the multipliers (see \cite[Theorem 6]{Lev10} or \cite{Eps09}), $\boldsymbol{\rho}^{-1}((\rho_1,..., \rho_{2d-2})) = \{[f]\}$.
Thus by shrinking $U$ if necessary, the image $\boldsymbol{\rho}(U) \subseteq \C^{2d-2}$ is open (see \cite[p. 107]{GR84}).
Since eventually-golden-mean irrational numbers are dense, we can find an eventually-golden-mean map $[f] \in U$.
Since the rational parameters are dense, $\mathfrak{S}$ is dense.
The moreover part can be proved in the same way.
\end{proof}
\subsection{Valuable-attracting domains}\label{subsec:vad}
\label{sss:ValDomain}
Let $D$ be an attracting Fatou component for $f$ of period $p$.
Assume that the multiplier of the attracting periodic point is $\rho$.
Then the first return map $f^{p}: D \longrightarrow D$ is conjugate to the Blaschke product
\begin{align*}
F: \D &\longrightarrow \D\\
z &\mapsto z \frac{z+\rho}{1+\bar\rho z}.
\end{align*}
Let $\Psi: D \longrightarrow \D$ be the conjugacy map, and let $r:= \max\{\frac{1}{2}, \vert \rho\vert\}$.
We call the closed Jordan disk
$$
\widehat D = \Psi^{-1}(\overline{B(0, r)}) \subseteq D
$$
the {\em valuable-attracting domain} for $D$.
One can easily verify by our construction that
\begin{lem}\label{lem:kad}
Let $\widehat D$ be the valuable-attracting domain for $D$. Then
\begin{itemize}
\item $\widehat D$ is forward invariant under $f^{p}$;
\item $\widehat D$ contains the unique critical point of $f^{p}$ in $D$;
\item The annulus $D - \widehat D$ has modulus $-\frac{1}{2\pi} \log (\max\{\frac{1}{2}, \vert \rho\vert\})$.
\end{itemize}
\end{lem}

\section{Core and pseudo-Core Surfaces of maps in $\partial_\egm \mathcal H $}
\label{sec:psd} In this section, we will introduce pseudo-Siegel disks and pseudo-core surfaces. The main construction is in Theorem \ref{thm:cps}; see also~\S\ref{sss:into:wX}.

Let us fix a hyperbolic component $\mathcal{H}$ of disjoint type. Recall from Definition~\ref{defn:EGM maps} that $\partial_\egm \mathcal H$ denotes the set of eventually-golden-mean maps in $\partial \mathcal H$: every neutral periodic cycle of a map in $\partial_\egm \mathcal H$ is Siegel of the eventually-golden-mean type. 

We discuss some combinatorial facts of irrational rotations on a circle in \S \ref{subsec:ci}. In~\S\ref{subsec:psd}, we review the notion of almost-invariant pseudo-Siegel disks. They are obtained from regular forward-invariant Siegel disks by filling-in parabolic fjords as illustrated on Figure~\ref{fig:CD}; see Definition~\ref{defn:pseudo-Siegel disks}.

The \emph{core surface} $X_f$ of $f\in \partial_\egm \mathcal H$ is the complement to the union of all periodic valuable-attracting domains and Siegel disks; see~\eqref{eq:defn:X_f}. The \emph{pseudo-core surface} $\widehat X_f\subset X_f$ is obtained by removing pseudo-Siegel disks instead of Siegel disks; see~\eqref{eq:defn:hat X_f}.  Properties of $X_f$ and $\widehat X_f$ are discussed in~\S\ref{ss:CoreSurfaces}.
We remark that some terminologies for degeneration of Riemann surfaces are summarised in \S \ref{subsec:dr}.

\subsection{Combinatorial intervals for Siegel disks}\label{subsec:ci}
In this subsection, we introduce the terminologies for dynamics on Siegel disks.
We remark that most of the discussions in this section work for any rational map with a Siegel disk with a single critical point on its boundary.

Let $\mathcal{H}$ be a hyperbolic component of disjoint type.
Let $[f] \in \partial \mathcal{H}$ be an eventually-golden-mean map.
Let $Z$ be a Siegel disk for $f$ of period $p$ with rotation number $\theta$.
Let $h: Z \longrightarrow \D$ be a Riemann mapping with $h(\alpha) = 0$, where $\alpha$ is the fixed point in $Z$.
Since $Z$ is a quasi-disk, $h$ extends continuously to
$$
h: \overline{Z} \longrightarrow \overline{\D}, \, h(\alpha) = 0
$$
which conjugate $f^p|_{\partial Z}$ with the rigid rotation on $S^1$.

We define the {\em combinatorial length} of an interval $I \subseteq \partial Z$ as
$$
\vert I \vert := \vert h(I)\vert_{\R/\Z} \in (0,1).
$$
Similarly, we define the {\em combinatorial distance} between $x, y \in \partial Z$ as
$$
\dist(x,y):= \dist_{\R/\Z}(h(x), h(y)) \in [0,1/2].
$$
Let $x \in \partial Z$ and $t\in \R/\Z$, we set
$$
x \boxplus t = h^{-1}(h(x) + t),
$$
i.e., $x \boxplus t$ is $x$ rotated by angle $t$.
Note that $f^p(x) = x \boxplus \theta$ for all $x\in \partial Z$.

Let $[0; a_1,..., a_m,...]$ be the continued fraction expansion for $\theta$.
Let $$
\mathfrak{p}_m/\mathfrak{q}_m :=
\begin{cases}
[0; a_1, ..., a_m], \text{ if }a_1 >1\\
[0; a_1, ..., a_{m+1}], \text{ if }a_1 = 1
\end{cases}
$$ 
be the sequence of approximations for $\theta$ given by the continued fraction.
We use the convention and set $\mathfrak{q}_0 = 1$.
Then $f^{\mathfrak{q}_0p} = f^p, f^{\mathfrak{q}_1p}, ...$ is the sequence of first returns of $f^p|_{\partial Z}$, i.e.,
$$
\dist(f^{ip}(x), x) > \dist(f^{\mathfrak{q}_mp}(x), x) =: \mathfrak{l}_m, \, x\in \partial Z \, \text{ for all } i < \mathfrak{q}_m.
$$
We define $\theta_m\in (-1/2, 1/2)$ so that
$$
f^{\mathfrak{q}_mp}(x) = x \boxplus \theta_m.
$$
Note that $\mathfrak{l}_m = \vert \theta_m\vert$.

Given two points $x,y \in \partial Z$ with $\dist(x,y) < 1/2$, we let $[x,y]$ be the shortest closed interval of $\partial Z$ between $x, y$.
Let $I \subseteq \partial Z$ be an interval.
We define the $\lambda$-scaling of $I$ as
$$
\lambda I := \{x \in \partial Z: \dist(x, I) \leq (\lambda-1) \vert I \vert /2\}.
$$

An interval $I \subseteq \partial Z$ is called a {\em combinatorial interval of level $m$}, or simply a {\em level $m$ interval} if $\vert I \vert = \mathfrak{l}_m$.
Note that a level $m$ interval is of the form
$$
I = [x, f^{\mathfrak{q}_mp}(x)].
$$

\begin{figure}[ht]
  \centering
  \resizebox{0.8\linewidth}{!}{
    \def\svgwidth{\columnwidth}
\begingroup%
  \makeatletter%
  \providecommand\color[2][]{%
    \errmessage{(Inkscape) Color is used for the text in Inkscape, but the package 'color.sty' is not loaded}%
    \renewcommand\color[2][]{}%
  }%
  \providecommand\transparent[1]{%
    \errmessage{(Inkscape) Transparency is used (non-zero) for the text in Inkscape, but the package 'transparent.sty' is not loaded}%
    \renewcommand\transparent[1]{}%
  }%
  \providecommand\rotatebox[2]{#2}%
  \newcommand*\fsize{\dimexpr\f@size pt\relax}%
  \newcommand*\lineheight[1]{\fontsize{\fsize}{#1\fsize}\selectfont}%
  \ifx\svgwidth\undefined%
    \setlength{\unitlength}{841.88976378bp}%
    \ifx\svgscale\undefined%
      \relax%
    \else%
      \setlength{\unitlength}{\unitlength * \real{\svgscale}}%
    \fi%
  \else%
    \setlength{\unitlength}{\svgwidth}%
  \fi%
  \global\let\svgwidth\undefined%
  \global\let\svgscale\undefined%
  \makeatother%
  \begin{picture}(1,0.33670034)%
    \lineheight{1}%
    \setlength\tabcolsep{0pt}%
    \put(0,0){\includegraphics[width=\unitlength,page=1]{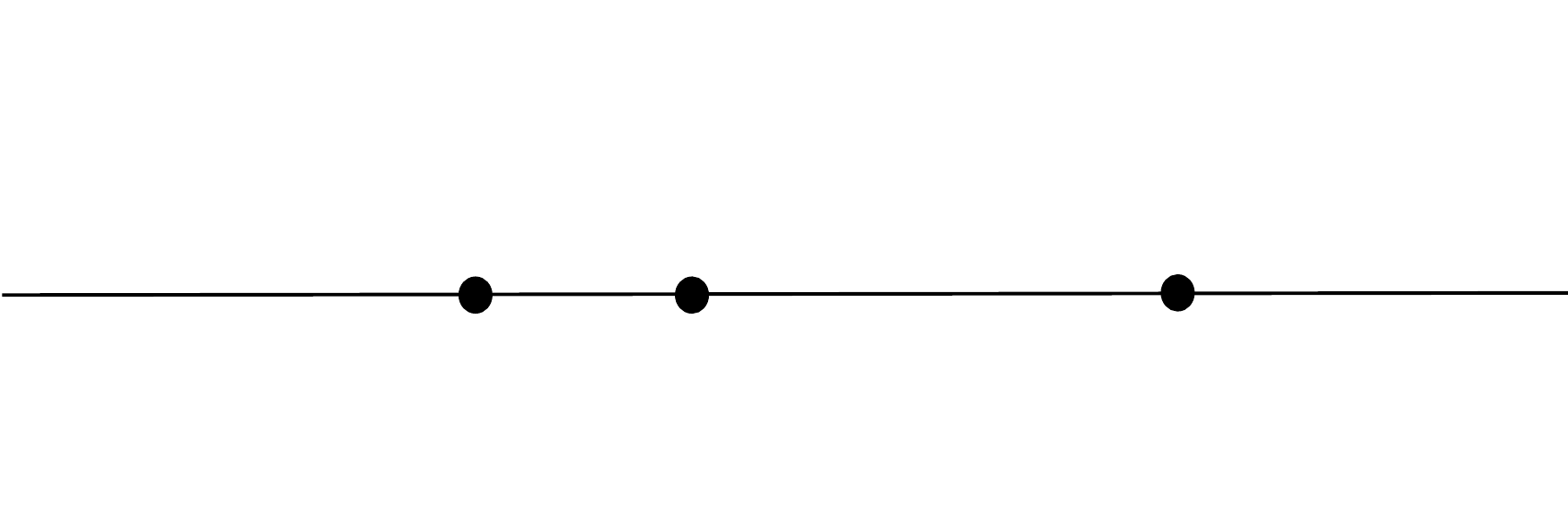}}%
    \put(0.4338289,0.10627646){\color[rgb]{0,0,0}\makebox(0,0)[lt]{\lineheight{1.25}\smash{\begin{tabular}[t]{l}$x$\end{tabular}}}}%
    \put(0.74270388,0.10066054){\color[rgb]{0,0,0}\makebox(0,0)[lt]{\lineheight{1.25}\smash{\begin{tabular}[t]{l}$f^{\mathfrak{q}_{mp}}(x)$\end{tabular}}}}%
    \put(0.26091771,0.1017659){\color[rgb]{0,0,0}\makebox(0,0)[lt]{\lineheight{1.25}\smash{\begin{tabular}[t]{l}$f^{\mathfrak{q}_{m+1}p}(x)$\end{tabular}}}}%
    \put(0.58124648,0.17226338){\color[rgb]{0,0,0}\makebox(0,0)[lt]{\lineheight{1.25}\smash{\begin{tabular}[t]{l}$\mathfrak{l}_m$\end{tabular}}}}%
    \put(0.35087986,0.17459872){\color[rgb]{0,0,0}\makebox(0,0)[lt]{\lineheight{1.25}\smash{\begin{tabular}[t]{l}$\mathfrak{l}_{m+1}$\end{tabular}}}}%
  \end{picture}%
\endgroup%

  }
  \caption{The first return and combinatorial intervals.}
  \label{fig:CI}
\end{figure}
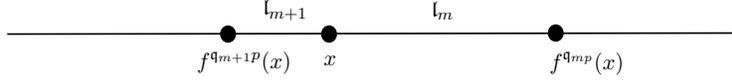

Let $I$ be a level $m$ interval.
We say the intervals
$$
\{f^{ip}(I): i \in \{0,1,..., \mathfrak{q}_{m+1}-1\} 
$$
are obtained by {\em spreading around} $I$.
We enumerate these intervals counterclockwise starting with $I = I_0$
$$
I_0 = I, I_1 = f^{i_1p}(I), ..., I_{\mathfrak{q}_{m+1} - 1} = f^{(i_{\mathfrak{q}_{m+1}-1})p}(I), \, i_j \in \{1,2, ..., \mathfrak{q}_{m+1} -1\}.
$$
Note that the interval $I_i$ is either attached to $I_{i+1}$ or there is a level $m+1$ combinatorial interval between $I_i$ and $I_{i+1}$.

\subsubsection{Diffeo-tiling $\mathfrak{D}_m$}\label{sss:diffeo tiling} There is a unique critical point $c$ of $f^p$ on $\partial Z$.
We denote by $\CP_m = \CP_m(Z)$ the set of critical points of $f^{\mathfrak{q}_{m+1}p}$ on $\partial Z$.
We define the {\em diffeo-tiling} $\mathfrak{D}_m$ of level $m$ as the partition of $\partial Z$ induced by $\CP_m$.
Note that there are $\mathfrak{q}_{m+1}$ intervals in $\mathfrak{D}_m$, and each interval has length either $\mathfrak{l}_m$ or $\mathfrak{l}_m + \mathfrak{l}_{m+1}$.

\subsection{Pseudo-Siegel disks}\label{subsec:psd} A pseudo-Siegel disk $\wZ^m$ is obtained from a Siegel disk $\overline Z$ by filling-in all parabolic fjords of levels $\ge m$. The formal definition of $\wZ^m$ for maps in $\partial _\egm \mathcal H$ (see~\S\ref{sss:defn:pseudo-Siegel disks}) is the same as for quadratic polynomials with the additional requirement that the ``territory'' $\XX(\wZ^m)$ containing all auxiliary objects of $\wZ^m$ is peripheral rel $\overline Z$; see~\S\ref{sss:zW:conventions} and Property~\ref{prop:P:wZ^m} in \S\ref{sss:defn:pseudo-Siegel disks}.

\subsubsection{Parabolic fjords and their protections}\label{sss:zW:conventions}
As in~\S\ref{subsec:ci}, let $Z$ be a periodic Siegel disk of $[f]\in \partial _\egm \mathcal H$ with period $p$ and rotation number $\theta$. 

We say that a disk $D\supset Z$ is {\em peripheral rel $\overline Z$} if $D\setminus \overline Z$ does not intersect the post-critical set of $f$. More generally, we say that a set $S\subset \wC$ is {\em peripheral rel $\overline Z$} if $S$ is within a peripheral disk $D$. In other words, $S$ is peripheral rel $\overline Z$ if $S$ can be ``contracted'' rel the postcritical set into $\overline Z$.

Consider a diffeo-tiling $\mathfrak D_m$ (see~\S\ref{sss:diffeo tiling}) and an interval $I\in \mathcal D_m$. Given a peripheral curve $\beta \subset \wC\setminus Z$ with endpoints in $I$, set $\widehat{\mathfrak F}_\beta$ to be the closure of the connected component of $\wC\setminus (Z\cup \beta)$ enclosed by $\beta \cup I$. If $\widehat{\mathfrak F}_\beta$ is peripheral rel $\overline Z$, then we call $\widehat{\mathfrak F}_\beta$ the \emph{parabolic fjord} bounded by $\beta$; see Figure \ref{fig:CD}. We will refer to $\beta$ as the \emph{dam} of $\widehat{\mathfrak F}_\beta$.

Let $\XX\subset \wC\setminus Z$ be a rectangle with \[\partial^{h} \XX\subset \overset{\circ} {I}\ := \ I\setminus \{\text{ends of $I$}\}.\] We denote by $\upbullet \XX$ the union of $\XX$ and the closure of the connected component of $\wC\setminus (\XX\cup I)$ enclosed by $\partial^v \XX \cup \overset{\circ} {I}$. We say that $\XX$ \emph{protects} a fjord $\widehat{\mathfrak F}_\beta$ if
\begin{itemize}
    \item $\widehat{\mathfrak F}_\beta\subset \upbullet \XX\setminus \XX$;
    \item $\upbullet \XX$ is peripheral rel $\overline Z$.
\end{itemize}

\subsubsection{Pseudo Siegel disks for rational maps in $\partial_\egm \mathcal H$} \label{sss:defn:pseudo-Siegel disks} Let us state a definition of pseudo-Siegel disks with slightly simplified notations. In particular, outer protections $\XX_I$ are considered below rel $\overline Z$ instead of $\widehat Z^{m+1}$ as it was in \cite[Assumption 6]{DL22}. (Protecting rectangle rel $\overline Z$ can be easily projected into a protecting rectangle rel $\wZ^{m+1}$ and vise-versa, see \cite[ Lemma 5.9]{DL22}.)

\begin{defn}
\label{defn:pseudo-Siegel disks} Following notations from~\S\ref{sss:zW:conventions}, a \emph{pseudo-Siegel disk} 
$\widehat Z^m$ of $m\ge -1$ and its \emph{territory} $\XX (\wZ^m)\supset \wZ^m$ are disks inductively constructed as follows (from bigger $m$ to smaller ones):
\begin{enumerate}
\item $\widehat Z^m = \overline{Z}$ and $\XX(\widehat Z^m) = \overline{Z}$ for all sufficiently large $m \gg 0$,
\item either \[\widehat Z^m := \widehat Z^{m+1}\text{ and } \XX(\widehat Z^m) := \XX(\widehat Z^{m+1}),\] 
or for every interval $I\in \mathfrak D_m$ there is 
\begin{itemize}
    \item a parabolic peripheral fjord $\widehat{\mathfrak F}_I\equiv \widehat{\mathfrak F}_{\beta_I}$ bounded by its dam $\beta_I$ with endpoints in $I$; and
    \item a peripheral rectangle $\XX_I$ protecting $\widehat{\mathfrak F}_I$ 
\end{itemize}
 such that    
   \begin{equation}\label{eq:defn:pseudo-Siegel disks}
      \begin{aligned}
     \widehat Z^m &:= \widehat Z^{m+1} \cup \bigcup_{I\in \mathfrak{D}_m} \widehat{\mathfrak F}_{I}\\ 
         \XX \big(\widehat Z^m\big) &:= \XX(\wZ^{m+1})\cup \bigcup_{I\in \mathfrak{D}_m} \upbullet \XX_I 
    \end{aligned}\end{equation}
and such that $\wZ^m$ and $\XX \big(\widehat Z^m\big)$ satisfy $7$ compatibility condition stated in \cite[\S~5.1]{DL22} and briefly summarized in~\S\ref{sss:compart:wZ}. 
\end{enumerate}
\end{defn}

We remark that in addition to $7$ compatibility conditions from \cite[\S~5.1]{DL22}, a pseudo-Siegel disk satisfies the following additional property:
\vspace{0.1cm}

\begin{enumerate}[start=16, label=(\Alph*)]
\item \label{prop:P:wZ^m} $\XX(\wZ^m)$ is peripheral rel $\overline Z$. 
\end{enumerate}
\vspace{0.1cm}

If $\widehat Z^m \neq \widehat Z^{m+1}$, then we say $\widehat Z^m$ is a {\em regularization} of $\widehat Z^{m+1}$ at level $m$.
We denote $\widehat Z = \widehat Z^{-1}$, and call it the pseudo-Siegel disk. We remark that Definition ~\ref{defn:pseudo-Siegel disks} allows us to potentially take $\widehat Z^n = \overline{Z}$ and $\XX(\wZ^n)=\overline{Z}$ for all $n$, which will satisfy all the compatible conditions.
Thus, $\overline{Z}$ is trivially a pseudo-Siegel disk (of any level). Similarly, any level $m$ pseudo-Siegel disk $\widehat Z^m$ can be extended to lower levels by setting $\widehat Z^n = \widehat Z^m$ for all $n \leq m$.
With this in mind, when we introduce definitions or state theorems for pseudo-Siegel disks, they apply to regular Siegel disks as well.

\begin{figure}[ht]
  \centering
  \resizebox{\linewidth}{!}{
    \def\svgwidth{\columnwidth}
    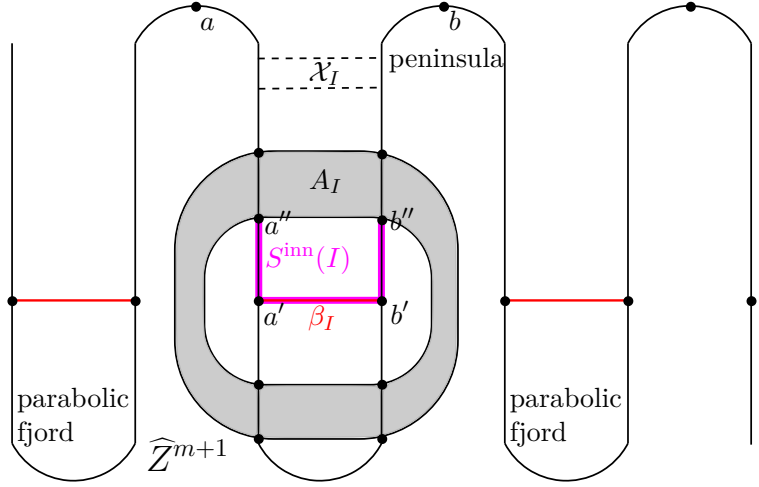

  }
  \caption{An illustration of psuedo-Siegel disk. 
  The intersection patterns of the protecting annulus $A_I$, the inner buffer $S^{\inn}(I)$, extra outer protection $\XX_I$ are indicated on the graph.}
  \label{fig:CD}
\end{figure}

\subsubsection{Regular intervals}
A point $x\in \partial \widehat Z^m$ is {\em regular} if $x\in \partial \widehat Z^m \cap \partial Z$.
By construction, if a point $x\in\partial \widehat Z^m$ is regular, then $x$ is regular on $\partial \widehat Z^k$ for all $k\geq m$.
A {\em regular interval} $I \subseteq \partial\widehat Z^m$ is an interval with regular endpoints.

The projection of a regular interval $I \subseteq \partial \widehat Z^m$ onto $\partial Z$ is the interval $I^\bullet \subseteq \partial Z$ with the same endpoints and the same orientation as $I$.
We define the combinatorial length of $I$ by $\vert I \vert := \vert I^\bullet \vert$. 
Similarly, we can define the projection $I^k$ of a regular interval $I \subseteq \partial \widehat Z^m$ onto $\partial \widehat Z^k$ for $k > m$.\

For an interval $I \subseteq \partial Z$, the projection $I^m$ onto $\partial Z^m$ is the smallest regular interval whose projection onto $\partial Z$ contains $I$.
Similarly, we can define the projection of an interval $I \subseteq \partial \widehat Z^m$ onto $\partial \widehat Z^n$ for $n < m$.

Let $I \subseteq \partial \widehat Z^m$ be a regular interval.
Abusing the notations, we denote $\lambda I\subseteq \partial \widehat Z^m$ as the projection of $\lambda I^\bullet$ on $\partial \widehat Z^m$ where $I^\bullet \subseteq \partial Z$ is the projection of $I$ onto $\partial Z$.

\subsubsection{Compatibility conditions between $\wZ^m$ and $\overline Z$} \label{sss:compart:wZ} The $7$ compatibility conditions stated in \cite[\S~5.1]{DL22} are designed to ensure the following key-properties of $\wZ^m$:
\begin{enumerate}[label=(\Alph*)]
  \item\label{prop:wZ:main:A} $\wZ^m$ is almost invariant under $f^{i}$ for $|i|\le \qq_{m+1}$;
   \item\label{prop:wZ:main:B} the ``slight'' shrinking 
   \[ \wC\setminus \overline Z\ \leadsto \ \wC\setminus \wZ^m \]
  has small affect on the width of rectangles in $\wC\setminus \overline Z$ that have vertices in $\wC\setminus \XX(\wZ^m)$; see Lemma~\ref{lem:C-Z to C-wZ}
\end{enumerate}

Below we will recall the main aspects of the axiomatization of $\wZ^m$ from~\cite[\S~5.1]{DL22}. Various minor technical conditions will be omitted. We remark that $\wZ^m$ can be defined explicitly using explicit hyperbolic geodesics in the complement of $\overline Z$; see Appendix~\ref{ss:good:wZ}.

Property~\ref{prop:wZ:main:B} follows the requirement that all $\Width(\XX_I)$ are sufficiently wide and will be discussed in~\S\ref{sss:Z vs wZ}. 

Let us now discuss~\ref{prop:wZ:main:A}. It follows from~\eqref{eq:defn:pseudo-Siegel disks} that
\begin{enumerate}[label=(\Alph*), start=3]
  \item\label{prop:wZ:main:C} that critical points $\CP_m$ of $f^{\mathfrak{q}_{m+1}p}\cap \partial Z$ are regular points of $\wZ^n$ for any $n\ge m$. 
\end{enumerate}
In particular, the projections $I^m$ of $I\in \mathfrak D_m$ induce a well-defined diffeo-tiling of $\partial \wZ^n$.

As illustrated on Figure~\ref{fig:CD}, for all $I\subset \mathfrak D_m$, we require the existence of annuli $A_I$ around the $\beta_I$ with 
\begin{equation}
\label{eq:dfn:A_I}
\mod(A_I)\ge \delta >0, \sp\sp\text{ where $\delta>0$ is small but fixed} \end{equation}
such that for all $|i|\le q_{m+1}$ the annuli $(A_I)_{I\in \mathfrak D_m}$ control the difference between $f^i(\wZ^m)$ and $\wZ^m$ in the following sense.
\begin{enumerate}[label=(\Alph*), start=4]
\item\label{prop:wZ:main:D}  Assume $f^i\colon \overline Z\to \overline Z$ maps $J\in \mathfrak D_{m}$ into most of the $J\in \mathfrak D_{m}$; i.e., the difference $f^i(J)\setminus I$ is either empty or consists of an interval in $\mathfrak D_{m+1}$. Then we require that $A_I$ also surrounds $f^i(\beta_J)$.
\end{enumerate}
We remark that $A_I$ was denote by $A^\text{out}(\beta_I)$ in~\cite{DL22}.

Write $I=[a,b]\subset \partial \overline Z$ and denote by $a',b'$ the endpoints of $\beta_I$ as shown on Figure~\ref{fig:CD}. Denote by $a'',b''$ the intersection of the inner boundary of $A$ with $[a,a']^{m+1},[b',b'']^{m+1} \subset I^{m+1}$, where the superscript indicates the projections of the intervals onto $\wZ^{m+1}$.  The \emph{inner buffer} is defined by \[ S^\text{inn}(I):=\ [a'',b'']^m\ =\ [a'',a']^{m+1}\cup \beta_I\cup [\beta',\beta'']^{m+1} \ \subset \partial \wZ^m .\] We also define
\[ S^\text{inn} (\wZ^m):= \bigcup_{n\ge m , \ I} S^\text{inn} (I)\subset \partial \wZ^m,\]
where the union is taken over all $I\in \mathfrak D_n$ and $n\ge m$.

It is required that there is an annulus $A^\text{inn} _I$ separating $\{a'', \ b''\}$ from $\beta_I$ with $\mod (A^\text{inn})\ge \delta$ such that $(A^\text{inn}_I)_{I\in \mathfrak D_m}$ also control the difference between $f^i(\wZ^m)$ and $\wZ^m$ as in \ref{prop:wZ:main:D} above. In short: 
\begin{itemize}
    \item the $A_I\equiv A^\text{out}_I$ guarantee that wide families typically submerge into $\wZ^m$ through ``grounded intervals,'' see~\S\ref{sss:grounded intervals};
    \item the $A^\text{inn}_I$ guarantee that $f^i\mid \wZ^m$ is ``geometrically close'' to the standard rigid rotation; conseequently, $\partial \wZ^m$ has inner geometry similar to that of a Siegel disk; see~\cite[Theorem~5.12]{DL22}.
\end{itemize}

\subsubsection{Robustness of the outer geometry under $\wC\setminus \overline Z\ \leadsto \ \wC\setminus \wZ^m$} \label{sss:Z vs wZ} In~\cite[Remark~5.11]{DL22}, it is assumed that 
\[ \Width(\XX_I) \ge \Delta, \sp\sp\text{ where $\Delta>0$ is sufficiently big but fixed}\]
for all $I\in \mathfrak D_n$ and all $n\ge m$.

Consider a rectangle $\RR\subset \wC\setminus Z$. Assume that:
\begin{itemize}
    \item vertices $V_\RR$ of $\RR$ are outside of $\intr  (\wZ^m)$; and
    \item $\RR\setminus \wZ^m$ has a connected component $\RR'$ such that $V_\RR \subset \partial \RR'$.
\end{itemize}
Then $\RR'$ is Jordan domain, and we view its closure $\RR^m:= \overline \RR'$ as a rectangle with vertex set $V_\RR$ and the same orientation of sides as $\RR$. We call $\RR^m$ the \emph{restriction} of $\RR$ to $\wC\setminus \intr (\wZ^m)$.

\begin{lem}[{\cite[{(5.12) in \S5.2.4}]{DL22}}]\label{lem:C-Z to C-wZ} If $\RR^m$ is the restriction of a rectnalge $\RR$ to $\wC\setminus \intr (\wZ^m)$ as above and if the vertex set $V_\RR$ of $\RR$ is outside of $\intr \XX(\ZZ^m)$, then
\[  1-\varepsilon_\Delta <\frac{\Width(\RR^m)}{\Width(\RR)}\le 1+\varepsilon_\Delta,\]
where $\varepsilon_\Delta \to 0$ as $\Delta \to \infty$.
\end{lem}

\subsubsection{Grounded intervals}\label{sss:grounded intervals}
We will be usually working with a special type of regular intervals called {\em grounded intervals}. An interval $I\subseteq \partial \widehat Z^m$ is a grounded interval if the end points $\partial I$ are in $\partial \widehat Z^m - S_{\inn}(\wZ^m)$.

Lemma~\ref{lem:C-Z to C-wZ} implies the following fact about grounded intervals. For a pair $I, J\subset \partial \widehat Z^m$ of disjoint grounded intervals, consider a rectangle \[\RR\subset \wC\setminus Z \sp\sp\text{ with }\partial^{h,0} \RR=I , \sp \sp \partial^{h,1} \RR=J.\] 
Since $I,J$ are grounded, $\RR$ restrict to a rectangle $\RR^m$ in $\wC\setminus \wZ^m$ with $\partial^{h,0} \RR^m=I^m$ and $\partial^{h,1} \RR^m=J^m$.

Then the argument in~\cite[Lemma 5.10]{DL22} implies that
\begin{equation}
\label{eq:RR vs RR^m}
\Width(\RR)-O_\delta (1) <\Width(\RR^m) <(1+\varepsilon_\Delta)\Width(\RR) +O_\delta(1).\end{equation}
In the paper, we will usually replace ``$1+\varepsilon$'' with ``$2$''; see for example Proposition \ref{prop:gi}.

\subsection{Stability of $\wZ^m$ and pseudo-bubbles}\label{subsec:stability}
 Recall that $Z$ has period $p$. Given a peripheral closed disk $D\supset \overline Z$ and an iteration $f^{np}$, set $\widetilde D$ to be the closure of the connected component of $f^{-np}(\Int(D))$ containing $Z$. If $\widetilde D\overset{f^{np}}{\longrightarrow} D$ has degree $1$ (i.e., it is a homeomorphism), then we call \[D(-np)\ :=\ \widetilde D\] the {\em pullback} of $D$ under $f^{np}$ (rel $\overline Z$).

Observe that $D(-np)$ is well defined if and only if $\partial (D\setminus \overline Z)$ does not contain any critical value of $f^{np}$. Since $D$ is peripheral, every critical value of $f^{np}$ in $D$ are necessary on $\partial Z$.

We say that a pseudo-Siegel disk $\wZ^m$ is $k$-stable if for every $n\le k q_{m+1}$ the pullback of $\XX(\wZ^m)$ under $f^{np}$ is well defined. It follows then 
\[[\wZ^m(-n), \XX(\wZ^m)(-n)] :=(f^n)^*[\wZ^m, \XX(\wZ^m)]\]
is a well-defined pseudo-Siegel disk together with its territory.

For every $I=[a,b]\in \mathfrak D_m$, set \[k_I := \frac{\dist (\partial^h \XX_I, \ \{a,b\})}{\mathfrak{l}_{m+1}}-2,\sp \sp k_I^+:=\max \{k_I, 0\} ,\]
and $k_m :=\min_{n\ge m}\  \min_{I \in \mathfrak D_n} \{k^+_{I}\}$. Then it follows from the above discussion that $\wZ^m$ is $k_m$-stable.

Observe that if $\wZ^m$ is $T$-stable, then so is any $\wZ^n$ for $n\ge m$. In particular, if $\wZ = \wZ^{-1}$ is $T$-stable, then $\wZ^m$ is $T$-stable for all $m$.
We remark that this $T$ can be chosen arbitrarily large, (see Remark \ref{rmk:stablitiy} and \S~\ref{subsubsect:stabilityConstruction}).

\subsubsection{Pseudo-bubbles}

 A \emph{bubble} $B$ is a closed strictly preperiodic Siegel disk; i.e., it is the closure of a connected component of $f^{-k}(Z)\setminus Z$. The \emph{generation} of $B$ is the minimal $k$ such that $f^k(B)=\overline Z$; i.e. $f^k\colon B\to \overline Z$ is the first landing. Given a pseudo-Siegel disk $\widehat Z^m$, the \emph{pseudo-bubble} $\widehat B$ is the closure of the connected component of $f^{-k}\big( \intr \widehat Z^m\big)$ containing $\intr B$. In other words, $\widehat B$ is obtained from $B$ by adding the lifts of all reclaimed fjords (components of $\widehat Z^m\setminus \overline Z$) along $f^k\colon B\to \overline Z$.

Dams $\beta_I$, collars $A_I$, extra protections $\XX_I$ are defined for $\widehat B$ as pullbacks of the corresponding objects along $f^k\colon \widehat B \to  \wZ^m$. 
For instance, $\XX(\wZ_\ell)$ is the pullback of $\XX(\wZ^m)$ under $f^k$. 
The length of an interval $I\subset \partial \widehat B$ is the length of its image $f^k(I)\subset \partial \wZ^m$. 
Properties of pseudo-Siegel disks can also be obtained for pseudo-bubbles by pulling back using the dynamics.

\subsection{Convention for valuable-attracting domains and pseudo-Siegel disks}
We assume the following convention throughout the paper.
Let $\mathcal{C}$ be a cycle of attracting Fatou components of $[f]$ with period $p$, and let $D \in \mathcal{C}$ be the unique Fatou component that contains the critical point.
Then
\begin{align}\label{eq:cvd}
\widehat{f^j(D)} = f^j(\widehat D)  \text{ for all $j = 1, ..., p-1$}.
\end{align}
Similar to valuable-attracting domains, we shall use the following convention for pseudo-Siegel disks throughout the paper.
Let $\mathcal{C}$ be a cycle of Siegel disks of $[f]$ with period $p$, and let $Z \in \mathcal{C}$ be the unique Fatou component that contains the critical point on its boundary.
Then
\begin{align}\label{eq:cvz}
\widehat {f^j(Z)}^m = f^j(\widehat Z^m)  \text{ for all $m$ and for all $j = 1, ..., p-1$}.
\end{align}

\subsection{Core and pseudo-Core Surfaces}\label{ss:CoreSurfaces}
Let $\mathcal{H}$ be a hyperbolic component of disjoint type.
Let $[f]\in \partial \mathcal{H}$ be an eventually-golden-mean map.
Let
$$
Z_{1,0},... Z_{1,p_1-1}, Z_{2, 0},..., Z_{k_1, p_{k_1}-1} \text{ and } D_{1,0},... D_{1,q_1-1}, D_{2, 0},..., D_{k_2, q_{k_2}-1}
$$
be the list of Siegel disks and attracting Fatou components of $f$.
Denote the corresponding pseudo-Siegel disks and valuable-attracting domains by $\widehat Z_{i,j}^m$ and $\widehat D_{i,j}$ respectively.
As usual, we define $\widehat Z_{i,j} := \widehat Z_{i,j}^{-1} = \bigcup_m \widehat Z_{i,j}^m$.

The indices are chosen so that for all 
$$
\widehat Z_{i,j+1}^m = f(\widehat Z_{i,j}^m) \text{ for all } m\in \N, i=1,..., k_1, j = 0,1,..., p_{k_i}-2;
$$
and
$$
\widehat D_{i,j+1} = f(\widehat D_{i,j}) \text{ for all } i=1,..., k_2, j = 0, 1,..., p_{k_i}-2.
$$

By construction, $\widehat{D}_{i,j}$ and $Z_{i,j}$ are all Jordan domains with pairwise disjoint closures.
Thus, we define the {\em core surface} as the Riemann surface with boundary by
\begin{equation}
 \label{eq:defn:X_f} X_f := \widehat\C - \bigcup_{i,j} \Int(\widehat{D}_{i,j}) - \bigcup_{i,j} Z_{i,j},
\end{equation}
and the {\em pseudo-core surface} as
\begin{equation}
    \label{eq:defn:hat X_f}
\widehat X_f := \widehat\C - \bigcup_{i,j} \Int(\widehat{D}_{i,j}) - \bigcup_{i,j} \Int(\widehat{Z}_{i,j}).
\end{equation}

Since we will construct pseudo-Siegel disks for a single cycle while fixing other cycles of pseudo-Siegel disks and valuable-attracting domains, it is useful to introduce the following notations.

Let $Z_k:= Z_{k,0}$.
We define the {\em level $m$ pseudo-core surface} for $k$-th cylce of Siegel disks $Z_k$ as
$$
\widehat X_{f}^m(Z_k):=  \widehat\C - \bigcup_{i,j} \Int(\widehat{D}_{i,j}) - \bigcup_{i\neq k, j} \Int(\widehat{Z}_{i,j}) - \bigcup_{j=0}^{p_k-1} \Int(\widehat{Z}_{k,j}^m).
$$
Note that under this notation, $\widehat X_f  = \widehat X_f^{-1}(Z_k)$ for any $k$.
We also define
$$
\widehat X_f^\infty(Z_k) = \widehat\C - \bigcup_{i,j} \Int(\widehat{D}_{i,j}) - \bigcup_{i\neq k, j} \Int(\widehat{Z}_{i,j}) - \bigcup_{j=0}^{p_k-1} \Int(\overline{Z}_{k,j}),
$$
as $\widehat{Z}_{k,j}^m = \overline{Z}_{k,j}$ for all sufficiently large $m$.

To avoid too many subindices, we shall simplify the notation as $X_{f}^m:= X_{f}^m(Z_k)$ if the underlying cycle of Siegel disks is not ambiguous.

\subsubsection{Pullback of pseudo-Siegel disks and pseudo-core surfaces}
Note that pseudo-Siegel disks are not necessarily forward invariant. Thus, it is important to introduce notations for the pullbacks.
Let us assume that all pseudo-Siegel disks are $T$-stable.

For each iterate $n \leq T$, the preimage $f^{-n}(\bigcup_{i,j}\Int(\widehat{Z}^m_{i,j}))$ is union of disks, each mapped conformally to some component $\Int(\widehat{Z}^m_{i,j})$.
We denote by 
$\widehat Z_{i,j}^m(-n)$ 
the closure of the unique component of $f^{-n}(\bigcup_{i,j}\Int(\widehat{Z}^m_{i,j}))$ that contains $Z_{i,j}$.

Generalizing such notations for pseudo-core surfaces, we define
$$
\widehat X_f(-n) := \widehat\C - \bigcup_{i,j} \Int(\widehat{D}_{i,j}) - \bigcup_{i,j} \Int(\widehat{Z}_{i,j}(-n)),
$$
and
$$
\widehat X_{f}^m(-n):=  \widehat\C - \bigcup_{i,j} \Int(\widehat{D}_{i,j}) - \bigcup_{i\neq k, j} \Int(\widehat{Z}_{i,j}(-n)) - \bigcup_{j=0}^{p_k-1} \Int(\widehat{Z}_{k,j}^m(-n)).
$$
We also define 
$$
\widehat Y_f(-n) := f^{-n}(\widehat X_f),
$$
and
$$
\widehat Y_f^m(-n) := f^{-n}(\widehat X_f^m).
$$
We summarize the relations between these spaces in the following diagram.
\begin{center}
\begin{tikzcd} 
\widehat Y_f(-n)
  \arrow[drr, "f^n"]
  \arrow[d, hook]\\
\widehat X_f(-n) \arrow[dr, hook] & & \widehat X_f\arrow[dl, hook]\\
 & X_f
\end{tikzcd}
\begin{tikzcd} 
\widehat Y_f^m(-n)
  \arrow[drr, "f^n"]
  \arrow[d, hook]\\
\widehat X_f^m(-n) \arrow[dr, hook] & & \widehat X_f^m\arrow[dl, hook]\\
 & X_f^m
\end{tikzcd}
\end{center}
Here the hooked arrows represent inclusions and $f^n: \widehat Y_f(-n) \longrightarrow \widehat X_f$ or $f^n: \widehat Y_f^m(-n) \longrightarrow \widehat X_f^m$ are covering maps.

\subsection{Local degeneration on pseudo-core surfaces}\label{subsec:sfc}
In this subsection, using the dynamics of $f$ on the boundaries of the Siegel disks, we introduce some special families of curves for the pseudo-core surfaces.
The extremal widths for such families are crucial in our analysis.

Set $Z := Z_{1,0}$ with period $p = p_1$. 
Consider the level $m$ pseudo-core surface for $Z$
$$
\widehat X_f^m:= \widehat \C - \bigcup_{i\neq 1} \Int(\widehat Z_{i,j}) - \bigcup \Int(\widehat D_{i,j}) - \bigcup_{j=0}^{p-1} \Int( \widehat {Z}_{1, j}^m).
$$

Let $I \subseteq \partial \widehat Z^m$ be a regular interval.
Denote
$$
B = B(I, \lambda) := \partial \widehat X_f^m - \lambda I.
$$
We use the notations
$$
\mathcal{F}_{\lambda}^+(I) = \mathcal{F}_{\lambda, \widehat Z^m}^+(I) \text{ and } \mathcal{F}_{\lambda}(I) =\mathcal{F}_{\lambda, \widehat Z^m}(I)
$$
to denote that curves families connecting $I$ with $B$ in $\widehat X_f^m$ and in $\widehat \C$ respectively.
As we will be working with pseudo-Siegel disks of $Z$ of different levels simultaneously, the subindex $\widehat Z^m$ is sometimes added to clarify which pseudo-Siegel disk we are considering.
We denote
\begin{align}
\mathcal{W}_{\lambda}^+(I) &=\mathcal{W}_{\lambda,m}^+(I) = \mathcal{W}_{\lambda, \widehat Z^m}^+(I) := \mathcal{W}_{\widehat X_f^m}(I, B(I, \lambda)) = \mathcal{W}(\mathcal{F}_{\lambda}^+(I))\\
\mathcal{W}_{\lambda}(I) &=\mathcal{W}_{\lambda,m}(I) = \mathcal{W}_{\lambda, \widehat Z^m}(I) := \mathcal{W}_{\widehat \C}(I, B(I, \lambda)) = \mathcal{W}(\mathcal{F}_{\lambda}(I)).
\end{align}
We shall refer to the quantities $\mathcal{W}_{\lambda}^+(I)$ as {\em local degenerations}.
We remark that here local means that we have localized one end of the arcs to be in the interval $I$.
The arcs are not necessarily restricted in a local part of $\widehat X_f^m$.

The family $\mathcal{F}_{\lambda}^+(I)$ can be decomposed into the peripheral and non-peripheral parts, and we denote the corresponding families by
$$
\mathcal{F}_{\lambda}^{+, per}(I) \text{ and } \mathcal{F}_{\lambda}^{+, np}(I).
$$
Their widths are denoted by
$$
\mathcal{W}_{\lambda}^{+, per}(I) \text{ and } \mathcal{W}_{\lambda}^{+, np}(I)
$$

We remark that $\lambda$ is chosen to be a large constant. 
Thus, there is a large combinatorial distance between the end points of any arc in $\mathcal{F}_{\lambda}^+(I)$.

We also use the notation $\mathcal{W}_{m}^{+, np}(I):= \mathcal{W}_{0, m}^{+, np}(I)$, i.e. the width of non-peripheral arcs in $\widehat X^m_f$ starting on the interval $I$.

\subsubsection{Comparing local degenerations}
One of the most important properties of grounded intervals is that local degenerations behave nicely as we pass from Siegel disks to Pseudo-Sigel disks.
Moreover, for non-peripheral degenerations, we can simply replace an interval by a grounded interval with some uniform control on the correction.
\begin{prop}\label{prop:gi}
    Let $I\subseteq \partial \widehat Z^m$ be a grounded interval, and let $I^\bullet \subseteq \partial Z$ be the projection of $I$ onto $\partial Z$.
    Suppose that $\lambda \geq 10$.
    Then
    $$
 \mathcal{W}^+_{\lambda, Z}(I^\bullet) - O(1) \leq \mathcal{W}^+_{\lambda, \widehat Z^m}(I) \leq 2 \mathcal{W}^+_{\lambda, Z}(I^\bullet) + O(1).
    $$
    
For any interval $I\subset \partial Z$,  let $I^{\text{GRND}}, I^{\text{grnd}} \subseteq \partial Z$ be the smallest grounded interval of level $m$ that contains $I$ and the largest grounded interval of level $m$ that is contained in $I$ respectively.
    Then
    \begin{align*}
       \mathcal{W}^{+, np}_{Z}(I^{\text{GRND}}) - O(1) &\leq \mathcal{W}^{+, np}_{Z}(I) \leq \mathcal{W}^{+, np}_{Z}(I^{\text{GRND}}),\\
        \mathcal{W}^{+, np}_{Z}(I^{\text{grnd}}) &\leq \mathcal{W}^{+, np}_{Z}(I) \leq \mathcal{W}^{+, np}_{Z}(I^{\text{grnd}}) + O(1).
    \end{align*}
\end{prop}
\begin{proof}
The Thin-Thick Decomposition (see \cite[Theorem 7.25]{Lyu} allows us, up to $O(1)$, to replace $\mathcal{F}^+_{\lambda,  Z}(I^\bullet)$ with a union of finitely many rectangles in $\mathcal{F}^+_{\lambda,  Z}(I^\bullet)$. Therefore,~\eqref{eq:RR vs RR^m} implies the first statement.
    
    The second statement follows from the observation that $I^{\text{GRND}}\setminus I$ is within a union of at most two intervals, each being surrounded by an annulus $A^\text{out}(.)$ with modulus $\ge \varepsilon$. Therefore, the width of curves in $\mathcal{F}^{+, np}_{Z}(I^{\text{GRND}}) \setminus \mathcal{F}^{+, np}_{Z}(I) $ is bounded by $\frac 2{\varepsilon}$. A similar argument holds for $I$ and $I^\text{grnd}$.
\end{proof}

We remark that it is possible to replace $2$ by by $1+\delta$, where $\delta=\delta(\Delta)$ can be arbitrary small if the protection $\Delta$ for the Pseudo-Siegel disk is sufficiently big.

\subsection{Non-uniform construction of pseudo-Siegel disks}\label{sec:cps}

In this subsection, we construct pseudo-Siegel disks so that non-peripheral degeneration dominates peripheral degeneration -- see~ \eqref{cl:2a:thm:cps} and \eqref{cl:2a:thm:cps} below.
More precisely, with the notations introduced in \S \ref{ss:CoreSurfaces} and \S \ref{subsec:sfc}, we will prove
\begin{theorem}\label{thm:cps}
Let $\mathcal{H}$ be a hyperbolic component of disjoint type.
Let $[f]\in \partial \mathcal{H}$ be an eventually-golden-mean map with the pseudo-core surface $X_f$.
Let $K:=\mathcal{W}_{arc}(X_f)$ be the arc degeneration of $X_f$.
There exist a constant $M = M(K)$ depending on $K$, pseudo-Siegel disks $\widehat{Z}_{i,j}^m$ for all $i,j$ so that for all $\lambda \geq 10$,
\begin{enumerate}
    \item $\widehat{Z}_{i,j} = \widehat{Z}_{i,j}^{-1}$ is an $M$ quasiconformal disk;
    \item \label{cl:2:thm:cps} for every grounded interval $J \subseteq \partial Z_{i,j}$ rel $\widehat{Z}_{i,j}^m$ with $\length_{m+1} < |J| \leq \length_m$, we have
    \begin{enumerate}
        \item\label{cl:2a:thm:cps} $\mathcal{W}^{+, np}_{m}(J^m) = O(K\length_m + 1)$; and 
        \item \label{cl:2b:thm:cps}$\mathcal{W}^{+, per}_{\lambda, m}(J^{m}) = O(\sqrt{K\length_{m}} + 1)$.
    \end{enumerate}
\end{enumerate}
\end{theorem}

We remark that~\eqref{cl:2a:thm:cps} and~\eqref{cl:2b:thm:cps} can be improved for all $m$ unless $m$ is the ``special transition'' level; see refined versions in Appendix~\ref{ap:psd}, Theorem~\ref{thm:apbs}. See also Remark~\ref{rem:thm:info} for an explanation of the estimates.  

\begin{proof}
    The construction is by induction on the cycles of Siegel disks.
    Suppose that we have constructed pseudo-Siegel disks $\widehat{Z}_{i,j}^m$ for $i<k$, and we want to construct the pseudo-Siegel disks for the $k$-th cycle.
    Abusing the notations, denote the pseudo-core surface by $\widehat X_f$. 
    Let $Z := Z_{k, 1}$ be a Siegel disk for $f$.
    Note that $Z$ is a boundary component of $\widehat X_f$.
    Since we can construct $\widehat{Z}_{k, j}$ by $f^{j-1}(\widehat{Z}_{k, 1})$, it suffices to construct the pseudo-Siegel disk $\widehat{Z} = \widehat{Z}_{k, 1}$.
    After passing to an iterate, we may assume $Z$ is fixed by $f$.

    The idea is to construct a $\psi^\bullet$-ql (pseudo-bullet-quadratic-like) map (see \S \ref{sss:qlb:Defn} for the definition).
    By the construction in \S \ref{subsec:qlfromr}, we can associate a $\psi^\bullet$-ql map \[ F=(f^p,\iota) \colon U\rightrightarrows V\]
    with 
    \[\Width^\bullet(F) = 2\Width_{arc}(Z) + O(1).\]
    Since the vertical (or non-vertical) degeneration for $F$ corresponds to, up to a width of $O(1)$, the non-peripheral (or peripheral) degeneration of $\widehat X_f$ with endpoints on $Z$ (see \S \ref{ss:TTD} and \S \ref{subsec:qlfromr}), the statements for intervals on $\partial Z$ now follow from Theorem \ref{thm:apbs}; more specifically from Eqation \ref{eqn:VandP}.
    
    By Proposition \ref{prop:gi}, replacing the Siegel disk $Z$ by a pseudo-Siegel disk only changes the the degeneration by a bounded error, we conclude that statements for intervals on $\partial Z_{i,j}$ for $i < k$ still holds.
    Since there are only finitely many cycles of Siegel disks, we conclude the theorem.
\end{proof}

\begin{rmk}\label{rmk:stablitiy}
    We remark that for any given $T>1$, we can construct the pseudo-Siegel disk that are $T$-stable.
    This parameter $T$ affects only the constant $M$ and constants representing the ``$O(\ )$'' in \eqref{cl:2:thm:cps} (see \S~\ref{subsubsect:stabilityConstruction}).  

In this paper, we will select $T$ to be sufficiently big to dominate the pulled-off constant $N$ in \S \ref{sec:pc} and the constant $\mathbf{a}$ in Theorem \ref{thm:ll}. See \S \ref{subsec:constants} for the choice of these constants. This selection will be used in: 
\begin{itemize}
    \item the proof of Localization of arc degenerations in Theorem \ref{thm:ll}; 
    \item the proof of the Calibration lemma on shallow levels in Theorem \ref{thm:calls}.
\end{itemize}
\end{rmk}

\section{The pulled-off constant and expanding model}\label{sec:pc}
Let $\mathcal{H}$ be a hyperbolic component of disjoint type.
Let $[f] \in \partial \mathcal{H}$ be an eventually-golden-mean map with the pseudo-core surface, and $[f_{pcf}] \in\mathcal{H}$ be the post-critically finite center.
Let $N_{Siegel}([f])$ and $N([f_{pcf}])$ be the pulled-off constant as in Definition \ref{defn:poc}.

In this section, we will show that a pulled-off constant is uniformly bounded for a Sierpinski carpet hyperbolic component.
This is one of the key reasons why a Sierpinski carpet hyperbolic component is bounded.
\begin{theorem}[Pulled-off Principle]\label{thm:pop}
    Let $\mathcal{H}$ be a Sierpinski carpet hyperbolic component of disjoint type.
    Then there exists a constant $\mathbf{N}$ so that for any eventually-golden-mean map $[f] \in \partial \mathcal{H}$, $N([f]) \leq \mathbf{N}$.
\end{theorem}
We will deduce this theorem by justifying that the expanding model of maps in $\mathcal{H}$ persists for eventually-golden-mean maps on $\partial \mathcal H$. Then, assuming Theorem \ref{thm:cd}, we will show that the expanding model persists for all maps $\partial \mathcal H$ implying Theorem B.

\subsection{Characterization of Sierpinski carpet hyperbolic component}
\begin{theorem}\label{thm:cshc}
    Let $\mathcal{H}$ be a hyperbolic component, and let $[f_{pcf}] \in \mathcal{H}$ be the post-critically finite center.
    Then $\mathcal{H}$ is Sierpinski if and only if $N([f_{pcf}]) < \infty$.
\end{theorem}
\begin{proof}
    By \cite[Corollary 5.18]{Pil94}, the map $f_{pcf}$ has Sierpinski carpet Julia set if and only if there is no periodic Levy arc.
    Here a Levy arc is a non-peripheral simple curve $\gamma$ with endpoints in the post-critical set $P(f_{pcf})$ so that $f_{pcf}^n(\gamma)$ is isotopic rel $P(f_{pcf})$ to $\gamma$ for some $n$.
    
    If there is a periodic Levy arc, then, up to isotopy, it can be realized as a concatenation of two internal rays and, hence, $N([f_{pcf}]) = \infty$.

    Conversely, suppose that $N([f_{pcf}]) = \infty$.
    Then there exist arbitrarily long essentially disjoint pull back sequence $\gamma_0, ..., \gamma_n$.
    Note that the number of essentially disjoint isotopic classes of arcs is bounded by the topological complexity of $\widehat\C - P(f_{pcf})$. 
    Thus, for all large $n$, some pairs in $\gamma_0, ..., \gamma_n$ are isotopic.
    Therefore, there exists a periodic Levy arc.
\end{proof}

\subsection{Semiconjugacy to an expanding model}
Let $\mathcal{H}$ be a Sierpinski carpet hyperbolic component of disjoint type.
In this subsection, we will show that an eventually-golden-mean map $[f] \in \partial \mathcal{H}$ is semiconjugate to a topologically expanding map.

Let $[f_{pcf}] \in \mathcal{H}$ be the center of $\mathcal{H}$, i.e., the unique post-critical finite map in $\mathcal{H}$.
We define $\bar{f}: S^2 \longrightarrow S^2$ as the quotient map of $f_{pcf}$ by collapsing each Fatou component to a point.
Note that $\bar{f}$ is topologically expanding, as $f_{pcf}$ has Sierpinski carpet Julia set.

Let $[f] \in \partial \mathcal{H}$ be an eventually-golden-mean map.
Using renormalization theory on Siegel disks, we will prove
\begin{theorem}[Expanding model for $\partial_\egm \mathcal H   $]\label{thm:sctssm}
Let $[f] \in \partial \mathcal{H}$ be an eventually-golden-mean map.
Then there exists a topological semiconjugacy
$$
h: \widehat \C \longrightarrow S^2 \text{ with } \bar{f} \circ h = h \circ f.
$$
In particular, $N_{Siegel}([f]) \leq   N([f_{pcf}])$.
\end{theorem}
\begin{proof}
Denote the multiplier profile for $[f]$ as $(\rho_1,..., \rho_{2d-2})$.
Let 
$$
(\rho_{1,n}, ..., \rho_{2d-2, n})
$$
be rational parameters converging strongly to $(\rho_1,..., \rho_{2d-2})$, with corresponding maps $[f_n] \in \partial \mathcal{H}$.
By approximating each $[f_n]$ with hyperbolic maps in $\mathcal{H}$ radially, there exists semiconjugacy (see \cite[Theorem 1.5]{CT18})
$$
h_n: \widehat \C \longrightarrow S^2 \text{ with } \bar{f} \circ h_n = h_n \circ f_n.
$$
We assume the representatives are chosen so that $f_n \to f$ as rational maps.

Since any orbit on the boundary of a Siegel disk is dense on the boundary, it is easy to see that the Siegel disks and valuable attracting domains have disjoint closures.
By \cite[Theorem 6.9]{DLS20}, for sufficiently large $n$, we can find {\em parabolic valuable flowers} $L_{i,j,n}$ approximating the Siegel disks $Z_{i,j}$.
Therefore, we can find disjoint small neighborhoods
$U_{i,j}$ of $Z_{i,j}$ and $W_{i,j}$ of $\widehat D_{i,j}$ so that for sufficiently large $n$, we have
$L_{i,j,n} \subseteq U_{i,j}$ and $\widehat D_{i,j,n} \subseteq W_{i,j}$.
For sufficiently large $n$, we can find a small perturbation $h^0$ of $h_n$ so that $h^0(U_{i,j})$ and $h^0(W_{i,j})$ are points.
Note that the union $U = \bigcup U_{i,j} \cup \bigcup W_{i,j}$ contains the union of parabolic valuable flowers and valuable attracing domains of $f_n$, so $U$ also contains the post-critical set of $f_n$.
Then we can pull back $h^0$ and get 
\begin{align*}
h^1_n: \widehat \C &\longrightarrow S^2 \text{ with } \bar{f} \circ h^0 = h^1_n \circ f_n,\\
h^1: \widehat \C &\longrightarrow S^2 \text{ with } \bar{f} \circ h^0 = h^1 \circ f.
\end{align*}
Since $h_n$ is a semiconjugacy between $f_n$ and $\bar{f}$, $h^1_n \sim h^0$ on $X_f$ for sufficiently large $n$.
Since $f_n \to f$, $h^1 \sim h^0$ on $X_f$ as well.

Since $\Int(X_f)$ contains no post-critical point and $\bar{f}$ is topologically expanding, a standard pull-back argument gives the semiconjugacy.

Note that any laminally disjoint pull-back sequence for an eventually-golden-mean map $[f]$ gives a laminally disjoint pull-back sequence for $\bar f$.
If $\gamma$ for $[f]$ connects boundaries of Siegel disks, then the corresponding arc $\delta$ for $\bar f$ connects points in critical periodic cycles.
Since each laminally disjoint pull-back sequence is essentially disjoint, and $\bar f$ is homotopically equivalent to $f_{pcf}$, we have that $N_{Siegel}([f]) \leq N([f_{pcf}])$.
\end{proof}

\begin{proof}[Proof of Theorem \ref{thm:pop}]
By Theorem \ref{thm:sctssm}, $N_{Siegel}([f]) \leq N([f_{pcf}])$.
By Theorem \ref{thm:cshc}, $N([f_{pcf}]) < \infty$.
Therefore, $N_{Siegel}([f])$ is uniformly bounded.
\end{proof}

\subsection{Proof of Theorem \ref{thm:B} (assuming Theorem~\ref{thm:A} and \ref{thm:cd})}\label{ss:proofB}
Recall that $\bar f: S^2 \longrightarrow S^2$ is the topologically expanding map obtained from collapsing Fatou components of the center $[f_{pcf}] \in \mathcal{H}$.
\begin{theorem}[Expanding model for maps in $\partial \mathcal H$]\label{thm:scem}
Let $[f] \in \overline{\mathcal{H}}$.
There exists a topological semiconjugacy
$$
h: \widehat \C \longrightarrow S^2 \text{ with } \bar{f} \circ h = h \circ f.
$$
\end{theorem}
\begin{proof}
Let $[f] \in \partial \mathcal{H}$ with multiplier profile $(\rho_1, ..., \rho_{2d-2})$.
Let $[f_n] \in \partial \mathcal{H}$ be a sequence of eventually-golden-mean maps with
\begin{itemize}
\item $[f_n] \to [f]$; and 
\item its multiplier profile $(\rho_{1,n}, ..., \rho_{2d-2,n}) \to_s (\rho_1, ..., \rho_{2d-2})$.
\end{itemize}
Assume that the representatives are chosen so that $f_n \to f$.

By Theorem \ref{thm:cd},
after passing to a subsequence, we may assume pseudo-Siegel disks $\widehat Z_{i,j,n}$ and valuable-attracting domains $\widehat D_{i,j,n}$ converge in Hausdorff topology to $\widehat Z_{i,j}$ and $\widehat D_{i,j}$.
Note that $\widehat Z_{i,j}$ and $\widehat D_{i,j}$ contain the post-critical set of $f$.
Denote the Riemann surface
$$
\widehat X_f:= \widehat\C - \bigcup \Int(\widehat Z_{i,j}) - \bigcup \Int(\widehat D_{i,j}).
$$

By Theorem \ref{thm:sctssm}, there exists semiconjugacies 
$$
h_n: \widehat \C \longrightarrow S^2 \text{ with } \bar{f} \circ h = h \circ f_n.
$$

We can construct a similar purtabation $h^0$ of $h_n$ for some sufficiently large $n$.
Let $h^1$ be the pull back of $h^0$ under $f$.
A similar argument as in Theorem \ref{thm:sctssm} gives that $h^0 \sim h^1$ on $\widehat X_f$ with
$$
\bar{f} \circ h^0 = h^1 \circ f
$$ 

Since $\widehat X_f$ is disjoint from the post-critical set and $\bar{f}$ is topologically expanding, a standard pull back argument gives the semiconjugacy $h$.
\end{proof}

Theorem \ref{thm:B} now follows immediately from Theorem \ref{thm:scem}.
\begin{proof}[Proof of Theorem \ref{thm:B}]
Let $\tilde{x} = h(x) \in S^2$, where $h$ is the semiconjugacy in Theorem \ref{thm:scem}.
Let $\tilde{x}\in \tilde{U}$ be a small neighborhood so that $\tilde{U} \Subset \bar{f}^p(\tilde{U})$.
Let $U = h^{-1}(\tilde{U})$.
Then $f^p: U \longrightarrow V= f^p(U)$ gives the quadratic-like restriction.
Since $\overline{\mathcal{H}}$ is compact by Theorem~\ref{thm:A}, the modulus of $V - \overline{U}$ is uniformly bounded.
This proves the theorem.
\end{proof}

\section{Localization of arc degeneration}\label{sec:ltd}
In this section, we will prove that if the arc degeneration $\mathcal{W}_{arc}(\widehat X_f)$ of the pseudo-core surface $\widehat X_f$ is sufficiently big, then there exists some small grounded interval $I$ whose local degeneration $\mathcal{W}_{\lambda}^+(I)$ is at least comparable to $\mathcal{W}_{arc}(\widehat X_f)$; compare with~\S\ref{sss:Localization}. 
More precisely, we will prove
\begin{theorem}[Localization of arc degeneration]\label{thm:ll}
Let $\mathcal{H}$ be a hyperbolic component of disjoint type.
Let $[f] \in \partial \mathcal{H}$ be an eventually-golden-mean map with pulled-off constant $N = N_{Siegel}([f])$.
There exist 
\begin{itemize}
    \item a constant $\mathbf{a} > 1$ that depends on $N$,
    \item and a threshold constant $\boldsymbol{\Lambda} \gg 1$
\end{itemize}
such that for every $\lambda \geq  \boldsymbol{\Lambda}$ and for every $0 < \epsilon < 1/2\lambda$, there exists a threshold constant $\mathbf{K}_{\epsilon,\lambda, N} \gg 1$ depending on $\epsilon, \lambda, N$ and the multipliers of attracting cycles of $f$ with the following properties.

Suppose that the Riemann surface $\widehat X_f:= \widehat \C - \bigcup \Int(\widehat Z_{i,j}) - \bigcup \Int(\widehat D_{i,j}) $ has
$$
\mathcal{W}_{arc}(\widehat X_f) := K \geq \mathbf{K}_{\epsilon,\lambda, N}.
$$
{Suppose that all pseudo-Siegel disks are at least $N$-stable.}
Then there exists a pseudo-Siegel disk $\widehat Z = \widehat Z_{i,j}$ and a grounded interval $I \subseteq \partial \widehat Z$ with $|I| \leq \epsilon$ such that
$$
\mathcal{W}^{+, np}(I) +\mathcal{W}_{\lambda}^{+, per}(I) \geq K/\mathbf{a}.
$$
\end{theorem}

\begin{rmk}
    We remark that if $\mathcal{H}$ is a Sierpinski carpet hyperbolic component, then by Theorem \ref{thm:pop}, the pulled-off constant is uniformly bounded.
    In this case, the constant $\mathbf{a}$ can be chosen to be universal, and the constant $\mathbf{K}_{\epsilon,\lambda, N}$ depends only on $\epsilon, \lambda$ and the multipliers of the attracting cycles.
\end{rmk}

\subsection*{Pulled-off argument}
We will follow the notations introduced in \S \ref{ss:CoreSurfaces}.
It follows from Proposition \ref{prop:gi} and the fact that the Siegel disks are $N$-stable that for sufficiently large arc degeneration $K = \mathcal{W}_{arc}(\widehat X_f)$, we have
$$
\frac{1}{8}\mathcal{W}_{arc}(\widehat X_f) \leq \mathcal{W}_{arc}(\widehat X_f(-N)) \leq 8 \mathcal{W}_{arc}(\widehat X_f).
$$

Note that there are only finitely many homotopy classes of non-peripheral arcs $\gamma \subseteq \widehat X_f(-N)$ with $\mathcal{W}(\gamma) \geq 2$.
This number is bounded by a constant $M$, which depends only on the number of boundary components of $\widehat X_f(-N)$.
Let $\gamma \subseteq \widehat X_f(-N)$ be a non-peripheral arc with 
$$
\mathcal{W}(\gamma) \geq \frac{\mathcal{W}_{arc}(\widehat X_f(-N))}{2M} \geq \frac{\mathcal{W}_{arc}(\widehat X_f)}{16M} = \frac{K}{16M}.
$$
We may realize such wide families by a rectangle $R_{\gamma}$ whose vertical arcs are homotopic to $\gamma$ and 
$$
\mathcal{W}(\mathcal{F}_\gamma) = \mathcal{W}(\gamma) - O(1),
$$
where $\mathcal{F}_\gamma$ is the family of vertical arcs in $R_{\gamma}$.

By our construction, the modulus of the annulus $D_{i,j} - \widehat D_{i,j}$ is bounded below in terms of the multipliers of the attracting cycles of $f$.
By making the threshold $\mathbf{K}_{\epsilon,\lambda, N}$ larger if necessary, we may assume $\gamma$ connects two pseudo-Siegel disks $\widehat Z(-N)$ and $\widehat Z'(-N)$.
Note that $\widehat Z(-N)$ may equal to $\widehat Z'(-N)$.

Let $U$ be a component of
$f^{-N}(\bigcup Z_{i,j})$.
Denote $\widehat U(-N)$ as the corresponding pseudo-Siegel disks, i.e., $\widehat U(-N)$ is the closure of the component of $\widehat\C - \widehat Y_f(-N)$ that contains $U$.
\begin{lem}[Submergence into $U$]\label{lem:pull}
There exists 
\begin{itemize}
    \item a constant $\mathbf{a}_1$ depending on $N$,
    \item a strictly pre-periodic Siegel disk $U \subseteq f^{-\mathbf{N}}(\bigcup Z_{i,j})$ so that $\partial Z$ and $\partial U$ are in different connected components of $\partial Y_f(-N)$,
    \item a family $\mathcal{G}$ of homotopically equivalent non-peripheral arcs connecting $\widehat Z$ and $\widehat U$, and 
    \item a subfamily $\mathcal{F}_1 \subseteq \mathcal{F}_\gamma$
\end{itemize}   
so that
\begin{itemize}
\item $\mathcal{W}(\mathcal{F}_1) \geq K/\mathbf{a}_1$;
\item the family $\mathcal{F}_1$ overflows $\mathcal{G}$.
\end{itemize}
\end{lem}
\begin{proof}
Note that $\gamma \subseteq \widehat X_f(-N)$ determines a unique homotopy class of non-peripheral arc $\widetilde \gamma \subseteq \widetilde X_f$.
We first claim after removing some buffers, we may assume any arc in $\mathcal{F}_\gamma$ intersects some strictly pre-periodic component of $f^{-N}(\bigcup Z_{i,j} \cup \bigcup \Int(\widehat D_{i,j}))$.
Indeed, otherwise, we get $N$ disjoint rectangles $R_0 \subseteq R_\gamma, R_1 = f(R_0), ..., R_N = f^N(R)$. 
All vertical arcs of each of the rectangles are homotopic to some non-peripheral arc in $X_f$. 
This produces a disjoint pull back sequence $\gamma_0, ..., \gamma_{N}$.
This is impossible by the definition of pulled-off constant $N$.

Since the modulus of the annulus $D_{i,j} - \widehat D_{i,j}$ is bounded below, by making the threshold $\mathbf{K}_{\epsilon,\lambda, N}$ larger if necessary, we may assume there exists a subfamily $\mathcal{F}' \subseteq \mathcal{F}_\gamma$ so that
\begin{itemize}
\item $\mathcal{W}(\mathcal{F}') \geq \mathcal{W}(\mathcal{F}_\gamma) / 2$;  and 
\item no curve in $\mathcal{F}'$ intersects the preimages of valuable-attracting domains $f^{-N}(\bigcup \Int(\widehat D_{i,j}))$.
\end{itemize}
Thus, any arc in $\mathcal{F}'$ must intersect some strictly pre-periodic Siegel disk in $f^{-N}(\bigcup Z_{i,j}) \subseteq f^{-N}(\bigcup \Int(\widehat Z_{i,j}))$.

Note that there are a bounded number (depending on $N$) of strictly pre-periodic Siegel disks $U \subseteq f^{-N}(\bigcup Z_{i,j})$.
So there are a bounded number of homotopy classes of wide non-peripheral arcs in $Y_f(-{N})$.
Thus there exists a constant $\mathbf{a}_1$ depending on $N$, some family $\mathcal{G}$ of homotopically equivalent arcs connecting $\widehat Z$ and a strictly pre-periodic pseudo-Siegel disk $\widehat U(-{N})$ so that the arcs in $\mathcal{F}'$ overflowing $\mathcal{G}$ has width at least $K/\mathbf{a}_1$. 
Let $\mathcal{F}_1 \subseteq \mathcal{F}'$ be this collection of arcs and we conclude the lemma.
\end{proof}

Let $\mathcal{F}_1$ be the family of arcs in Lemma \ref{lem:pull}.
Consider an arc $\gamma: [0,1] \longrightarrow \widehat X_f(-{N})$ in $\mathcal{F}_1$ with $\gamma(0) \in \partial \widehat Z(-{N})$ and $\gamma(1) \in \partial \widehat Z'(-{N})$.
Let $t_0 > 0$ be the first time that $\gamma(t_0) \in \partial \widehat Y_f(-{N})$.
Let $\gamma' = \gamma|_{[0, t_0]}$.
By our construction, $\gamma' \in \mathcal{G}$. Thus $\gamma' \subseteq \widehat Y_f(-{N})$ is an arc connecting $\widehat Z(-{N})$ and $\widehat U(-{N})$.

\begin{figure}[ht]
  \centering
  \resizebox{\linewidth}{!}{
    \def\svgwidth{\columnwidth}
\begingroup%
  \makeatletter%
  \providecommand\color[2][]{%
    \errmessage{(Inkscape) Color is used for the text in Inkscape, but the package 'color.sty' is not loaded}%
    \renewcommand\color[2][]{}%
  }%
  \providecommand\transparent[1]{%
    \errmessage{(Inkscape) Transparency is used (non-zero) for the text in Inkscape, but the package 'transparent.sty' is not loaded}%
    \renewcommand\transparent[1]{}%
  }%
  \providecommand\rotatebox[2]{#2}%
  \newcommand*\fsize{\dimexpr\f@size pt\relax}%
  \newcommand*\lineheight[1]{\fontsize{\fsize}{#1\fsize}\selectfont}%
  \ifx\svgwidth\undefined%
    \setlength{\unitlength}{751.18110236bp}%
    \ifx\svgscale\undefined%
      \relax%
    \else%
      \setlength{\unitlength}{\unitlength * \real{\svgscale}}%
    \fi%
  \else%
    \setlength{\unitlength}{\svgwidth}%
  \fi%
  \global\let\svgwidth\undefined%
  \global\let\svgscale\undefined%
  \makeatother%
  \begin{picture}(1,0.60377358)%
    \lineheight{1}%
    \setlength\tabcolsep{0pt}%
    \put(0,0){\includegraphics[width=\unitlength,page=1]{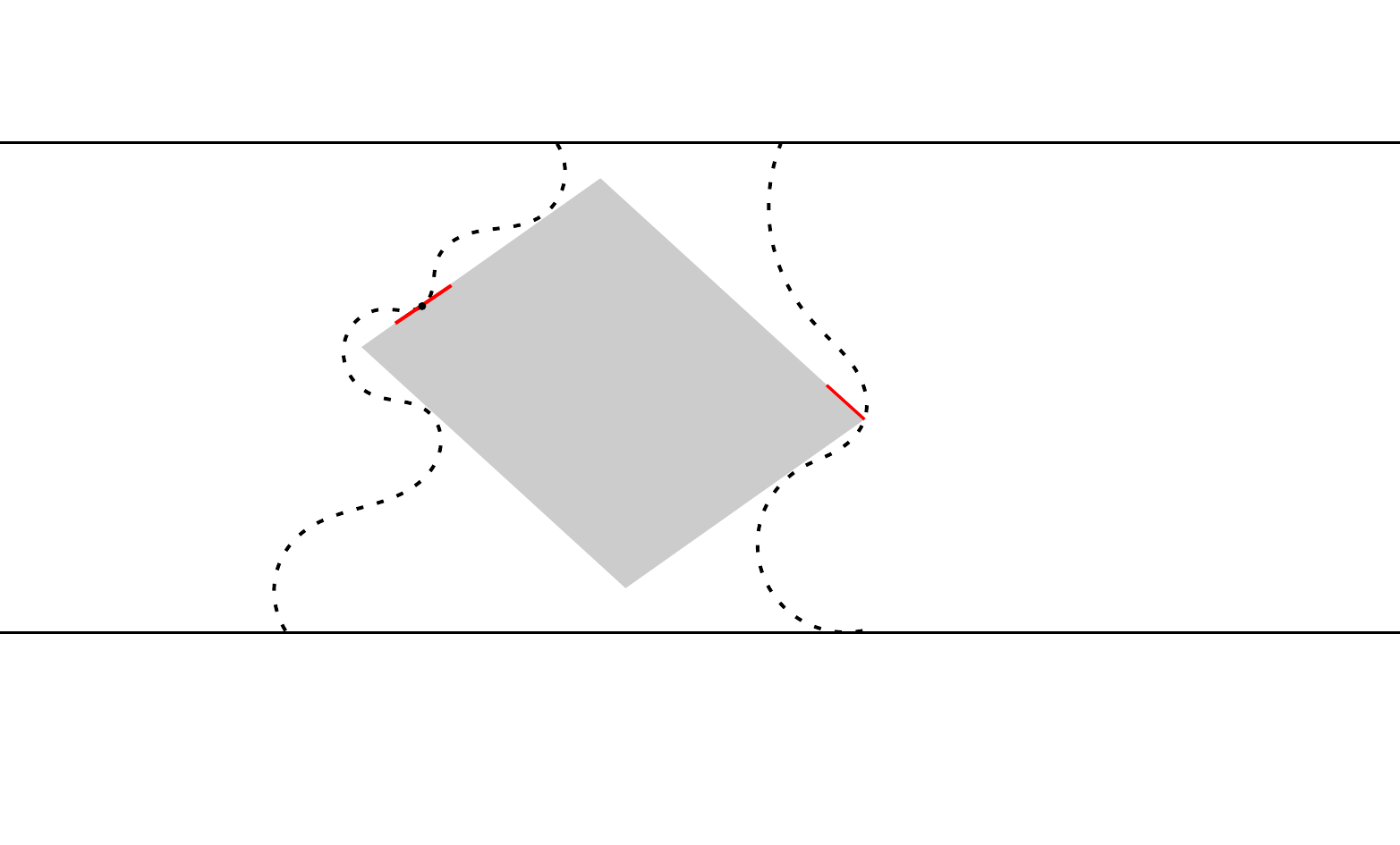}}%
    \put(0.21965407,0.25717655){\color[rgb]{0,0,0}\makebox(0,0)[lt]{\lineheight{1.25}\smash{\begin{tabular}[t]{l}$\gamma_1$\end{tabular}}}}%
    \put(0.56638944,0.40853561){\color[rgb]{0,0,0}\makebox(0,0)[lt]{\lineheight{1.25}\smash{\begin{tabular}[t]{l}$\gamma_2$\end{tabular}}}}%
    \put(0,0){\includegraphics[width=\unitlength,page=2]{LT.pdf}}%
    \put(0.3038073,0.35773249){\color[rgb]{1,0,0}\makebox(0,0)[lt]{\lineheight{1.25}\smash{\begin{tabular}[t]{l}$I_1$\end{tabular}}}}%
    \put(0.56320225,0.30380934){\color[rgb]{1,0,0}\makebox(0,0)[lt]{\lineheight{1.25}\smash{\begin{tabular}[t]{l}$I_2$\end{tabular}}}}%
    \put(0.43003707,0.34061656){\color[rgb]{0,0,0}\makebox(0,0)[lt]{\lineheight{1.25}\smash{\begin{tabular}[t]{l}{\Large$\widehat U$}\end{tabular}}}}%
  \end{picture}%
\endgroup%

  }
  \caption{The curve $\gamma_1$ and $\gamma_2$ are the left and right most arcs in $\mathcal{F}_1$. Most of the arcs in $\mathcal{F}_1$ passes through $I_1$ or $I_2$.}
  \label{fig:LT}
\end{figure}

\begin{lem}[Localization of the submergence as $I'\subset \partial \widehat U(-N)$]
\label{lem:4.4}
There exists a threshold constant $\mathbf{K}_{\epsilon,\lambda, N} \gg 1$ so that if 
$$
K = \mathcal{W}_{arc}(\widehat X_f) \geq \mathbf{K}_{\epsilon,\lambda, N},
$$ 
then
there exist a constant $\mathbf{a}_2$ depending on $N$ and a grounded interval $I' \subseteq \partial \widehat U(-{N})$ with $\vert I' \vert \leq \epsilon$ so that the collection $\mathcal{F}_2 \subseteq \mathcal{F}_1$ of arcs passing through $I'$ has width $\mathcal{W}(\mathcal{F}_2) \geq K/ \mathbf{a}_2$.
\end{lem}
\begin{proof}
Let $\gamma_1, \gamma_2 \in \mathcal{F}_1$ be the left and right most arcs in $\mathcal{F}_1$.
Let $x_i$ be an intersection point of $\gamma_i$ with $\partial \widehat U(-{N})$.
Since the end points of $\gamma_i$ are outside of $\XX(\widehat U(-{N}))$ and the extra outer protection $\XX_I$ has width bounded below, by removing a collection of arcs of bounded width, we may assume that $x_i$ is away from the extra outer protections $\XX(\widehat U(-{N}))$.

Let $x_i \in I_i \subseteq \partial \widehat U(-{N})$ be a grounded interval with $\vert I_i \vert \leq \epsilon$.
Let $\mathcal{F}' \subseteq \mathcal{F}_1$ be the family of arcs that is disjoint from $I_1 \cup I_2$.
Let $\mathcal{F}''$ be the collection of arcs $\delta \subseteq \widehat U(-{N})$ connecting the two components of $\partial U - I_1 - I_2$.
Then any arc $\mathcal{F}'$ must overflow an arc in $\mathcal{F}''$.
However, the width $\mathcal{W}(\mathcal{F}'') \preceq \vert \log \epsilon \vert$.
Thus, $\mathcal{W}(\mathcal{F}') \preceq \vert \log \epsilon \vert$.

By choosing the threshold $\mathbf{K}_{\epsilon,\lambda, N}$ larger if necessary, we may assume $K/ \mathbf{a}_1 \gg \vert \log \epsilon \vert$.
Thus the collection of arcs passing through $I_1 \cup I_2$ has width $\geq K/ 2\mathbf{a}_1$.
Without loss of generality, we assume the arcs of at least half of the width pass through $I_1$.
The lemma now follows by letting $I' = I_1$.
\end{proof}

We are ready to prove Theorem \ref{thm:ll}. In the proof, we first push forward by $f^N$ the wide family from Lemma~\ref{lem:4.4} passing through $I'$ to obtain a wide family $\mathcal F'$ based at $I''=f^N(I')\subset \partial \widehat Z$. We then lift the appropriate restriction of $\mathcal F'$ to the associated $\psi^\bullet$-map $\bf g$ around $\widehat Z$. Applying Lemma~\ref{lem:rect through PB}, we construct an appropriate interval $\widetilde{I}$ in the dynamical plane of $\bf g$. Using natural properties of the Thin-Thick decomposition, the projection of $\widetilde I$ back to $\partial \widehat Z$ gives a required interval $I$.

\begin{proof}[Proof of Theorem \ref{thm:ll}]
Consider an arc $\gamma: [0,1] \longrightarrow \widehat X_f(-{N})$ in $\mathcal{F}_2$ with $\gamma(0) \in \partial \widehat Z(-{N})$ and $\gamma(1) \in \partial \widehat Z'(-{N})$.
Let $t_0 > 0$ be the first time that $\gamma(t_0) \in I'$.
Denote the truncation $\gamma|_{[0, t_0]}$ by $\gamma'$.
Let $\mathcal{F}_2'$ be the collection of such truncations 
$$
\mathcal{F}_2' := \{\gamma': \gamma\in \mathcal{F}_2\}.
$$
{Let $I'' = f^N(I')$.
Then $I''$ is a grounded interval on some periodic pseudo-Siegel disk $\widehat Z''$.
Since $\partial Z$ and $\partial U$ are in different connected components of $\partial Y_f(-N)$, we conclude that $\alpha:= f^N(\gamma')$ is a non-peripheral arc starting at $I''$, i.e., $\alpha$ is an arc in $\widehat\C$ so that at least one component of $\alpha \cap \widehat X_f$ is non-peripheral.
Let $\alpha'$ be the smallest sub-arc of $\alpha$ starting at $I''$ so that $\alpha' \cap \widehat X_f$ contains a non-peripheral component, and let $\mathcal{F}_3$ be the collection of such truncations.
Note that $\mathcal{W}(\mathcal{F}_3) \geq K/\mathbf{a}_3$ for some constant $\mathbf{a}_3$ depending only on $N$.
Let $\bf g$ be the associated $\psi^\bullet$-map around $\widehat Z$. Then the arcs in $\mathcal{F}_3$ can be lifted to vertical arcs for the $\psi^\bullet$-map (see \S \ref{subsec:qlfromr} and \S \ref{subsec:aprioribounds}).
Let $\widetilde I$ be the lift of the interval $I''$.
By the Thin-Thick Decomposition (see \S \ref{ss:TTD}), we obtain a rectangle $\RR$ of width $K/\mathbf{a}_4$ consisting of vertical arcs that start at $\widetilde I$, for some constant $\mathbf{a}_4$ depending only on $N$.
We now apply Lemma \ref{lem:rect loc} with $\lambda' = \max \{\lambda, \frac{1}{\epsilon}\}$.
We choose the threshold $\mathbf{K}_{\epsilon,\lambda, N}$ large enough so that there exists some constant $\mathbf{a}_5$ depending only on $N$ so that there exists either
\begin{itemize}
    \item a subrectangle  $\RR_1$ of $\RR$ with $\Width(\RR_1) \ge K/\mathbf{a}_5 $ such that $\RR_1$ is outside of $\intr \widetilde \wZ$; or   
    \item a grounded interval $J\subset \partial \widetilde \wZ$ such that $\Width^{+, per}_{\lambda}(J) \geq \Width^{+, per}_{\lambda'}(J) \ge K/\mathbf{a}_5$.
\end{itemize}
Note that in the second case, since $\Width^{+, per}_{\lambda'}(J)> 0$, we have $|J| < \frac{1}{\lambda'} \leq \epsilon$.
We project the wide lamination down to the dynamical plane of $f$.
In the first case, we obtain some grounded interval $I \subseteq I''$ with 
$$
\mathcal{W}^{+,np}(I) \geq K/\mathbf{a}.
$$
for some constant $\mathbf{a}$ depending on $N$. 
In the second case, we obtain some grounded interval $I \subseteq \wZ$ with $|I| < \epsilon$ so that 
$$
\mathcal{W}_{\lambda}^{+,per}(I) \geq K/\mathbf{a}.
$$
This proves the theorem.
}
\end{proof}

\section{Calibration lemma on shallow levels for $\Width^{+, np}_m(I)$}\label{sec:cld}
In this section, we will prove a calibration lemma for non-peripheral arc degenerations. 
Roughly speaking, we will show that if there exists a grounded interval $I$ with sufficiently large non-peripheral arc degeneration $\mathcal{W}_{m}^{+, np}(I)$, then there is a grounded interval on a deeper level with comparable local degeneration; see also~\S \ref{sss:Calibration}.

\begin{theorem}[Calibration lemma on shallow levels]\label{thm:calls}
Let $\mathcal{H}$ be a hyperbolic component of disjoint type.
Let $[f] \in \partial \mathcal{H}$ be an eventually-golden-mean map  with pulled-off constant $N = N([f])$.
Let $Z$ be a Siegel disk of period $p$ and $\widehat Z = \widehat Z^m$ be a pseudo-Siegel disk of level $m$.

For every $a > N$, there is a constant $\boldsymbol{\chi}_{a} > 1$ 
and a threshold constant $\mathbf{K}_{a} > 1$ with the following property. 

Suppose that $\mathcal{W}_{arc}(\widehat X_f) \geq \mathbf{K}_{a}$, {that all pseudo-Siegel disks are at least $4apN$-stable} and that
$I$ is a grounded interval with $\mathfrak{l}_{m+1} < \vert I \vert \leq \mathfrak{l}_m$, $\mathfrak{l}_m > 1/4a$ such that 
$$
K:= \mathcal{W}_{m}^{+, np}(I) \geq \mathcal{W}_{arc}(\widehat X_f)/a.
$$
Then there is a grounded interval $J \subseteq \partial \widehat Z$ with $\vert J \vert \leq \mathfrak{l}_{m+1}$ such that $$\mathcal{W}_{m+1}^{+, np}(J) \geq K/\boldsymbol{\chi}_{a} \geq \mathcal{W}_{arc}(\widehat X_f)/\boldsymbol{\chi}_{a}a.
$$
\end{theorem}
We remark that the shallow level refers to that $\mathfrak{l}_m$ (and hence $m$) is bounded from below.

\subsection*{Bounds on $\mathfrak{q}_{m+1}$}
One important observation is the following lemma, which bounds the iterations of $f$ to consider.
\begin{lem}\label{lem:q<a}
Suppose that $\mathfrak{l}_m \geq 1/4a$.
Then
$$
\mathfrak{q}_{m+1} \leq 4a.
$$
\end{lem}
\begin{proof}
Note that by spreading around a level $m$ interval $I$, we get $\mathfrak{q}_{m+1}$ disjoint intervals of length $\mathfrak{l}_m$:
$$
I_0 = I, I_1 = f^{i_1p}(I), ..., I_{\mathfrak{q}_{m+1} - 1} = f^{(i_{\mathfrak{q}_{m+1}-1})p}(I), \, i_j \in \{1,2, ..., \mathfrak{q}_{m+1} -1\}.
$$
So
$\mathfrak{q}_{m+1} \mathfrak{l}_m \leq 1$.
Thus, if $\mathfrak{l}_m \geq 1/4a$, $\mathfrak{q}_{m+1} \leq 1/ \mathfrak{l}_{m} \leq 4a$.
\end{proof}

\subsection*{Proof of the calibration lemma}
Following the definitions in \S \ref{ss:CoreSurfaces}, 
let $\widehat Z_{i,j}(-n)$ be the closure of the component of $f^{-n}(\Int(\bigcup \widehat Z_{i,j}))$ that contains $Z_{i,j}$. Since $\mathfrak{q}_{m+1}p \leq 4ap$ by Lemma \ref{lem:q<a} and all pseudo-Siegel disks are $4apN$-stable, we define
$$
\widehat X_f(-N\mathfrak{q}_{m+1}p) := \widehat\C - \bigcup \Int(\widehat D_{i,j}) - \bigcup \Int(\widehat Z_{i,j}(-N\mathfrak{q}_{m+1}p)),
$$
and
$$
\widehat Y_f(-N\mathfrak{q}_{m+1}p) := f^{-N\mathfrak{q}_{m+1}p}(\widehat X_f).
$$
Since $\mathfrak{q}_{m+1}p \leq 4ap$ and $N < a$, the topological complexity, i.e., the number of boundary components of $\widehat Y_f(-N\mathfrak{q}_{m+1}p)$ is bounded in terms of $a$.

Note that since the interval $I$ is grounded, 
by Proposition \ref{prop:gi},
the wide families for $I$ of $\widehat X_f$ and $\widehat X_f(-N\mathfrak{q}_{m+1}p)$ have compatible width, as they are both compatible to the corresponding family in $X_f$.
Denote the corresponding family of $\widehat X_f(-N\mathfrak{q}_{m+1}p)$ for $\mathcal{W}_{m}^{+, np}(I)$ by $\mathcal{F}$.
Let $\gamma$ be a directed arc in $\widehat X_f(-N\mathfrak{q}_{m+1}p) \subseteq \widehat Y_f(-N\mathfrak{q}_{m+1}p)$.
The {\em initial segment} $\delta$ of $\gamma$ in $\widehat Y_f(-N\mathfrak{q}_{m+1}p)$ is the first segment of the union of arcs $\gamma \cap Y_f(-N\mathfrak{q}_{m+1}p)$.
We say two initial segments $\delta_1, \delta_2$ are {\em homotopic} if they are homotopic in $\widehat Y_f(-N\mathfrak{q}_{m+1}p)$ and they both connect $\partial{U}$ with $\partial{V}$ where $U, V$ are component of $\widehat\C - \widehat Y_f(-N\mathfrak{q}_{m+1}p)$.
We remark that the homotopy condition does not imply the second condition as $\partial U \cup \partial V$ may be connected.

Since the topological complexity of $\widehat Y_f(-N\mathfrak{q}_{m+1}p)$ is bounded in terms of $a$, there exists a constant $C_1 = C_1(a)$ depending on $a$ and a subfamily $\mathcal{F}_1 \subseteq \mathcal{F}$ with
\begin{itemize}
\item $\mathcal{W}(\mathcal{F}_1) \geq \mathcal{W}(\mathcal{F})/C_1 = K/C_1$; and
\item all arcs in $\mathcal{F}_1$ have homotopic initial segments.
\end{itemize}
Note that we may assume $\mathcal{F}_1$ forms the vertical foliations of a rectangle connecting $I_1 \subseteq I$ and $L_1 \subseteq \partial \widehat X_f(-N\mathfrak{q}_{m+1}p)$.
Since the non-peripheral arc degeneration for $I_1$ is large, by Proposition \ref{prop:gi}, we may assume that $I_1$ is grounded.
Since $\vert I_1 \vert \leq \length_m$, there are at most $N$ critical points of $f^{N\mathfrak{q}_{m+1}p}$ in $I_1$. Subdividing $I_1$ into $N+1$ subintervals if necessary, we may also assume that there are no critical points of $f^{N\mathfrak{q}_{m+1}p}$ on $I_1$.

We are now ready to prove Theorem \ref{thm:calls}.
\begin{proof}[Proof of Theorem \ref{thm:calls}]
Suppose for contradiction that any grounded interval $J \subseteq \partial Z$ with $\vert J \vert \leq \length_{m+1}$ satisfies $\mathcal{W}^{+, np}(J) \leq K/\boldsymbol{\chi}_a$, where $\boldsymbol{\chi}_a$ is some constant to be determined.

Note that if $\vert I_1 \vert \leq \mathfrak{l}_{m+1}$, then we may take $\boldsymbol{\chi}_a \geq C_1$ and obtain a contradiction.
Thus, we may assume $\mathfrak{l}_{m+1} < \vert I_1 \vert \leq \mathfrak{l}_m$.
Then the symmetric difference 
$$
(f^{N\mathfrak{q}_{m+1}p}(I_1) - I_1) \cup (I_1 - f^{N\mathfrak{q}_{m+1}p}(I_1))
$$
consists of $2N$ number of level $m+1$ combinatorial intervals.
Note by assumption, the widths $\mathcal{W}^{+, np}$ for these combinatorial intervals are bounded by $K$.
Thus, the width $\mathcal{W}^{+, np}$ for the union of these $2N$ intervals is bounded from above by $2N K/\boldsymbol{\chi}_a$.
Thus there exists a rectangle in $\widehat X_f$ with base $f^{N\mathfrak{q}_{m+1}p}(I_1)$ so that the family $\widetilde{\mathcal{F}}$ of its vertical arcs satisfies
$$
|\mathcal{W}(\widetilde{\mathcal{F}}) - \mathcal{W}(\mathcal{F}_1)| = O(K/\boldsymbol{\chi}_a).
$$
Let $\mathcal{G}$ be the pull back $\widetilde{\mathcal{F}}$ under $f^{N\mathfrak{q}_{m+1}p}$ that starts at the interval $I_1$.
Since $f^{N\mathfrak{q}_{m+1}p}$ is univalent on $I_1$, we have
$|\mathcal{W}(\mathcal{G}) - \mathcal{W}(\widetilde{\mathcal{F}})| = O(1)$ (see \cite[Lemma A.10]{DL22}).
Thus,
$$
|\mathcal{W}(\mathcal{G}) - \mathcal{W}(\mathcal{F}_1)| = O(K/\boldsymbol{\chi}_a).
$$
After removing two $3NK/\boldsymbol{\chi}_a$-buffers from $\mathcal{F}_1$, we get a subfamily $\mathcal{F}_{1, new} \subseteq \mathcal{F}_1$ starting at some interval $I_{1, new}$.
Note that by our assumption, the length of each of the intervals in $I_1 - I_{1, new}$ is at least $3N\length_{m+1}$, and at most $O(1)$ curves in $\mathcal{F}_{1, new}$ can cross the $1$-buffers of $\mathcal{G}$ starting at $I_1 - I_{1, new}$.
Thus, by removing these curves if necessary, we obtain a subfamily $\mathcal{F}_{1, new} \subseteq \mathcal{F}_1$ with
$$
|\mathcal{W}(\mathcal{F}_{1, new}) - \mathcal{W}(\mathcal{F}_1)| = O(K/\boldsymbol{\chi}_a)
$$
that overflows $\mathcal{G}$ (see Figure \ref{fig:ShL}).

\begin{figure}[ht]
  \centering
  \resizebox{\linewidth}{!}{
    \def\svgwidth{\columnwidth}
    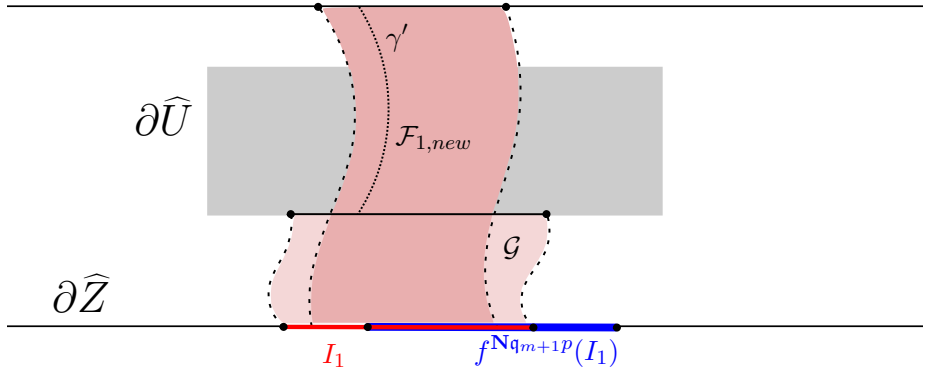

  }
  \caption{The configuration of the families $\mathcal{F}_{1, new}$ and $\mathcal{G}$.}
  \label{fig:ShL}
\end{figure}

Note that $N\mathfrak{q}_{m+1}p \geq N$, arcs in $\mathcal{G}$ do not connect periodic pseudo-Siegel disks by the definition of pulled-off constant. 
Thus, $\mathcal{G}$ consists of homotopically equivalent arcs connecting $\partial \widehat Z$ with some strictly pre-periodic pseudo-Siegel disk $\partial \widehat U$.

Let $\gamma'$ be the part of $\gamma \in \mathcal{F}_{1,new}$ after its first intersection with $\partial \widehat U$.
Consider $\mathcal{F'}:= \{\gamma': \gamma \in \mathcal{F}_{1, new}\}$.
Note that for each arc $\gamma' \in \mathcal{F'}$, its image $f^{N\mathfrak{q}_{m+1}p}(\gamma')$ contains some non-peripheral arc.
Since the iteration $N\mathfrak{q}_{m+1}p$ is bounded, there exists a constant $C_2 = C_2(a)$ so that 
$$
\mathcal{W}(\mathcal{F}') \leq C_2 \mathcal{W}_{arc}(\widehat X_f).
$$
By Proposition \ref{prop:gi}, $\mathcal{W}(\mathcal{F}_1)  \geq \mathcal{W}_{arc}(\widehat X_f
(-N\mathfrak{q}_{m+1}p))/aC_1 \geq \mathcal{W}_{arc}(\widehat X_f
)/4aC_1$.
Thus there exists a constant $C_3 = C_3(a)$ so that
$$
\mathcal{W}(\mathcal{F}') \leq C_3 \mathcal{W}(\mathcal{F}_1).
$$
By the series law, we have
$$
\mathcal{W}(\mathcal{F}_{1,new}) \leq \mathcal{W}(\mathcal{G}) \oplus \mathcal{W}(\mathcal{F}').
$$
Equivalently, we have
$$
1/\mathcal{W}(\mathcal{G}) + 1/ \mathcal{W}(\mathcal{F}') \leq 1/\mathcal{W}(\mathcal{F}_{1,new}).
$$
Using our estimates on $\mathcal{W}(\mathcal{F}_{1,new}), \mathcal{W}(\mathcal{G}), \mathcal{W}(\mathcal{F}')$, we have
$$
1/(\mathcal{W}(\mathcal{F}_1)+O(K/\boldsymbol{\chi}_a)) + 1/(C_3 \mathcal{W}(\mathcal{F}_1)) \leq 1/ (\mathcal{W}(\mathcal{F}_1)+O(K/\boldsymbol{\chi}_a)).
$$
Note that $\mathcal{W}(\mathcal{F}_1) \in [K/C_1, K]$.
Let $\mathcal{W}(\mathcal{F}_1) = cK$ for $c \in [1/C_1, 1]$.
Thus, by cancelling the common term $K$, we obtain
$$
1/(c+O(1/\boldsymbol{\chi}_a)) + 1/(C_3 c) \leq 1/ (c+O(1/\boldsymbol{\chi}_a)).
$$
This is a contradiction as such an inequality cannot hold if $\boldsymbol{\chi}_a \gg 1$.
This concludes the proof of Theorem \ref{thm:calls}.
\end{proof}

\section{Bounds on arc degeneration}\label{sec:utd}
In this section, we shall prove that $\mathcal{W}_{arc}(\widehat X_f)$ is bounded in terms of the pulled-off constant.
By Theorem \ref{thm:pop}, this would immediately imply that $\mathcal{W}_{arc}(\widehat X_f)$ is uniformly bounded for eventually-golden-mean maps on the boundary of a Sierpinski carpet hyperbolic component of disjoint type.
More precisely, we will show
\begin{theorem}\label{thm:utd}
Let $\mathcal{H}$ be a hyperbolic component of disjoint type.
Let $[f] \in \partial \mathcal{H}$ be an eventually-golden-mean map  with pulled-off constant $N = N([f])$.
There exists a constant $\mathbf{K}$ depending only on $N$ and the multipliers of attracting cycles of $f$ with the following properties.
\begin{enumerate}
\item For each Siegel disk $Z_{i,j}$ of $f$, there exists a pseudo-Siegel disk $\widehat Z_{i,j}$ which is a $\mathbf{K}$-quasiconformal closed disk.

\item For each attracting domain $D_{i,j}$, there exists a valuable-attracting domain $\widehat D_{i,j}$ with $\Mod(D_{i,j} - \widehat D_{i,j}) \geq 2\pi/\mathbf{K}$.

\item The pseudo-core surface $\widehat X_f:= \widehat \C - \bigcup \Int(\widehat D_{i,j}) - \bigcup \Int(\widehat Z_{i,j})$ has uniformly bounded arc degeneration
$$
\mathcal{W}_{arc}(\widehat X_f) \leq \mathbf{K}.
$$
\end{enumerate} 
\end{theorem}

Let us outline the strategy of the argument.
As a preparation, we first construct pseudo-Siegel disks as in Theorem \ref{thm:cps} with bounds depending on $\mathcal{W}_{arc}(\widehat X_f)$.
Let $\widehat X_f$ be the pseudo-core surface.
To prove Theorem \ref{thm:utd},  by Theorem \ref{thm:cps}, it suffices to show that $\mathcal{W}_{arc}(\widehat X_f)$ is uniformly bounded.

We will argue by contradiction. 
Suppose $\mathcal{W}_{arc}(\widehat X_f)$ is sufficiently large. Then
\begin{enumerate}[label=\roman*)]
\item we can first localize the arc degeneration (Theorem \ref{thm:ll}) and obtain a small grounded interval $I_1$ with comparable local degeneration 
$$
\mathcal{W}^{+, np}(I) +\mathcal{W}_{\lambda}^{+, per}(I)\succeq \mathcal{W}_{arc}(\widehat X_f). 
$$
\item 
By Property (2)(b) in Theorem \ref{thm:cps}, the peripheral part $\mathcal{W}_{\lambda, m}^{+, per}(I)$ is relatively small.
\item This means $\mathcal{W}_{m}^{+, np}(I)$ is large. 
We will apply the  calibration lemma (see Theorem \ref{thm:calls}), and construct an interval on a deeper level with big local degeneration, which contradicts Property (2)(a) in Theorem \ref{thm:cps}.
\end{enumerate}

\subsection{Choosing the constants}\label{subsec:constants}
There are many constants in the proof.
We summarize their relations and the order we choose them here.
\begin{itemize}
\item $\mathbf{a}$ is the constant in the localization lemma (Theorem \ref{thm:ll}), and we assume $\mathbf{a} > N$;
\item $\boldsymbol{\chi}$ is chosen so that it satisfies the calibration lemma (with constant $a = \mathbf{a}$)(Theorem \ref{thm:calls});
\item $\boldsymbol{\lambda}$ is chosen so that $\boldsymbol{\lambda} \gg \boldsymbol{\Lambda}, \boldsymbol{\chi}$, where $ \boldsymbol{\Lambda}$ is the constant in the localization lemma (Theorem \ref{thm:ll}).
\end{itemize}

Let $[f] \in \partial \mathcal{H}$ be an eventually-golden-mean map.
Let $\theta_1,..., \theta_l$ be the list of rotation numbers for Siegel disks of $f$.
\begin{itemize}
\item $\boldsymbol{\epsilon}$ is chosen so that 
$$
\boldsymbol{\epsilon} \ll 1/\boldsymbol{\chi}\mathbf{a}.
$$
\end{itemize}
We also assume that all psuedo-Siegel disks are $T$-stable where 
$$
T \gg 4\mathbf{a}N \max\{p_i\}
$$ 
where $p_i$ is the period of Siegel disks (see Remark \ref{rmk:stablitiy}).
We remark that all the constants above depend only on the pulled-off constant $N$.

\subsection{Uniform geometric control}\label{subsec:ugc}
\begin{lem}\label{lem:tb}
There exists a constant $\mathbf{K}$ depending on the pulled-off constant $N$ so that $\mathcal{W}_{arc}(\widehat X_f)\leq \mathbf{K}$.
\end{lem}
\begin{proof}
Choose $\mathbf{K} \gg 1$ so that it is much bigger than the threshold in Theorem \ref{thm:ll} (with constant $\epsilon=\boldsymbol{\epsilon}, \lambda=\boldsymbol{\lambda}$ and $N$), Theorem \ref{thm:calls} (with constant $a= \mathbf{a}$).

Suppose by contradiction that $\mathcal{W}_{arc}(\widehat X_f) = K \geq \mathbf{K}$.
By the localization lemma (Theorem \ref{thm:ll}), there exist a pseudo-Siegel disk $\widehat Z$ and a grounded interval $I$ of $\widehat Z^m$ with 
\begin{itemize}
\item $\vert I \vert \leq \boldsymbol{\epsilon} \ll 1/\boldsymbol{\chi}\mathbf{a}$; and
\item $\mathcal{W}_m^{+, np}(I) +\mathcal{W}_{\boldsymbol{\lambda},m }^{+, per}(I)  \geq 2K/\mathbf{a}$
\end{itemize}
Suppose that $\mathfrak{l}_{m+1} < \vert I \vert \leq \mathfrak{l}_m$.
By Property (2)(b) of Theorem \ref{thm:cps}, the peripheral part 
$$
\mathcal{W}_{\boldsymbol{\lambda}, m}^{+, per}(I) = O(\sqrt{\length_{m} K} + 1) \ll K/\mathbf{a},
$$
for all sufficiently large $K$.

Thus, $\mathcal{W}_{m}^{+, np}(I) \geq K/\mathbf{a}$.
 By Property (2)(a) for level $m$ of Theorem \ref{thm:cps}, $K/\mathbf{a}\leq \mathcal{W}^{+, np}_{m}(I) +\mathcal{W}_{\boldsymbol{\lambda}, m}^{+, per}(I) = O(K\mathfrak{l}_m + 1)$, so $\mathfrak{l}_m \geq \frac{1}{4\mathbf{a}}$.
Therefore, we can apply Theorem \ref{thm:calls}, and obtain a grounded interval $J$ with $\vert J \vert \leq \length_{m+1}$ with 
$$
\mathcal{W}_{m+1}^{+, np}(J)\geq K/\boldsymbol{\chi}\mathbf{a}.
$$

By Property (2)(a) for level $m+1$ of Theorem \ref{thm:cps}, we have
$$
\mathcal{W}^{+, np}_{m+1}(J) = O(\length_{m+1}K + 1).
$$
Since $\mathfrak{l}_{m+1} \leq \vert I \vert \leq \boldsymbol{\epsilon} \ll 1/\boldsymbol{\chi}\mathbf{a}$, increase $\mathbf{K}$ if necessary, we have
$$
\mathcal{W}^{+, np}_{m+1}(J) \ll K/\boldsymbol{\chi}\mathbf{a},
$$
which is a contradiction.
The lemma now follows.
\end{proof}

\begin{proof}[Proof of Theorem \ref{thm:utd}]
The theorem follows by combining Theorem \ref{thm:cps}, Lemma \ref{lem:kad} and Lemma \ref{lem:tb}.
\end{proof}

\section{Dynamics on limiting trees and bounds on loop degeneration}\label{sec:uthind}
In this section, we shall prove the following theorem giving uniform bounds of loop degeneration for eventually-golden-mean maps.
\begin{theorem}\label{thm:uthind}
Let $\mathcal{H}$ be a hyperbolic component of disjoint type.
Let $[f] \in \partial \mathcal{H}$ be an eventually-golden-mean map  with (Siegel) pulled-off constant $N = N_{Siegel}([f])$.
Let 
$$
\widehat X_f:= \widehat \C - \bigcup \Int(\widehat D_{i,j}) - \bigcup \Int(\widehat Z_{i,j})
$$ 
be the Riemann surface as in Theorem \ref{thm:utd}.
There exists a constant $\mathbf{K}_{loop}$ depending only on $N$ and the multipliers of the attracting cycle of $[f]$ so that
$$
\mathcal{W}_{loop}(\widehat X_f) \leq \mathbf{K}_{loop}.
$$
\end{theorem}

\subsection{Limiting maps on trees}
Recall that the marked hyperbolic component $\mathcal{H}$ are parameterized by the $2d-2$ multipliers $\rho_1, ..., \rho_{2d-2}$:
$$
\mathcal{H} \cong \D_1 \times ... \times \D_{2d-2}.
$$
Fixing $a_1, ..., a_k \in \D$ and a constant $N$, and consider the slice
\begin{align*}
\mathfrak{A} := \{&[f] \in \partial \mathcal{H}:  [f] \text{ is an eventually-golden-mean map, with }\\ 
&\rho_i = a_i , i=1,..., k, \vert \rho_{i} \vert = 1, i = k+1,..., 2d-2, N_{Siegel}([f]) \leq N\}.
\end{align*}
To prove Theorem \ref{thm:uthind}, it suffices to show that there exists a constant $\mathbf{K}_{loop}$ with $\mathcal{W}_{loop}(\widehat X_f) \leq \mathbf{K}_{loop}$ for any map $[f] \in \mathfrak{A}$.

The proof is by contradiction.
We first show that, after passing to a subsequence, any sequence $[f_n] \in \mathfrak{A}$ converges to a limiting map on a tree of Riemann spheres.
If there is no such constant $\mathbf{K}_{loop}$, then the limiting tree is non-trivial.
The dynamics on the tree is recorded by a Markov matrix $M$ and a degree matrix $D$.
We show that there exists a non-negative vector $\vec{v} \neq \vec{0}$ with $M\vec v = D \vec v$.
As matrices with non-negative entries, we show that $D^{-1}M$ is no bigger than the Thurston matrix for $f_n$ for sufficiently large $n$.
So the spectral radius of the Thurston's matrix is greater or equal to $1$, giving a contradiction.

Following the notations in \cite{Luo21b}, we define
\begin{defn}\label{defn:trs}
A {\em tree of Riemann spheres} $(\RT, \widehat\C^\RV)$ consists of a finite tree $\RT$ with vertex set $\RV$, a disjoint union of Riemann spheres $\widehat \C^\RV := \bigcup_{a\in \RV}\widehat \C_{a}$, together with markings $\xi_a: T_a\RT \xhookrightarrow{} \widehat\C_a$ for $a\in \RV$.

The image $\Xi_a :=\xi_a(T_a\RT)$ is called the {\em singular set} at $a$, and $\Xi = \bigcup_{a\in \RV}\Xi_a$.

A {\em rational map} $(F, R)$ on $(\RT, \widehat\C^\RV)$ is a map 
$$
F: (\RT, \RV) \longrightarrow (\RT, \RV) \text{ that is injective on edges},
$$ 
and a union of maps $R:= \bigcup_{a\in \RV} R_a$ so that
\begin{itemize}
\item $R_a: \widehat\C_a \longrightarrow \widehat \C_{F(a)}$ is a rational map;
\item $R_a \circ \xi_a = \xi_{F(a)} \circ DF_a$.
\end{itemize}
It is said to have degree $d$ if $R$ has $2d-2$ critical points in $\widehat\C^\RV-\Xi$.

A sequence $f_n$ of degree $d$ rational maps is said to {\em converge} to $(F, R)$ on $(\RT, \widehat\C^\RV)$ if there exist rescalings $A_{a,n} \in \PSL_2(\C)$ for $a \in \RV$ such that
\begin{itemize}
	\item $A_{F(a),n}^{-1} \circ f_n \circ A_{a,n}(z) \to R_a(z)$ compactly on $\widehat \C_a- \Xi_a$;
	\item $A_{b,n}^{-1} \circ A_{a,n}(z)$ converges to the constant map $\xi_a(v)$, where $v\in T_a\RT$ is the tangent vector in the direction of $b$.
\end{itemize}
\end{defn}

\begin{theorem}\label{thm:gc}
Let $[f_n] \in \mathfrak{A}$.
Then after passing to a subsequence $[f_n]$ converges to a degree $d$ rational map $(F, R)$ on $(\RT, \widehat\C^\RV)$.

Moreover, $[f_n]$ converges in $\M_{d, \fm}$ if and only if $\RT$ is trivial, i.e., $\RT$ consists of a single vertex.
\end{theorem}

We remark that a similar result is proved for {\em quasi-post critically finite} degenerations of arbitrary rational maps in \cite{Luo21b}.
For quasi-post critically finite degeneration, the orbit of the critical points is controlled uniformly throughout the sequence.
In our setting, the critical orbits are controlled in those valuable-attracting domains, as the multipliers stay constant, and Theorem \ref{thm:utd} gives a uniform bound for pseudo-Siegel disks.
More precisely, we use these uniform bounds crucially in two places.
\begin{enumerate}
    \item We use the fact that valuable-attracting domains and pseudo-Siegel disks are uniform quasi disks to show that their rescaling limits converge (Definition \ref{defn:res} and Lemma \ref{lem:cqd}).
    \item We use the uniform bound on arc degeneration to control the holes for the rescaling limit (Proposition \ref{prop:dss}), which allow us to construct Thurston obstructions for large $n$ (Proposition \ref{prop:srg1}).
\end{enumerate}
With these modifications, the proof is similar to the quasi-post critically finite case as in \cite{Luo21b}.

We also remark that the same proof of Theorem \ref{thm:gc} also gives the following
\begin{theorem}\label{thm:gc2}
    Let $[f_n] \in \mathcal{H}$ be a seuqnece of eventually-golden-mean maps with uniformly bounded $\mathcal{W}_{arc}(\widehat X_f)$.
    Then after passing to a subsequence $[f_n]$ converges to a degree $d$ rational map $(F, R)$ on $(\RT, \widehat\C^\RV)$.

    Moreover, $[f_n]$ converges in $\M_{d, \fm}$ if and only if $\RT$ is trivial, i.e., $\RT$ consists of a single vertex.
\end{theorem}

Since maps in the slice $\mathfrak{A}$ are marked, we use $\mathcal{V}$ to denote the collection of valuable-attracting domains and open pseudo-Siegel disks.

More precisely, this means that if $U \in \mathcal{V}$ and $[f] \in \mathfrak{A}$, then $U(f)$ is either a valuable-attracting component or the interior of a pseudo-Siegel disk for $f$.
Note that we have an induced dynamics
$$
f_*: \mathcal{V} \longrightarrow \mathcal{V}.
$$

\subsection*{Construction of the rescaling}
Let us fix a sequence $[f_n] \in \mathfrak{A}$.
We define the rescaling for $U\in \mathcal{V}$ as follows.
\begin{defn}\label{defn:res}
Let $U\in \mathcal{V}$.
Let $\alpha_n \in U(f_n)$ be the corresponding non-repelling periodic point.
A sequence $A_{U, n} \in \PSL_2(\C)$ is called a {\em rescaling} for $U$ if 
\begin{itemize}
\item $A_{U, n}(0) = \alpha_n \in U(f_n)$;
\item $A_{U, n}(1) \in \partial U(f_n)$;
\item $A_{U, n}(\infty) \in \widehat \C - \overline{U(f_n)}$.
\end{itemize}
\end{defn}

\begin{lem}\label{lem:cqd}
    After passing to a subsequence, $A_{U,n}^{-1}(U(f_n))$ converges in Hausdorff topology to some quasiconformal disk.
\end{lem}
\begin{proof}
    Since $U(f_n)$ are uniformly quasiconformal disks by Theorem \ref{thm:utd}, there is a sequence $\Psi_n: \widehat\C \longrightarrow \widehat\C$ of uniformly quasiconformal maps so that $\Psi_n(\D) = U(f_n)$, normalized so that $\Psi_n(0) = A_{U, n}(0), \Psi_n(1) = A_{U,n}(1)$ and $\Psi_n(\infty) = A_{U,n}(\infty)$.
    Then $A_{U,n}^{-1} \circ \Psi_n$ is uniformly quasiconformal and fixes $0, 1, \infty$.
    Therefore, after passing to a subsequence, $A_{U,n}^{-1} \circ \Psi_n$ converges to a quasiconformal map $\Psi$.
    Thus, $A_{U,n}^{-1}(U(f_n))$ converges to the quasiconformal disk $\Phi(\D)$.
\end{proof}

The following lemma follows from the same argument as in \cite[Lemma 4.3]{Luo21b}.
\begin{lem}
If $A_{U, n}, B_{U, n}$ are two rescalings for $U \in \mathcal{V}$, then the sequence $ B_{U, n}^{-1} \circ A_{U, n}$ is bounded.

Equivalently, if we identify the hyperbolic 3-space $\Hyp^3$ as the unit ball and $\PSL_2(\C) \cong \Isom(\Hyp^3)$, then
$$
d_{\Hyp^3}(A_{U, n}({\bf 0}), B_{U, n}({\bf 0})) \text{ is bounded,}
$$
where ${\bf 0} \in \Hyp^3$ is the center of the unit ball.
\end{lem}

Let us now fix rescaling $A_{U,n}$ for $U \in \mathcal{V}$.
\begin{lem}\label{lem:UC}
Let $U\in \mathcal{V}$.
Let $A_{U, n}, A_{f_*(U), n}$ be rescalings for $U$ and $f(U)$.
Then after passing to a subsequence,
$$
A_{f_*(U), n}^{-1} \circ f_n \circ A_{U,n}
$$ 
converges (away from finitely many points) to a non-constant map.
\end{lem}
\begin{proof}
Note that $A_{f_*(U), n}^{-1} \circ f_n \circ A_{U,n}$ is a sequence of rational maps, so after passing to a subsequent, it converges (away from finitely many points) to a rational map with degree $\leq d$.
By Lemma \ref{lem:cqd}, after passing to a subsequence, $A_{U, n}^{-1}(U(f_n))$ and $A_{f_*(U), n}^{-1}(f_*U(f_n))$ converge to $U_\infty$ and $(f_*U)_\infty$.
By \cite[Theorem 5.6]{McM94}, $A_{f_*(U), n}^{-1} \circ f_n \circ A_{U,n}$ cannot converge to a constant map on $U_\infty$, and the lemma follows.
\end{proof}

After passing to a subsequence, we may assume for different $U, V \in \mathcal{V}$, $A_{U,n}^{-1}\circ A_{V,n}$ converges to either
\begin{itemize}
\item a M\"obius transformation; or
\item a constant map.
\end{itemize}
This defines an equivalence relation on $\mathcal{V}$: $U\sim V$ if and only if $A_U^{-1}\circ A_V$ converges to a M\"obius transformation.
It follows from the same proof of \cite[Lemma 4.7]{Luo21b} that if $U \sim V$, then $f_*(U) \sim f_*(V)$.

Let $\Pi:= \mathcal{V}/\sim$ be the set of equivalence classes.
By the previous remark, we have an induced map
$$
F: \Pi \longrightarrow \Pi.
$$

\begin{defn}\label{defn:rld}
For each equivalence class $a\in \Pi$, we choose a representative $U \in a$, and define the {\em rescaling} at $a$ by
$$
A_{a,n}:= A_{U,n} \in \PSL_2(\C).
$$
\end{defn}

\subsection*{Construction of the tree of Riemann spheres $(\RT, \widehat\C^\RV)$}
Recall that we identify the hyperbolic 3-space $\Hyp^3$ as the unit ball in $\R^3$ and $\widehat\C$ as the conformal boundary of $\Hyp^3$.
Denote ${\bf 0} \in \Hyp^3$ as the center of the unit ball.
We denote $x_{a,n} = A_{a, n}({\bf 0}) \in \Hyp^3$ and $\Pi_n = \{x_{a, n} \in \Hyp^3: a\in \Pi\}$.
Note that by our construction, $A_{a,n}^{-1} \circ A_{b,n} \to \infty$ in $\PSL_2(\C)$, so
$$
d_{\Hyp^3}(x_{a,n}, x_{b,n}) \to \infty \text{ if } a\neq b.
$$

Thus, the hyperbolic polyhedra $\chull( \Pi_n)$ is degenerating.
One can construct a sequence of trees $\RT_n$ as the spine for $\chull( \Pi_n)$ capturing the degenerations of the polyhedra.
We summarize some properties for $\RT_n$ and refer the readers to \cite[\S 3 and \S 4]{Luo21b} for more details.
\begin{itemize}
    \item The vertex set $\RV_n$ for $\RT_n$ is a finite set consisting of $\Pi_n$ and branched points of $\RT_n$;
    \item Each edge of $\RT_n$ is a hyperbolic geodesic segment whose length goes to $\infty$ as $n\to\infty$;
    \item There exists a uniform lower bound on the angle between two adjacent edges of $\RT_n$;
    \item The finite tree $\RT_n\subseteq \chull( \Pi_n)$ and any point $x\in \chull( \Pi_n)$ is within uniform bounded distance from $\RT_n$.
\end{itemize}
Since there are a bounded number of endpoints for $\RT_n$, after passing to a subsequence, we assume that $\RT_n$ are isomorphic as finite trees.
Denote this isomorphic class of finite trees as $(\RT, \RV)$, and we have a marking for each $n$ 
$$
\Psi_n: (\RT, \RV) \longrightarrow (\RT_n, \RV_n).
$$

We remark that $\Pi \subseteq \RV$, and any point $a \in \RV - \Pi$ is a branch point.
We extend the definition of rescalings for $\RV - \Pi$.
Let $a\in \RV-\Pi$, a sequence $A_{a,n} \in \PSL_2(\C) \cong \Isom(\Hyp^3)$ is defined to be a {\em resacling} at $a$ if 
$$
A_{a,n}({\bf 0}) = \Psi_n(a).
$$
Note that different choices of rescalings at $a$ are differed by pre-composing with a rotation that fixes ${\bf 0}$, which form a compact group.

Denote $\overline{\B} = \Hyp^3 \cup \widehat \C$, and $\overline{\B}_a = \Hyp^3_a \cup \widehat \C_a$ for $a\in \RV$.
We define a sequence $z_n \in \overline{\B}$ converges to $z\in \overline{\B}_a$ {\em in $a$-coordinate} or {\em with respect to the rescaling at $a$}, denoted by $z_n \to_a z$ or $z = \lim_a z_n$ if
$$
\lim_{n\to\infty} A_{a,n}^{-1}(z_n) = z.
$$

By construction, $A_{a,n}^{-1} \circ A_{b,n}$ converges to a constant map $x_b$ for $a\neq b \in \RV$.
Thus, we can associate the point $x_b \in \widehat \C_a$ to $b$.
It is interpreted that the Riemann sphere $\widehat \C_b$ converges to $x_b$ in the rescaling coordinate $\widehat \C_a$.
We denote 
$$
\Xi_a:= \bigcup_{b\neq a} x_b \subseteq \widehat\C_a
$$
as the {\em singular set} at $a$ and $\Xi := \bigcup_{a\in \RV} \Xi_a \subseteq \widehat \C^\RV$ as the {\em singular set}.

It follows from the construction that $\Psi_n(b) \to_a x_b \in \widehat\C_a$ if $a \neq b \in \RV$.
Since the angle 
$\angle \Psi_n(a) \Psi_n(b) \Psi_n(c)$ 
is uniformly bounded below from $0$ for any distinct triple $a, b, c\in \RV$,
the singular set $\Xi_a$ is in correspondence with the tangent space $T_a \RT$ at $a$.
We denote this correspondence by
$$
\xi_a: T_a\RT \longrightarrow \Xi_a.
$$

\subsection*{Construction of the rescaling rational maps}
The following lemma allows us to construct rescaling rational maps.
\begin{lem}\label{lem: extmap}
Let $a\in \RV$, after passing to a subsequence, there exists a unique $b\in \RV$ so that
$$
A_{b,n}^{-1} \circ f_n \circ A_{a,n}
$$
converges to a rational map $R_a = R_{a\to b}$ of degree at least $1$.

Moreover, the holes of $R_a$ are contained in $\Xi_a$.
\end{lem}
\begin{proof}
    If $a\in\Pi \subseteq \RV$, i.e., if $a$ is represented by some $U \in \mathcal{V}$, then Lemma \ref{lem: extmap} follows immediately from Lemma \ref{lem:UC}.

    Otherwise, the proof is the same as \cite[Lemma 4.12]{Luo21b}.
\end{proof}

The above lemma allows us to define $F: \RV \longrightarrow \RV$ by setting $F(a)$ as the unique vertex $b$ in Lemma \ref{lem: extmap} extending the map $F: \Pi\subseteq \RV \longrightarrow \Pi\subseteq \RV$.

We define the map $F: \RT \longrightarrow \RT$ by extending continuously on any edge $[a,b]$ to the arc $[F(a), F(b)]$.

\subsection*{Modulus estimate for dynamics on edges}
Let $E = [a,b]$ be an edge of $\RT$.
In the following, we shall associate it with annuli $\mathcal{A}_{E,n}$ and define the local degree of $E$.

Let $x_b \in \Xi_a \subseteq \widehat\C_a$ and $x_a\in \Xi_b \subseteq \widehat\C_b$ be the points associated to $b$ and $a$ respectively.
Choose a small closed curve $C_a \subseteq \widehat\C_a$ around $x_b$.
We assume $C_a$ bounds no holes nor critical points of $R_a$ other than possibly $x_b$ and $R_a: C_a \longrightarrow R_a(C_a)$ is a covering map of degree $\deg_{\xi_a(v_b)}(R_a)$.
Similarly, we define $C_b \subseteq \widehat\C_b$ around $x_a$.

Since $A_{F(a), n}^{-1} \circ f_n \circ A_{a, n}$ converges uniformly to $R_a$ in a neighborhood of $C_a$, we can find a sequence of closed curves $C_{a,n}$ such that
\begin{itemize}
\item $f_n: C_{a,n} \longrightarrow f_n(C_{a,n})$ is a covering of degree $\deg_{\xi_a(v_b)}(R_a)$;
\item $A_{a,n}^{-1}(C_{a,n})$ converges in Hausdorff topology to $C_a$.
\end{itemize}
Similarly, let $C_{b,n}$ be the corresponding closed curves for $C_b$.
Note that $C_{a, n}$ and $C_{b,n}$ are disjoint for sufficiently large $n$ and bounds an annulus $\mathcal{A}_{E,n}\subseteq \widehat\C$.
We call $\mathcal{A}_{E,n}$ a sequence of annuli associated to $E$.

We claim that $\mathcal{A}_{E,n}$ contains no critical points of $f_n$ for sufficiently large $n$.
Indeed, otherwise, after passing to a subsequence, let $c_n \in \mathcal{A}_{E,n}$, and let $c\in \Pi$ corresponds to the sequence $(c_n)$.
Consider the projection 
$$
\proj_{[\Psi_n(a), \Psi_n(b)]}(\Psi_n(c))
$$ 
of $\Psi_n(c) \in \Hyp^3$ onto the geodesics $[\Psi_n(a), \Psi_n(b)]$.
Since we assume $C_a$ bounds no holes nor critical points other than possibly $\xi_a(v_b)$ and similarly for $C_b$, 
$$
d_{\Hyp^3}(\proj_{[\Psi_n(a), \Psi_n(b)]}(\Psi_n(c)), \partial [\Psi_n(a), \Psi_n(b)]) \to \infty.
$$
This contradicts that $[a, b]$ is an edge of $\RT$.

Therefore, $f_n: \mathcal{A}_{E,n} \longrightarrow f_n(\mathcal{A}_{E,n})$ is a covering map.
Thus, $\deg_{x_b}(R_a)= \deg_{x_a}(R_b)$, and we define the local degree at $E$
$$
\delta(E) := \deg_{x_b}(R_a)= \deg_{x_a}(R_b),
$$
as the degree of this covering map.

The proof of following modulus estimate can be found in \cite[Proposition 4.15]{Luo21b}.
\begin{prop}\label{prop: moduluse}
Let $\mathcal{A}_{E,n}$ associated to an edge $E = [a,b]$ of $\RT$.
There exists a constant $K$ such that
the modulus 
$$
|\Mod(\mathcal{A}_{E,n})-\frac{d_{\Hyp^3}(\Psi_n(a), \Psi_n(b))}{2\pi}| \leq K,
$$
and 
$$
|\Mod(f_n(\mathcal{A}_{E,n}))-\frac{d_{\Hyp^3}(\Psi_n(F(a)), \Psi_n(F(b)))}{2\pi}| \leq K.
$$
\end{prop}

The above modulus estimate gives the following corollaries.
\begin{cor}\label{cor:inje}
Let $E = [a,b]$ be an edge of $\RT$. Then
$$
d_{\Hyp^3}(\Psi_n(F(a)), \Psi_n(F(b))) = \delta(E) d_{\Hyp^3}(\Psi_n(a), \Psi_n(b)) + O(1).
$$
\end{cor}

Corollary \ref{cor:inje} implies that the map $F$ is injective on edges.
Thus, we can define the tangent map $DF_a: T_a \RT \longrightarrow T_{F(a)}\RT$.
Since $A_{b,n}^{-1} \circ f_n \circ A_{a,n}$ converges to a rational maps away from the singular set, we also have the following compatibility property.
\begin{cor}\label{cor:comtm}
Let $a \in \RT$. Then
$$
R_a \circ \xi_a = \xi_{F(a)} \circ DF_a.
$$
\end{cor}

\begin{proof}[Proof of Theorem \ref{thm:gc}]
By our construction, $f_n$ converges to $(R, F)$ and $(R, F)$ is a degree $d$ rational map on $(\RT, \widehat\C^\RV)$.

For the moreover part, we note that if $[f_n]$ converges in $\M_{d, \fm}$, then all rescaling limits $A_{U, n}$ are equivalent for $U \in \mathcal{V}$, so $\RT$ consists of a single vertex.
On the other hand, if all rescaling limits $A_{U, n}$ are equivalent, then $A_{U,n}^{-1} \circ f_n \circ A_{U,n}$ converges to a degree $d$ rational map, so $[f_n]$ converges in $\M_{d, \fm}$.
\end{proof}

\subsection*{Matrix encoding}
Index the set of edges of $\RT$ by
$\{E_1,..., E_k\}$.
We define the following two matrices to encode the dynamics $F: \RT\longrightarrow \RT$.
\begin{itemize}
\item (Markov matrix): $M_{i,j} = \begin{cases} 1 &\mbox{if } E_i \subseteq F(E_j) \\ 
0 & \mbox{otherwise } \end{cases}$
\item (Degree matrix): $D_{i,j} = \begin{cases} \delta(E_i) &\mbox{if } i = j \\ 
0 & \mbox{otherwise } \end{cases}$
\end{itemize}
\begin{prop}\label{prop:mdto}
Let $M$ and $D$ be the Markov matrix and the degree matrix respectively.
If $\RT$ is not trivial, then there exists a non-negative vector $\vec{v} \neq \vec{0}$ so that
$$
M\vec v = D \vec v.
$$
\end{prop}
\begin{proof}
Let $\vec{v}_n =\begin{bmatrix}
l(\Psi_n(E_1))\\
\vdots\\
l(\Psi_n(E_k))
\end{bmatrix}$, where $l(\Psi_n(E_i))$ is the hyperbolic length of the edge $\Psi_n(E_i)$.
Let $\rho_n = \max_{i=1,..., k} l_{\Hyp^3} (\Psi_n(E_i)) \to \infty$
After passing to a subsequence, we assume the limit $\vec{v} = \lim_{n\to\infty} \vec{v}_n/\rho_n$ exists.
Then $\vec{v}$ is non-negative and $\vec{v} \neq \vec{0}$.

If suffices to check $M\vec v = D \vec v$.
If $a, b \in \RV$ are connected by a sequence of edges $E_{i_1} \cup E_{i_2} \cup... E_{i_m}$, 
since the angles between different incident edges at a vertex of $\RT_n$ are uniformly bounded below from $0$, we have
$$
d_{\Hyp^3}(\Psi_n(a), \Psi_n(b)) = \sum_{j=1}^m l_{\Hyp^3}(\Psi_n(E_{i_j})) + O(1).
$$ 
Thus, by Corollary \ref{cor:inje}, if $F(E_i) = E_{i_1} \cup E_{i_2} \cup... E_{i_m}$, then 
$$
\delta(E_i) l_{\Hyp^3} (\Psi_n(E_i)) = \sum_{j=1}^m l_{\Hyp^3}(\Psi_n(E_{i_j})) + O(1).
$$
Dividing both sides by $\rho_n$ and taking limits, we conclude the result.
\end{proof}

\subsection{Thurston's obstruction}
Let $f: \widehat\C \longrightarrow \widehat\C$ be a rational map with post-critical set $P_f$.
Let $P_f \subseteq U$ be a forward invariant set, i.e., $f(U) \subseteq U$. 
A simple closed curve $\sigma$ on $\widehat \C - U$ is {\em essential} if it does not bound a disk in $\widehat C - U$, and a curve is {\em peripheral} if it encloses a single point of $U$.
Two simple curves are {\em parallel} if they are homotopic in $\C- U$.

A {\em curve system} $\Sigma = \{\sigma_i\}$ in $\widehat\C - U$ is a finite nonempty collection of disjoint simple closed curves, each essential and non-peripheral, and no two parallel.
A curve system determines a {\em transition matrix} $A(\Sigma) : \R^\Sigma \longrightarrow \R^{\Sigma}$ by the formula
$$
A_{\sigma\tau} = \sum_\alpha \frac{1}{\deg(f: \alpha \to \tau)}
$$
where the sum is taken over components $\alpha$ of $f^{-1}(\tau)$ isotopic to $\sigma$.

Let $\lambda(\Sigma)\geq 0$ denote the spectral radius of $M(\Sigma)$.
Since $A(\Sigma) \geq 0$, the Perron-Frobenius theorem guarantees that $\lambda(\Sigma)$ is an eigenvalue for $A(\Sigma)$ with a non-negative eigenvector.

The same proof of \cite[Theorem B.4]{McM94} gives
\begin{prop}\label{prop:mtc}
Let $[f] \in \mathfrak{A}$ be an eventually-golden-mean Siegel map.
Let $\mathcal{U}_f$ be the union of Siegel disks and valuable-attracting domains.
Let $\Sigma$ be a curve system in $\widehat \C - \mathcal{U}_f$, then $\lambda(\Sigma) < 1$. 
\end{prop}

\subsection*{Curve system for edges of $\RT$}
Let $\mathcal{U}_n$ be the union of Siegel disks and valuable-attracting domains for $f_n$, and let $\widehat{\mathcal{U}}_n$ be the union of open pseudo-Siegel disks and valuable-attracting domains for $f_n$.
Note that $\mathcal{U}_n$ is forward invariant, and $\mathcal{U}_n \subseteq \widehat{\mathcal{U}}_n$.
Let $\mathcal{A}_{E_,n}$ be the annulus associated to an edge $E$.

After passing to a subsequence, we may assume that for any $a\in \RV$, and any pseudo-Siegel disk or a valuable-attracting domain $U(f_n)$, the limit $A_{a,n}^{-1}(U(f_n))$ exists.
Since $U(f_n)$ are uniformly quasiconformal disks, the limit is either a point or a quasiconformal disk.
We say $U(f_n)$ is trivial for $a$ if the limit is a point, and non-trivial otherwise.

The following lemma is the crucial step that we use the geometric control of valuable-attracting domains and pseudo-Siegel disks.
\begin{prop}\label{prop:dss}
Let $a \in \RV$. 
The singular set $\Xi_a$ is disjoint from the closure of any non-trivial limits of pseudo-Siegel disks and valuable-attracting domains in $\widehat\C_a$.
\end{prop}
\begin{proof}
Let $U = \lim A_{a,n}^{-1}(U(f_n)) \subseteq \widehat\C_a$ be a non-trivial limit.
It is easy to see that the singular set is disjoint from $U$.
Now suppose $x \in \Xi_a \cap \partial U$.
Since $x$ is a singular point, there exists a sequence of pseudo-Siegel disk or a valuable-attracting domain $W(f_n)$ with $\lim A_{a,n}^{-1}(W(f_n)) = x$.
Without loss of generality, we assume $a$ is fixed.
Since there is a critical point on the boundary of U, $R_a$ has degree at least $2$. Therefore, $\widehat\C_a$ contains at least two non-trivial limit of pseudo-Siegel disks or a valuable-attracting domains.
Consider a small arc $\gamma \subseteq \widehat\C_a - U$ with $\partial \gamma \subseteq \partial U$ that encloses $x$.
Then the corresponding arc with end points in $\partial U(f_n)$ for $f_n$ is non-peripheral and its extremal width goes to infinity as $n\to\infty$.
This is a contradiction to Theorem \ref{thm:utd}.
\end{proof}

As a corollary, we have
\begin{cor}\label{cor:dj}
For sufficiently large $n$, the core curve $\sigma_{E,n}$ of $\mathcal{A}_{E_,n}$ is a curve in $\widehat \C - \widehat{\mathcal{U}}_n \subseteq \widehat \C - \mathcal{U}_n$.
\end{cor}
Let 
$$
\Sigma_n =\{\sigma_{E,n}: E \text{ is an edge of } \RT\}.
$$
Then $\Sigma_n$ is a curve system in $\widehat\C - \mathcal{U}_n$ for all sufficiently large $n$.

\begin{lem}\label{lem:lto}
If $E \subseteq F(E')$, then for sufficiently large $n$, $\sigma_{E, n}$ has a lift of degree $\delta(E')$ homotopic to $\sigma_{E',n}$ in $\widehat\C-\mathcal{U}_n$.
\end{lem}
\begin{proof}
    Let $\mathcal{A}_{E,n}$ and $\mathcal{A}_{E',n}$ be annuli associated to $E$ and $E'$ respectively.
Modify the boundaries of $\mathcal{A}_{E,n}$ if necessary, we may assume $\mathcal{A}_{E,n} \subseteq f_n(\mathcal{A}_{E',n})$, where the inclusion induces an isomorphism on the fundamental group.
So for sufficiently large $n$, there exists an essential simple closed curve $\gamma_n \subseteq \mathcal{A}_{E,n}$ that has a degree $\delta(E)$ lift $\gamma_n' \subseteq \mathcal{A}_{E',n}$.
By Proposition \ref{prop:dss}, these $\gamma_n$ and $\gamma_n'$ are homotopic to core curves $\sigma_{E, n}$ and $\sigma_{E', n}$ in $\widehat\C-\mathcal{U}_n$, and the lemma follows.
\end{proof}

Combining Proposition \ref{prop:mdto} and Lemma \ref{lem:lto}, we have
\begin{prop}\label{prop:srg1}
Let $\Sigma_n$ be the curve system associated with the edges of $\RT$ in $\widehat\C - \mathcal{U}_n$.
If $\RT$ is not trivial, i.e. it contains more than one vertex, then for sufficiently large $n$, the spectral radius $\lambda(\Sigma_n) \geq 1$.
\end{prop}

\begin{proof}[Proof of Theorem \ref{thm:uthind}]
Suppose for contradiction that Theorem \ref{thm:uthind} does not hold.
Then there exists a sequence $[f_n] \in \mathfrak{A}$ with $\mathcal{W}_{loop}(X_{f_n}) \to \infty$.
After passing to a subsequence, $[f_n]$ converges to a degree $d$ rational map $(F,R)$ on $(\RT, \widehat\C^\RV)$ by Theorem \ref{thm:gc}.
Since $[f_n] \to \infty$ in $\mathcal{M}_{d,\fm}$, the tree $\RT$ is not trivial.
By Proposition \ref{prop:srg1}, the curve system $\Sigma_n$ has spectral radius $\lambda(\Sigma_n) \geq 1$ for all sufficiently large $n$.
This is a contradiction to Proposition \ref{prop:mtc}.
\end{proof}

The same proof also gives Theorem \ref{thm:ubd} and Theorem \ref{thm:A}.
\begin{proof}[Proof of Theorem \ref{thm:ubd}]
Since the arc degeneration is uniformly bounded, $f_n$ converges to a degree $d$ rational map $(F, R)$ on $(\RT, \widehat\C^\RV)$ by Theorem \ref{thm:gc2}.
Suppose for contradiction that $[f_n]$ diverges, then the tree $\RT$ is non-trivial.
Therefore, $\mathcal{W}_{loop}(X_{f_n}) \to \infty$ which is a contradiction.
Since the degeneration is uniformly bounded, the psuedo-Siegel disks and valuable-attracting domains do not collide.
So $[f]$ has $2d-2$ non-repelling cycles.
\end{proof}
\begin{proof}[Proof of Theorem \ref{thm:A}]
It suffices to show that the marked hyperbolic component $\mathcal{H}$ is bounded.
Note that $\mathcal{H}$ is identified with $\D_1 \times ... \times \D_{2d-2}$.
It suffices to realize the multiplier $(\lambda_1,..., \lambda_{2d-2}) \in \partial (\D_1 \times ... \times \D_{2d-2})$ by a map $[f] \in \partial \mathcal{H}$.
Let $[f_n] \in \partial \mathcal{H}$ be a sequence of eventually-golden-mean maps with the corresponding multipliers $(\lambda_{1,n},..., \lambda_{2d-2,n})$ converging to $(\lambda_1,..., \lambda_{2d-2})$ strongly (see Definition \ref{defn:strong}).

By Theorem \ref{thm:pop}, the pulled-off constant is uniformly bounded in this case.
By Theorem \ref{thm:uthind}, $[f_n]$ has uniformly bounded degeneration, so $[f_n] \to [f] \in \mathcal{M}_d$ by Theorem \ref{thm:ubd}.
By construction, the corresponding multiplier profile of $f$ is $(\lambda_1,..., \lambda_{2d-2})$, and the theorem follows.
\end{proof}

\appendix

\section{Degenerations of Riemann surfaces}\label{subsec:dr}
In this section, we introduce some terminologies to study degenerations of Riemann surfaces using extremal length and extremal width.
There is a wealth of sources containing background material on this topic (see \cite{A73}, \cite{Kah06} or \cite[Appendix 4]{KL09}).
We will briefly summarize the necessary minimum.

\subsection{Arcs and simple closed curves}
Let $X$ be a compact Riemann surface with boundary.
An {\em arc} $\gamma$ of $X$ is a continuous map
$$
h: [0,1] \longrightarrow X
$$
with $h(0), h(1) \in \partial X$.
We shall not differentiate the continuous map with its image $\gamma$ in $X$.

We say two arcs $\gamma_0, \gamma_1$ are homotopic, denoted by $\gamma_0 \sim \gamma_1$, if there exists a continuous path in the space of all arcs that connects $\gamma_0$ and $\gamma_1$.
This means that there exists a continuous map 
$$
H:[0,1] \times [0,1] \longrightarrow X
$$
with $H(t,0) = \gamma_0(t)$, $H(t,1) = \gamma_1(t)$, $H(0, s), H(1, s) \in \partial X$.

We remark that this is different from homotopy relative to $\partial X$, as we allow the homotopy to slide points on the boundary $\partial X$.

An arc $\gamma$ is said to be {\em peripheral} if it is a homotopic to an arc that is contained in a boundary component of $X$.
Note that each component of $\partial X$ is a circle and an arc is peripheral if and only if it is homotopic to a point.

Similarly, a {\em closed curve} $\alpha$ of the Riemann surface $X$ is a continuous map
$$
h: \mathbb{S}^1 \longrightarrow X.
$$
We do not differentiate the continuous map with its image $\alpha$ in $X$.
Two closed curves are {\em homotopic} if the two continuous maps are homotopic. 
We denote this by $\alpha_0 \sim \alpha_1$.
It is said to be {\em simple} if $h$ is an embedding.

For simplicity, we refer to both arc and closed curve as curves.

\subsection{Extremal length and extremal width}
Let $\mathcal{F}$ be a family of curves on $X$.
Given a (measurable) conformal metric $\rho = \rho(z)|dz|$ on $X$, let
$$
L(\mathcal{F}, \rho):= \inf_{\gamma \in \mathcal{F}} L(\gamma, \rho),
$$
where $L(\gamma, \rho)$ stands for the $\rho$-length of $\gamma$. The extremal length of $\mathcal{F}$ is
$$
\mathcal{L}_X(\mathcal{F}):= \sup_\rho \frac{L(\mathcal{F}, \rho)^2}{A(X, \rho)},
$$
where $A(U, \rho)$ is the area of $X$ with respect to the measure $\rho^2$, and the supremum is taken over all $\rho$ subject to the condition $0 < A(X, \rho) < \infty$.
The extremal width of $\mathcal{F}$ is defined as the inverse of the extremal length:
$$
\mathcal{W}_X(\mathcal{F}) = \frac{1}{\mathcal{L}_X(\mathcal{F})}.
$$

\subsubsection{Series law and parallel law}
One of the key properties of the extremal width is that it behaves like resistance in an electric circuit.

We say a family of curves $\mathcal{F}$ {\em overflows} another family of curves $\mathcal{G}$ if every curve $\gamma \in \mathcal{F}$ contains a subcurve $\gamma' \in \mathcal{G}$.
By definition, if $\mathcal{F}$ overflows $\mathcal{G}$, then
$$
\mathcal{W}_X(\mathcal{F}) \leq \mathcal{W}_X(\mathcal{G}).
$$

We say $\mathcal{F}$ {\em disjointly overflows} two families $\mathcal{G}_1, \mathcal{G}_2$ if any curve $\gamma \in \mathcal{F}$ contains the disjoint union $\gamma_1 \sqcup \gamma_2$ of two curves $\gamma_i \in \mathcal{G}_i$ (see Figure \ref{fig:SPL}).
If $\mathcal{F}$ disjointly overflows $\mathcal{G}_1, \mathcal{G}_2$, then the Gr\"otzsch inequality states that 
\begin{align}\label{eq:sl}
\mathcal{W}_X(\mathcal{F}) \leq \mathcal{W}_X(\mathcal{G}_1) \bigoplus \mathcal{W}_X(\mathcal{G}_2),
\end{align}
where $x\bigoplus y = \frac{1}{\frac{1}{x}+\frac{1}{y}}$ is the harmonic sum.
We shall refer to Equation \ref{eq:sl} the series law.

\begin{figure}[ht]
  \centering
  \resizebox{\linewidth}{!}{
    \def\svgwidth{\columnwidth}
\begingroup%
  \makeatletter%
  \providecommand\color[2][]{%
    \errmessage{(Inkscape) Color is used for the text in Inkscape, but the package 'color.sty' is not loaded}%
    \renewcommand\color[2][]{}%
  }%
  \providecommand\transparent[1]{%
    \errmessage{(Inkscape) Transparency is used (non-zero) for the text in Inkscape, but the package 'transparent.sty' is not loaded}%
    \renewcommand\transparent[1]{}%
  }%
  \providecommand\rotatebox[2]{#2}%
  \newcommand*\fsize{\dimexpr\f@size pt\relax}%
  \newcommand*\lineheight[1]{\fontsize{\fsize}{#1\fsize}\selectfont}%
  \ifx\svgwidth\undefined%
    \setlength{\unitlength}{841.88976378bp}%
    \ifx\svgscale\undefined%
      \relax%
    \else%
      \setlength{\unitlength}{\unitlength * \real{\svgscale}}%
    \fi%
  \else%
    \setlength{\unitlength}{\svgwidth}%
  \fi%
  \global\let\svgwidth\undefined%
  \global\let\svgscale\undefined%
  \makeatother%
  \begin{picture}(1,1.01010101)%
    \lineheight{1}%
    \setlength\tabcolsep{0pt}%
    \put(0,0){\includegraphics[width=\unitlength,page=1]{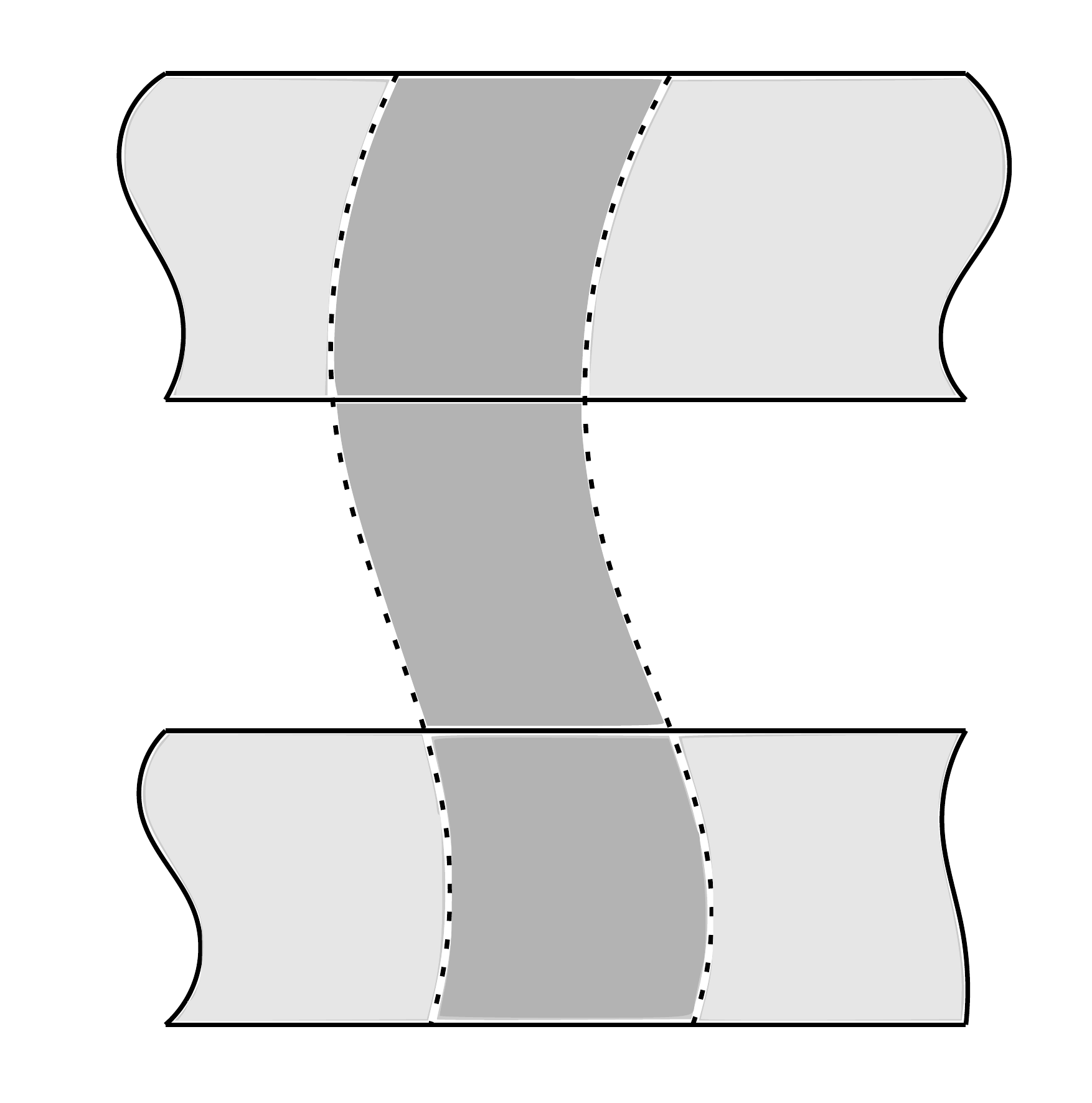}}%
    \put(0.64395938,0.78495272){\color[rgb]{0,0,0}\makebox(0,0)[lt]{\lineheight{1.25}\smash{\begin{tabular}[t]{l}{\Huge $\mathcal{G}_1$}\end{tabular}}}}%
    \put(0.7168426,0.20705165){\color[rgb]{0,0,0}\makebox(0,0)[lt]{\lineheight{1.25}\smash{\begin{tabular}[t]{l}{\Huge$\mathcal{G}_2$}\end{tabular}}}}%
    \put(0.39104495,0.49976043){\color[rgb]{0,0,0}\makebox(0,0)[lt]{\lineheight{1.25}\smash{\begin{tabular}[t]{l}{\Huge$\mathcal{F}$}\end{tabular}}}}%
  \end{picture}%
\endgroup%

    \def\svgwidth{\columnwidth}
\begingroup%
  \makeatletter%
  \providecommand\color[2][]{%
    \errmessage{(Inkscape) Color is used for the text in Inkscape, but the package 'color.sty' is not loaded}%
    \renewcommand\color[2][]{}%
  }%
  \providecommand\transparent[1]{%
    \errmessage{(Inkscape) Transparency is used (non-zero) for the text in Inkscape, but the package 'transparent.sty' is not loaded}%
    \renewcommand\transparent[1]{}%
  }%
  \providecommand\rotatebox[2]{#2}%
  \newcommand*\fsize{\dimexpr\f@size pt\relax}%
  \newcommand*\lineheight[1]{\fontsize{\fsize}{#1\fsize}\selectfont}%
  \ifx\svgwidth\undefined%
    \setlength{\unitlength}{841.88976378bp}%
    \ifx\svgscale\undefined%
      \relax%
    \else%
      \setlength{\unitlength}{\unitlength * \real{\svgscale}}%
    \fi%
  \else%
    \setlength{\unitlength}{\svgwidth}%
  \fi%
  \global\let\svgwidth\undefined%
  \global\let\svgscale\undefined%
  \makeatother%
  \begin{picture}(1,1.01010101)%
    \lineheight{1}%
    \setlength\tabcolsep{0pt}%
    \put(0,0){\includegraphics[width=\unitlength,page=1]{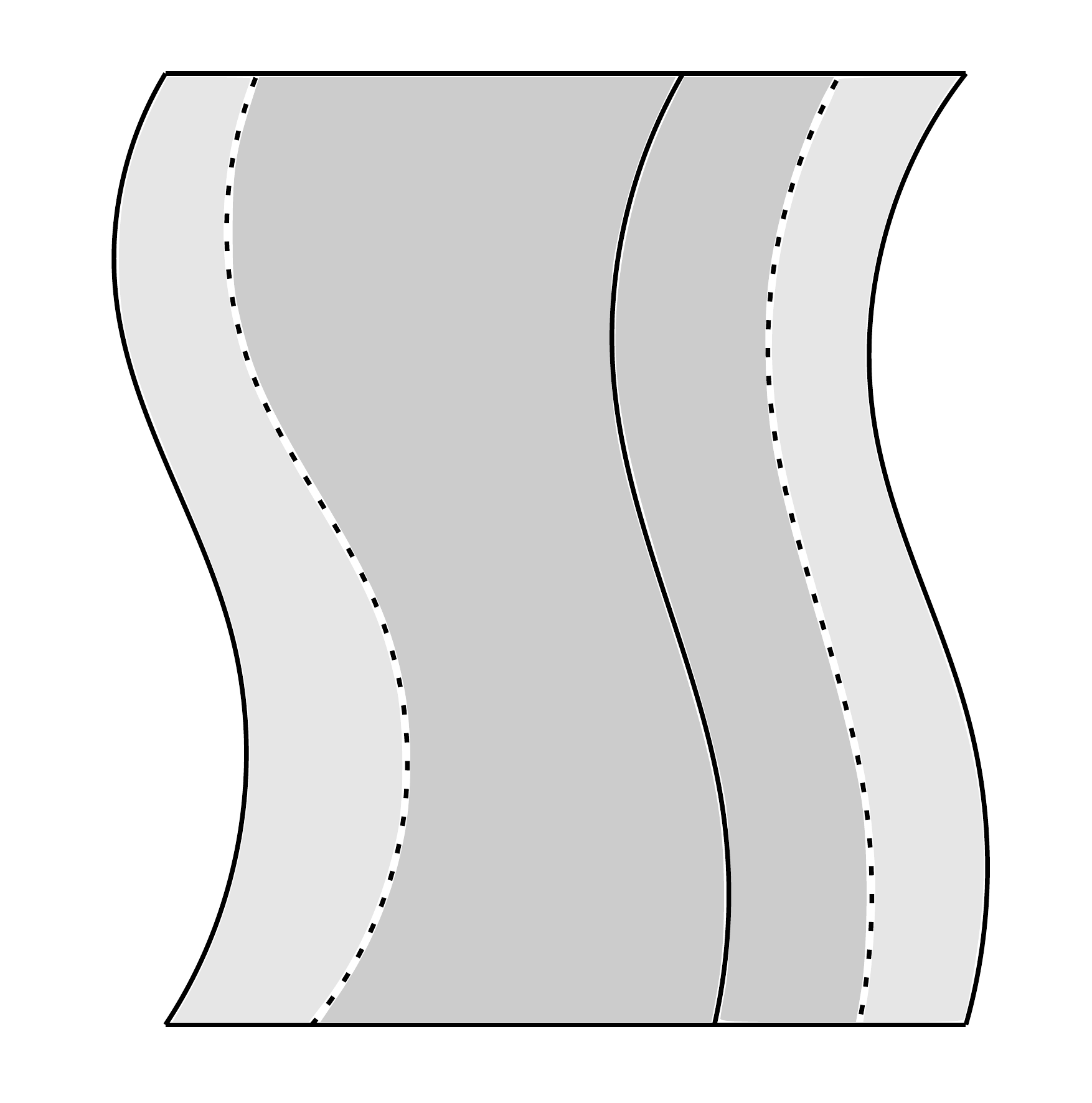}}%
    \put(0.4090225,0.63354788){\color[rgb]{0,0,0}\makebox(0,0)[lt]{\lineheight{1.25}\smash{\begin{tabular}[t]{l}{\Huge$\mathcal{F}$}\end{tabular}}}}%
    \put(0.33067029,0.33829378){\color[rgb]{0,0,0}\makebox(0,0)[lt]{\lineheight{1.25}\smash{\begin{tabular}[t]{l}{\Huge $\mathcal{G}_1$}\end{tabular}}}}%
    \put(0.734,0.26193494){\color[rgb]{0,0,0}\makebox(0,0)[lt]{\lineheight{1.25}\smash{\begin{tabular}[t]{l}{\Huge$\mathcal{G}_2$}\end{tabular}}}}%
  \end{picture}%
\endgroup%

  }
  \caption{An illustration of the series law on the left and parallel law on the right.}
  \label{fig:SPL}
\end{figure}
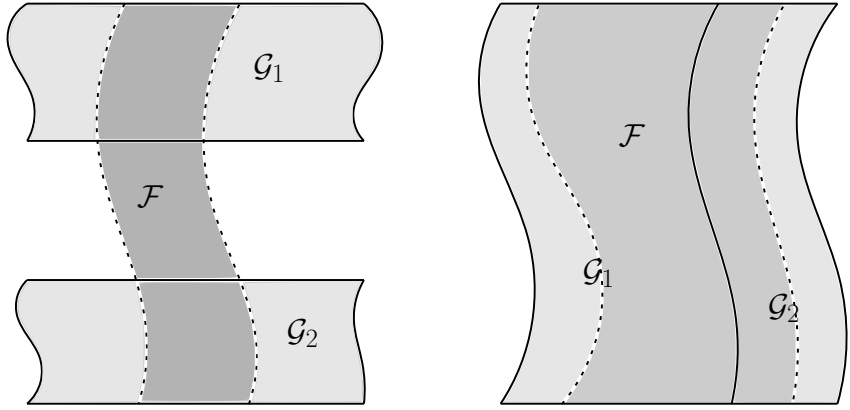

On the other hand, if $\mathcal{F} \subseteq \mathcal{G}_1 \cup \mathcal{G}_2$, i.e., every curve in $\mathcal{F}$ is either a curve in $\mathcal{G}_1$ or a curve in $\mathcal{G}_2$ (see Figure \ref{fig:SPL}), then
\begin{align}\label{eq:pl}
\mathcal{W}_X(\mathcal{F}) \leq \mathcal{W}_X(\mathcal{G}_1) + \mathcal{W}_X(\mathcal{G}_2).
\end{align}
We shall refer to Equation \ref{eq:pl} the parallel law.

\subsubsection{Extremal width between two sets}
Let $I, J \subseteq X$ be subsets of $X$.
By an arc connecting $I$ and $J$ in $X$, we mean an arc parameterized by a continuous map $\gamma:[0,1] \longrightarrow X$ with $\gamma(0) \in I$, $\gamma(1) \in J$ and $\gamma((0,1)) \subseteq X - (I\cup J)$.

We use
$\mathcal{W}_X(I, J)$
to denote the conformal widths of the family of arcs connecting $I$ and $J$ in $X$.
When the underlying Riemann surface $X = \widehat\C$, we will sometimes omit the subindex, and simply write
$$
\mathcal{W}(I, J) := \mathcal{W}_{\widehat\C}(I, J).
$$

\subsection{Euclidean rectangles and (topological) rectangles}
A \emph{Euclidean rectangle} is a rectangle $E_x\coloneqq[0,x]\times [0,1] \subset \C$, where:
\begin{itemize}
\item $(0,0), (x,0), (x,1), (0,1)$ are four vertices of $E_x$,
\item $\partial^h E_x=[0,x]\times\{0,1\}$ is the horizontal boundary of $E_x$,
\item $\partial^{h,0} E_x\coloneqq [0,x]\times\{0\}$ is the \emph{base} of $E_x$,
\item $\partial^{h,1} E_x\coloneqq [0,x]\times\{1\}$ is the \emph{roof} of $E_x$,
\item $\partial^v E_x=\{0,x\}\times [0,1]$ is the \emph{vertical} boundary of $E_x$,
\item $\partial^{v,\ell} E_x\coloneqq \{0\}\times [0,1], \sp  \partial^{v,\rho} E_x\coloneqq \{x\}\times [0,1]$ is the \emph{left} and \emph{right vertical} boundaries of $E_x$; 
\item $\Fam(E_x)\coloneqq \{\{t\}\times[0,1]\mid t\in [0,x]\}$ is the \emph{vertical foliation} of $E_x$,
\item $\Fam^\full(E_x) := \{\gamma: [0,1]\to E_x \mid \gamma(0)\in \partial ^{h,0}E_x,\ \gamma(1)\in \partial ^{h,1}E_x \}$ is the \emph{full family of curves} in $E_x$; 
\item $\Width(E_x)=\Width(\Fam(E_x))=\Width(\Fam^\full(E_x))=x$ is the \emph{width} of $E_x$,
\item $\mod (E_x)=1/\Width(E_x)=1/x$ the extremal length of $E_x$.
\end{itemize}

By a \emph{(topological) rectangle} in a Riemann surface we mean a closed Jordan disk $\mathcal{R}$ together with a conformal map $g: \mathcal{R} \longrightarrow E_x$.
We call the preimage $\partial^{h, 0}\mathcal{R}$ of $[0,x] \times\{0\}$ the {\em base}, and the preimage $\partial^{h, 1}\mathcal{R}$ of $[0,x] \times \{1\}$ the {\em roof}.
We denote the horizontal boundaries by
$$
\partial^h \mathcal{R} := \partial^{h, 0}\mathcal{R} \cup \partial^{h, 1}\mathcal{R}.
$$
Similarly, we denote the vertical boundaries by
$$
\partial^v \mathcal{R} := \partial^{v, 0}\mathcal{R} \cup \partial^{v, 1}\mathcal{R}.
$$
The width of a rectangle $R$ is 
$$
\mathcal{W}(\mathcal{R}) := \mathcal{W}_\mathcal{R}(\partial^{h, 0}\mathcal{R}, \partial^{h, 1}\mathcal{R}) = x.
$$
A {\em $K$-buffer} of a rectangle $\mathcal{R}$ is the image $g([0,K] \times [0,1] \cup [x-K,x]\times [0,1])$.

The collection of vertical arcs 
$$
\mathcal{F}_{v, \mathcal{R}} :=\{g(\{t\} \times [0,1]): t\in [0,x]\}
$$
is called the {\em vertical foliation} of the rectangle $\mathcal{R}$. Similarly, the {\em horizontal foliation} of $\mathcal{R}$ is the collection
$$
\mathcal{F}_{h, \mathcal{R}} :=\{g([0,x] \times \{t\}): t\in [0,1]\}.
$$
Abusing the notations, when we say remove $K$-buffers for the vertical foliation, we mean the foliation
$$
\mathcal{F} :=\{g(\{t\} \times [0,1]): t\in [K,x-K]\}.
$$

A \emph{genuine subrectangle} of $E_x$ is any rectangle of the form $E'=[x_1,x_2]\times [0,1]$, where $0\le x_1<x_2\le x$; it is identified with the standard Euclidean rectangle $[0,x_2-x_1]\times [0,1]$ via $z\mapsto z- x_1$. A genuine subrectangle of a topological rectangle is defined accordingly. 

\subsection{Arc and loop degenerations}\label{subsec:ald}
Let $\gamma$ be a non-peripheral arc of $X$, and $\mathcal{F}(\gamma)$ be the family of arcs homotopic to $\gamma$.
We define the degeneration for $\gamma$ as the extremal width
$$
\mathcal{W}(\gamma) := \mathcal{W}(\mathcal{F}(\gamma)).
$$
Since majority of wide rectangles do not intersect (see, for example, \cite[\S 3]{Kah06}), there are only finitely many homotopy classes of non-peripheral arcs $\gamma$ with $\mathcal{W}(\gamma) \geq 2$.
In fact, this number is bounded by the topological complexity of $X$.
We define the arc degeneration for $X$ as
$$
\mathcal{W}_{arc}(X) = \sum_{\gamma: \mathcal{W}(\gamma) \geq 2} \mathcal{W}(\gamma).
$$
Similarly, if $Z$ is a component of $\partial X$, we defined 
$$
\mathcal{W}^{loc}_{arc}(Z) = \sum_{\gamma \in \Gamma_1: \mathcal{W}(\gamma) \geq 2} \mathcal{W}(\gamma) + 2\sum_{\gamma\in \Gamma_2: \mathcal{W}(\gamma) \geq 2} \mathcal{W}(\gamma)
$$ 
where $\Gamma_1$ (or $\Gamma_2$) contains homotopy classes of non-peripheral arcs with exactly one endpoint (or two endpoints) on $Z$.
Note that by definition,
$$
2\mathcal{W}_{arc}(X) = \sum_Z \mathcal{W}^{loc}_{arc}(Z),
$$
where the sum is over all boundary components of $X$.

Similarly, let $\alpha$ be a homotopically non-trivial simple closed curve, and let $\mathcal{G}$ be the family of simple closed curves isotopic to $\alpha$.
We remark here that the curve $\alpha$ is allowed to be homotopic to a boundary component of $X$.
We define the {\em degeneration} for $\alpha$ of $X$ as the extremal width 
$$
\mathcal{W}(\alpha) := \mathcal{W}(\mathcal{G}).
$$
We define the {\em loop degeneration} for $X$ as
$$
\mathcal{W}_{loop}(X) = \sum_{\alpha: \mathcal{W}(\alpha) \geq 2} \mathcal{W}(\alpha).
$$

\begin{figure}[ht]
  \centering
  \resizebox{0.8\linewidth}{!}{
    \def\svgwidth{\columnwidth}
\begingroup%
  \makeatletter%
  \providecommand\color[2][]{%
    \errmessage{(Inkscape) Color is used for the text in Inkscape, but the package 'color.sty' is not loaded}%
    \renewcommand\color[2][]{}%
  }%
  \providecommand\transparent[1]{%
    \errmessage{(Inkscape) Transparency is used (non-zero) for the text in Inkscape, but the package 'transparent.sty' is not loaded}%
    \renewcommand\transparent[1]{}%
  }%
  \providecommand\rotatebox[2]{#2}%
  \newcommand*\fsize{\dimexpr\f@size pt\relax}%
  \newcommand*\lineheight[1]{\fontsize{\fsize}{#1\fsize}\selectfont}%
  \ifx\svgwidth\undefined%
    \setlength{\unitlength}{751.18110236bp}%
    \ifx\svgscale\undefined%
      \relax%
    \else%
      \setlength{\unitlength}{\unitlength * \real{\svgscale}}%
    \fi%
  \else%
    \setlength{\unitlength}{\svgwidth}%
  \fi%
  \global\let\svgwidth\undefined%
  \global\let\svgscale\undefined%
  \makeatother%
  \begin{picture}(1,0.60377358)%
    \lineheight{1}%
    \setlength\tabcolsep{0pt}%
    \put(0,0){\includegraphics[width=\unitlength,page=1]{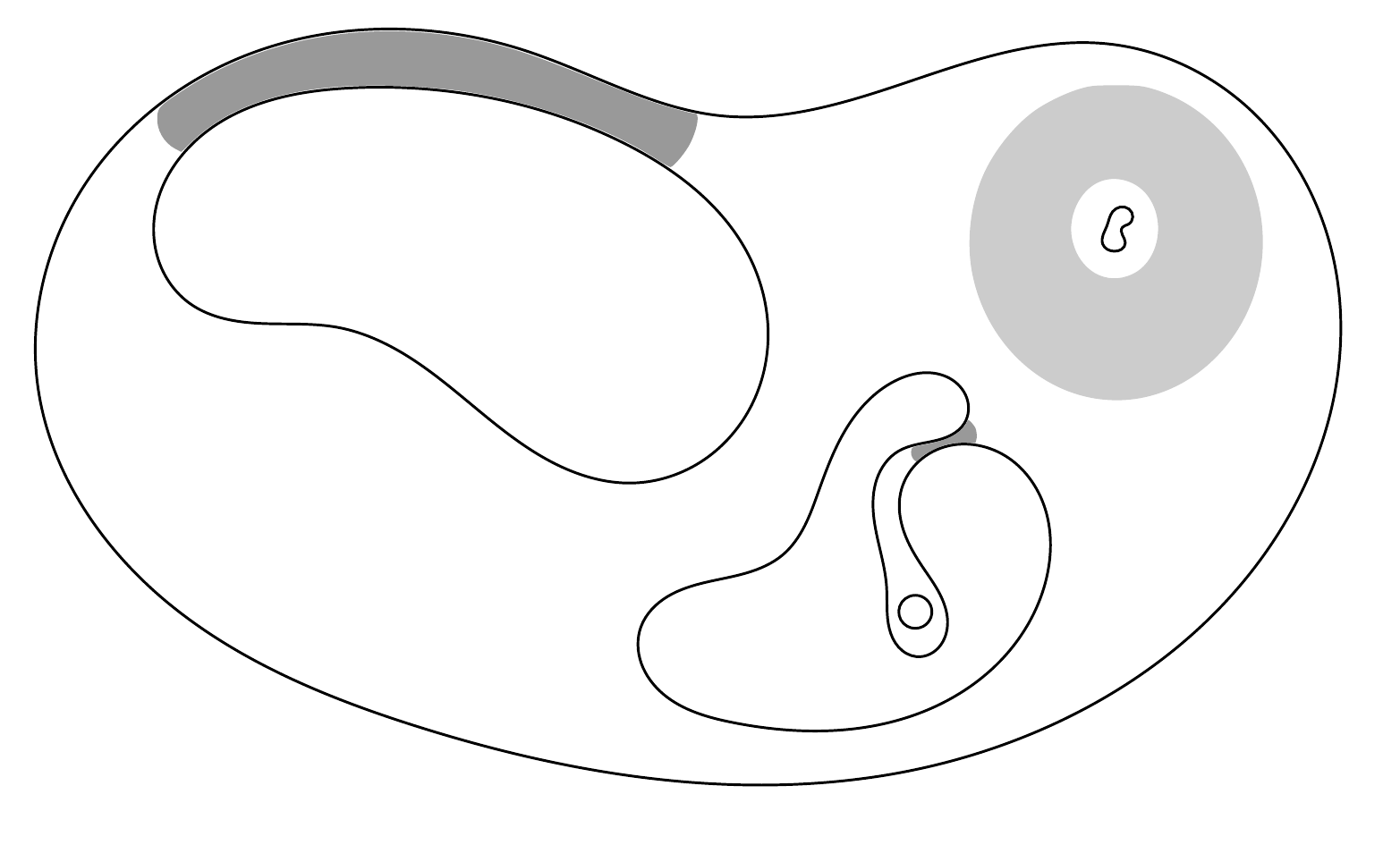}}%
    \put(0.20171284,0.20135645){\color[rgb]{0,0,0}\makebox(0,0)[lt]{\lineheight{1.25}\smash{\begin{tabular}[t]{l}{\Huge$X$}\end{tabular}}}}%
    \put(0,0){\includegraphics[width=\unitlength,page=2]{TTD.pdf}}%
  \end{picture}%
\endgroup%

  }
  \caption{An illustration of arc and loop degenerations. Here $X$ is a genus $0$ Riemann surface with 5 boundary components. Arc and loop degenerations are indicated by dark and light grey respectively.}
  \label{fig:TTD}
\end{figure}

\subsection{The Thin-Thick Decomposition}
\label{ss:TTD}
Here, we summarize a few variations of the fundamental fact that wide families of curves are supported within finitely many pairwise disjoint wide rectangles. We refer the readers to \cite[\S7.6]{Lyu} for more details.

Let $I, J \subset \partial X$ be two intervals. Let $\Fam_X(I,J)$ be the family of arcs connecting $I$ and $J$ in $X$. Then the \emph{Thin-Thick Decomposition} of $X$ rel the pair $I,J$ says that, up to $O_{\chi(X)}(1)$, we can replace the family $\Fam_X(I,J)$ by a finitely many disjoint rectangles. More precisely, there exist finitely many pairwise disjoint non-homotopic rectangles $\RR_1,..., \RR_s$ connecting $I$ and $J$, i.e,
\[ \partial^{h,0} \RR_i \subset I \sp\sp\text{ and }\sp\sp  \partial^{h,1} \RR_i \subset J,\]
 such that 
\begin{equation}
\label{eq:TTD}    
 \sum_{i=1}^s \Width(\RR_i)=\Width_X(I,J)-O_{\chi(X)}(1),
\end{equation}
where $\chi(X)$ is the Euler characteristic of $X$.
We remark that since the rectangle are disjoint and non-homotopic, $s$ is bounded by the topological complexity of the surface $X$.

The \emph{Thin-Thick Decomposition} of $X$ says that there are finitely many pairwise disjoint non-homotopic rectangles and annuli in $X$ 
\[ \mathcal T _X = (\RR_1,\RR_2,\dots, \RR_t, \mathcal{A}_1,\mathcal{A}_2,\dots \mathcal{A}_s), \sp\sp\sp \partial^h R\subset \partial X\] such that
\begin{itemize}
    \item families \[\Fam(\gamma), \sp \Fam_{arc}(X), \sp \Fam^{loc}_{arc}(Z)\] for the corresponding 
\[\Width(\gamma), \sp \Width_{arc}(X), \sp \Width^{loc}_{arc}(Z)\] 
introduced in~\S\ref{subsec:ald} are supported, up to $O_{\chi(X)}(1)$, within finitely rectangles from $\mathcal T _X$,
\item the family $\Fam_{loop}(X)$ (for $\Width_{loop}(X)$) is formed by the annuli from $\mathcal T _X$.
\end{itemize}

Given a component $Z$ of $\partial X$, the covering annuli $\mathbb A(X, Z)$ of $X$ with respect to $Z$ is obtained by opening up all loops except $Z$; see~\cite[\S 3.3.3]{DL23} for a more detailed description. Then the family $\Fam\big(\mathbb A(X, Z)\big)$ of curves in $\mathbb A(X, Z)$ connecting its boundary components is, up to $O_{\chi(X)}(1)$, supported in the univalent lifts $\RR_i^\tau \subset \mathbb A(X, Z), \tau \in \{0,1\}$ of the rectangles $\RR_i$ from $\mathcal T_X$ with
\begin{itemize}
    \item $\tau=0$ if $\partial^{h,0}\RR_i \subset Z$,
    \item  $\tau=1$ if $\partial^{h,1}\RR_i \subset Z$.
\end{itemize}
In particular, this means that 
\begin{align}
\label{eq:AC}
\Width \left( \mathbb A(X, Z)\right) &= \sum_{i, \tau} \Width(\RR_i^\tau) + O_{\chi(X)}(1) \\
\notag &= \sum_{\RR_i \in \mathcal T_X, \, \partial^{h,0}\RR_i \subset Z} \Width(\RR_i) + \sum_{\RR_i \in \mathcal T_X, \, \partial^{h,1}\RR_i \subset Z} \Width(\RR_i) + O_{\chi(X)}(1)\\
\notag &= \mathcal{W}^{loc}_{arc}(Z) + O_{\chi(X)}(1).
\end{align}

\section{Siegel $\psi^\bullet$-ql maps and psuedo-Siegel disks}\label{ap:psd}
In this appendix, we summarize pseudo-Siegel bounds from \cite{DL22, DLL25} adopted to $\psi^\bullet$-ql maps.

We recall that $\psi$-quadratic-like maps were introduced in~\cite{Kah06}. They generalize the notion of quadratic-like maps with the goal of explicitly relating the geometry of various renormalizations of quadratic polynomials. It is essential for the theory that the post-critical set of the map is $\iota$-proper (see~\S\ref{sss:qlb:Defn}). In~\cite{Kah06}, it is assumed that the filled-in Julia set itself is $\iota$-proper in the definition of $\psi$-ql maps. In this appendix, we will consider $\psi^\bullet$-ql Siegel maps  with the requirement that the closed Siegel disk (and its iterated preimages) is $\iota$-proper. We refer to~\cite{DL23} for a related notion of $\psi^\bullet$-ql ``bush'' maps. For technical reasons, we require in Item~\ref{dfn:psib-ql:2} that $\iota$ is a covering onto its image in the complement of the Siegel disk.

\subsection{$\psi^\bullet$-ql Siegel maps}
\label{sss:qlb:Defn}
A map $\iota\colon A\to B$ between open Riemann surfaces is called an \emph{immersion} if every $x\in X$ has a neighborhood $U_x$ such that $\iota\colon  U_x\to \iota(U_x)$ is a conformal isomorphism. Immersions arising in applications are compositions of covering maps and embeddings in various orders. A compact subset $S\Subset B$ is called \emph{$\iota$-proper} if $\iota\mid \iota^{-1}(S)\to S$ is a homeomorphism. In this case, we often \emph{identify} $S\simeq \iota^{-1}(S)$. 

We say that an immersion $\iota\colon A\to B$ is a \emph{covering embedding rel $S\subset B$} if 
\begin{itemize}
  \item $S$ is $\iota$-proper;
  \item $\iota\mid A\setminus \iota^{-1}(S)$ is a covering onto the image; i.e., 
  \[\iota \colon \ A\setminus \iota^{-1}(S)\longrightarrow \iota(A)\setminus S\]
is a covering map. 
\end{itemize}

\begin{defn}
\label{dfn:psi:Siegel:ql}
A {\em  pseudo${}^\bullet$-quadratic-like Siegel map} (``$\psi^\bullet$-ql Siegel map'') 
is a pair of holomorphic  maps
\begin{equation}
\label{eq:dfn:psib-ql}
F=(f,\iota)\colon \sp ( U, \overline Z_U)  \rightrightarrows (V, \overline Z ),\hspace{0.5cm}\text{ so }\sp  \overline Z_U \subseteq f^{-1}(\overline Z) \cap  \iota^{-1}(\overline Z) 
\end{equation} between two conformal  disks $U$, $V$
with the following properties:

\begin{enumerate}[label=\text{(\Roman*)},font=\normalfont,leftmargin=*]
\item\label{dfn:psib-ql:1}  $f\colon U\to V$ is a double branched covering with a unique critical point $c_0\in \partial Z_U$;

\item \label{dfn:psib-ql:2} $\overline Z$ is $\iota$-proper; in particular, $\iota \colon \overline Z_U\ \overset{\simeq}\longrightarrow \ \overline Z;$

\item \label{dfn:psib-ql:3}
 there exist neighborhoods $X_U\supset \overline Z _U$ and $X \supset \overline Z$ with the following property: $\iota: X_U \to X$ is a conformal isomorphism such that
\begin{equation}\label{eq:dfn:ql b domain}
f_X\coloneqq f\circ \big(\iota\mid X_U\big)^{-1} : X \to f(X_U)\eqqcolon Y 
\end{equation} 
 is a \emph{Siegel map}: $\overline Z \Subset X\cap Y$ is a closed qc Siegel disk around the fixed point $\alpha \in Z=\intr \overline Z$ with bounded-type rotation number;

\end{enumerate}
Define inductively $K_0\coloneqq \overline Z$, $K_{1, U}\coloneqq f^{-1}(\overline Z)$, $K_{1}\coloneqq\iota(K_{1, U})$,  and, for $n\ge 1$   
\begin{equation}
\label{eq:C:dfn:K_n}
\begin{matrix} K_{n,U}\coloneqq f^{-1}(K_{n-1})=f^{-1}\circ \big(\iota \circ f^{-1}\big)^{n-1} (\overline Z), \vspace{0.1cm}\\ K_n \coloneqq \iota(K_{n,U})=\big(\iota \circ f^{-1}\big)^n (\overline Z);\hspace{1.2cm} \end{matrix}
\end{equation}
\begin{enumerate}[label=\text{(\Roman*)},font=\normalfont,leftmargin=*,start=4]
\item  \label{dfn:psib-ql:4} for all $n\ge 0$, the restriction
$\iota \colon \ K_{n,U} \overset{\simeq}{\longrightarrow} K_n $ is a homomorphism;

\item \label{dfn:psib-ql:5} $\iota\colon U\to V$ is a covering embedding rel $K_1$;
\end{enumerate}

We remark that $F\colon U\rightrightarrows V$ can be naturally iterated $F^n\colon U^n\rightrightarrows V$ as fiber product, see~\cite[\S 2]{DLL25}. We also require that
\begin{enumerate}[label=\text{(\Roman*)},font=\normalfont,leftmargin=*,start=6]
\item \label{dfn:psib-ql:6} for all $n\ge 1$, the iterates $\iota^n\colon U^n \to V$ are covering embedding rel $K_n$.
\end{enumerate}
\end{defn}

Since $\iota$ is a conformal isomorphism in a neighborhood of $\overline Z$, we will below identify 
\[\overline Z\simeq \overline Z_U\equiv \overline Z_F\sp\sp \text{ and write }\sp F\colon ( U, \overline Z)  \rightrightarrows (V, \overline Z) \sp\sp\text{ or }\sp F\colon U \rightrightarrows V.\]
Similarly, we identify $K_{n}\simeq K_{n,U}$.

The \emph{width} of $F$ is 
\[\Width^\bullet(F)\coloneqq \Width(V\setminus \overline Z).\]
 If $\Width^\bullet(F)\le K$, then $X_U\simeq ~X$ in Item~\ref{dfn:psib-ql:3} can be selected so that 
\begin{equation}
    \label{eq:compt of psi maps}\mod(X\setminus \overline Z)\ge \varepsilon(K).
 \end{equation} 
Therefore, $\psi^\bullet$-ql Siegel maps $f$ with $\Width^\bullet(f)\le K$ form a compact set.

\begin{example} Consider a quadratic rational map $g\in \partial_\egm\mathcal{H}_{z^2}$, where $H_{z^2}$ is the hyperbolic component of $z\mapsto z^2$, see~\S\ref{ss:intro:z^2}. Assume that $g$ has closed Siegel qc-disks $\overline Z_0, \overline Z_\infty$ at $0$ and $\infty$. We naturally obtain two $\psi^\bullet$-ql maps:
\[G_0= (g, \hookrightarrow )\  \colon\  (U_0, \overline Z_0)\rightrightarrows (V_0,  \overline Z_0), \sp\sp V_0=\wC \setminus \overline Z_\infty, \ U_0=g^{-1}(V_0),\]
\[G_\infty= (g, \hookrightarrow )\  \colon\  (U_\infty, \overline Z_\infty)\rightrightarrows (V_\infty,  \overline Z_\infty), \sp\sp V_\infty=\wC \setminus \overline Z_0, \ U_\infty=g^{-1}(V_\infty),\]
where the immersion $\iota=$``$\hookrightarrow$'' is an embedding. We have:
\[\Width^\bullet(G_0) =\Width^\bullet(G_\infty)=\Width\big(\wC\setminus [\overline Z_0\cup \overline Z_\infty]\big).\]

\end{example}

\subsection{$\psi^\bullet$-ql renormalization} \label{subsec:qlfromr} Consider a disjoint type hyperbolic component $\mathcal H$ and an eventually-golden mean map $[f]\in \partial_\egm \mathcal H$; see Definitions~\ref{dfn:HH_fm} and \ref{defn:EGM maps}. The construction below is an adaptation of $\psi$-ql renormalization from~\cite{Kah06}; see also~\cite[\S2.4]{DLL25} and \cite[\S3]{DL23}.

Consider a periodic Siegel disk $Z_{s}=Z_{i,j}$ of $f$ with period $p\ge 1$. We will now define a $\psi^\bullet$-ql map associated with $Z$. Write
\[X \coloneqq \widehat\C - \bigcup_{i,j} D_{i,j} - \bigcup_{i,j} Z_{i,j}\sp\sp \text{ and }\sp\sp
X'\coloneqq f^{-p} (X).\]
Since $X'\subset X$, we obtain a correspondence:
\begin{equation}
\label{eq: corr: X' X}
(f^p,\ \hookrightarrow)\colon \ X'\rightrightarrows X, 
\end{equation}
where $\hookrightarrow $ is a natural embedding. Consider the covering $\widetilde X\to X$ opening up all loops except $\partial Z$; in particular, $\widetilde X$ is an annulus. Similarly, the covering $\widetilde X'\to X'$ opens up all loops except (slightly thickened) $\partial (Z\cup Z')$, where $Z'$ is the unique preperiodic $f^p$-lift of $Z$ attached to $Z$. 

Then~\eqref{eq: corr: X' X} induces a correspondence 
\[  F=( f^p,\iota)\colon \widetilde X' \rightrightarrows \widetilde X,\] 
where $f^p$ is a $2:1$ covering map and $\iota$ is an immersion obtained by lifting ``$\hookrightarrow$''. (In fact, $\iota$ is a covering onto its image). Gluing $\widetilde X$ with $\overline Z$ and gluing $\widetilde X'$ with $\overline{Z\cup Z'}$, we obtain a $\psi^\bullet$-ql map 
\begin{equation}
\label{eq:psi bullet renorm}
     F=( f^p,\iota) \colon U\rightrightarrows V
\end{equation}

The Thin-Thick Decomposition in \S\ref{ss:TTD}, or more precisely, the Equation \ref{eq:AC} implies that
\[\Width^\bullet\big( F\big) = \Width^{loc}_{arc}(Z) + O(1),\]
where $\Width^{loc}_{arc}$ is defined in \S \ref{subsec:ald}. Moreover, the rectangles in the Thin-Thick Decomposition of $X$ adjacent to $Z$ lift univalently into the dynamical plane of $F$; their lifts are disjoint rectangles connecting $\partial V$ and $\partial Z$ with total width being $\Width(F)-O(1)$.

\subsection{A priori-bounds for  $\psi^\bullet$-ql Siegel maps}\label{subsec:aprioribounds}
\hide{Let $F$ be an eventually-golden-mean $\psi^\bullet$-ql map as in~\S\ref{sss:qlb:Defn}.
We say that a disk $D\supset Z$ is {\em peripheral rel $\overline Z$} if $D\subset V$. Similar to~\S\ref{sss:zW:conventions}, a set $S\subset V$ is {\em peripheral rel $\overline Z$} if $S$ is within a peripheral disk $D$.} 

The definition of pseudo-Sigel disks for $\psi^\bullet$-ql Siegel maps is the same as Definition~\ref{defn:pseudo-Siegel disks} (for maps in $\partial_\egm \mathcal H$) with no peripheral requirements as in ~\S\ref{sss:zW:conventions} -- every set in $\intr V$ is peripheral rel $\overline Z$. In fact, in this case, we require that the territory $\XX(\wZ^m)$ is a topological disk contained in the {\em costal zone} - the region in $V$ bounded by the core curve of $V\setminus Z$; see \cite[\S3.5]{DLL25} for more discussions. (Actually, such a requirement is satisfied automatically after removing $1$-buffers from outer protections.)

Let $D\supset Z$ be a peripheral disk.
We say a curve $\gamma$ in $V$ is {\em vertical} (rel $D$) if $\gamma$ connects $\partial D$ and $\partial V$ in $V$, and we say it is {\em peripheral} (rel $D$) if $\partial \gamma \subseteq \partial D$. 

Let $\lambda \geq 1$ and let $I$ be an interval on $\wZ^m$.
The families $\Fam^{+,ver}_{\wZ^m}(I), \Fam^{+,per}_{\lambda, \wZ^m}(I)$ and their corresponding widths $\Width^{+,ver}_{\wZ^m}(I), \Width^{+,per}_{\lambda, \wZ^m}(I)$ are defined accordingly as in~\S\ref{subsec:sfc}. We remark that here $+$ in the superscript means that the curves in the family have interior contained in in $V - \wZ^m$.

Let $\bK\gg 1$ be a sufficiently big threshold.
Let $K_F\coloneqq \Width^\bullet(F)$ be the width of $F$.
We define the {\em special transition level} $\bbm_F$ for $F$ with respect to $\bK$ as follows. 
\begin{itemize}
    \item If $K_F\le \bK$, we set $\bbm_F\coloneqq -2$;
    \item Otherwise, we set $\bbm_F$ to be the level satisfying 
\[ \frac{\bK}{\length_{\bbm_F}} < K_F \le \frac{\bK}{\length_{\bbm_F\ +1}} \sp\sp\sp  \begin{array}{c}\text{ or, }\\ \text{equivalently,}\end{array}\sp\sp\sp \sp   \begin{array}{c}
\length_{\bbm_F} K_F>\bK, \text{ and}\vspace{0.2cm}\\
\length_{\bbm_F\ +1} K_F\le \bK.
\end{array} \]
\end{itemize} 
We recall that $\length_{-1}=1$ and $\length_0=\dist(x, f(x))$.

The following theorem is proved in \cite[Theorem 1.3]{DLL25}.
\begin{theorem}\label{thm:apbs}
There exists an absolute constant $\bK\gg 1$ so that the following holds.

Consider an eventually-golden-mean $\psi^\bullet$-ql map $F$ (see~\S\ref{sss:qlb:Defn}) of width $K_F\coloneqq \Width^\bullet(F)$ and the special transition level $\bbm_F$. Then there is an increasing sequence of pseudo-Siegel disks $\wZ^m,\  m\ge -1$ such that for every grounded interval $J\subset \partial Z$ with $\length_{m+1}< |J|\le \length_{m}$ the following holds:
\begin{enumerate}[ label=(\Alph*)]
    \item\label{thm:apbs:A} if $m> \bbm_F$, then \[\Width_{\wZ^m}^{+,ver} (J^m)=O(1) \sp\sp\text{ and }\sp\sp \Width_{3,\wZ^m}^{+,per} (J^m)\asymp 1,\]
    \item\label{thm:apbs:B} if $m< \bbm_F$, then 
    \[\Width_{\wZ^m}^{+,ver} (J^m)\asymp |J|K_F \sp\sp\text{ and }\sp\sp \Width_{3,\wZ^m}^{+,per} (J^m)=O(1),\]
    \item \label{thm:apbs:C} if $m= \bbm_F$, then 
    \[\Width_{\wZ^m}^{+,ver} (J^m)= O(\length_m K_F)  \sp\sp\text{ and }\sp\sp \Width_{3,\wZ^m}^{+,per} (J^m)=O(\sqrt{\length_m K_F }).\]
\end{enumerate}

Moreover, $\wZ^{-1}$ is $M(K_F)$-qc disk; i.e.~the dilatation of $\wZ^{-1}$ is bounded in terms of $K_F$.

\end{theorem}
We remark that in Cases~\ref{thm:apbs:B} and~\ref{thm:apbs:C}, we have
$ |J|K_F, \ \length_m K_F \  \ge\  1  $.

We also remark that in all three cases, we have the following bounds
\begin{align}
\label{eqn:VandP}
\begin{split}
    \Width_{\wZ^m}^{+,ver} (J^m)&= O(\length_m K_F +1) \vspace{0.3cm} \\ \Width_{3,\wZ^m}^{+,per} (J^m)&=O(\sqrt{\length_m K_F }+1).
\end{split}
\end{align}

\begin{rmk}
\label{rem:thm:info}   
In short, $\psi^\bullet$-formalism stated in Theorem~\ref{thm:apbs} takes care of all scales except the special transition scale $m=\bbm_F$.  Case \ref{thm:apbs:A} says that on deep scales, the geometry of $F$ is uniformly bounded, and the estimates are equivalent to that of quadratic polynomials. Case \ref{thm:apbs:B} says that on shallow scales, vertical degeneration dominates peripheral and is uniformly distributed at all intervals.

Theorem~\ref{thm:apbs} does not provide a satisfactory description of $\Width^{+,ver}$ and $\Width_{3}^{+,per}$ in Case \ref{thm:apbs:C}. In our paper, such information comes from the global analysis of pseudo-Core surface degenerations stated in Theorem~\ref{thm:ll} and Theorem~\ref{thm:calls}; see Remark~\ref{rem:to:CL}.
\end{rmk}

For an explicit construction of ``geodesic'' pseudo-Siegel disks satisfying Theorem~\ref{thm:apbs}, see~\S\ref{ss:good:wZ}.

\subsection{Localization of submergence}
 Let us say that a rectangle $\RR$ \emph{submerges} into a pseudo-bubble $\wZ_i$ if
\begin{itemize}
    \item $\partial^h \RR$ is disjoint from $\XX(\wZ_i)$; and
    \item every curve $\gamma\in \Fam(\RR)$ intersects $\wZ_i$.
\end{itemize}

\begin{figure}[htp]
    \centering
    \begin{tikzpicture}
    \node[anchor=south west,inner sep=0] at (0,0) {\includegraphics[width=0.8\textwidth]{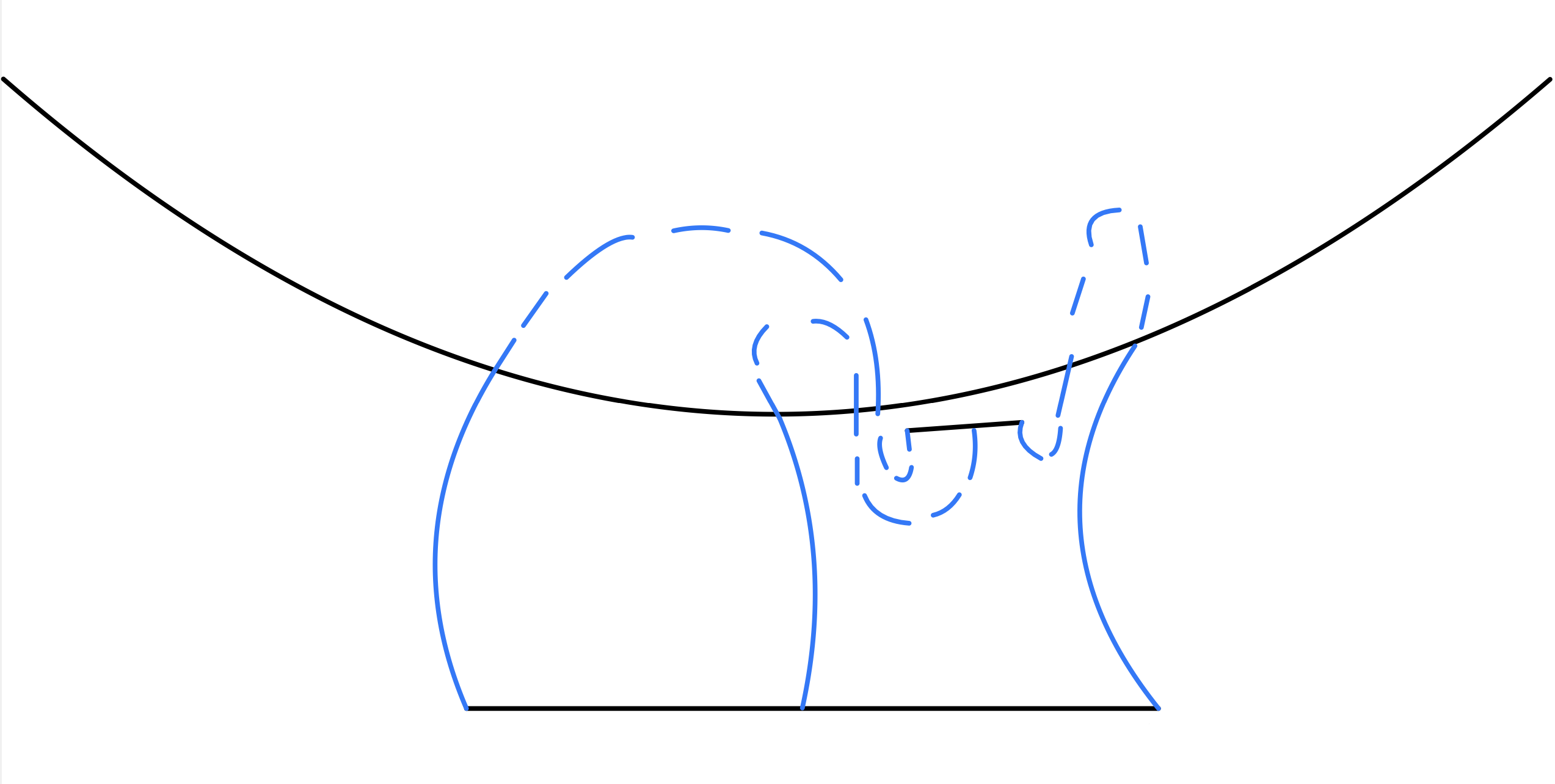}};
    \node at (4,1.2) {\begin{normalsize}$\textcolor{red}{A_0}$\end{normalsize}};
    \node at (6.3,1.2) {\begin{normalsize}$\textcolor{red}{A_1}$\end{normalsize}};
    \node at (2.5,1.5) {\begin{footnotesize}$\gamma_0$\end{footnotesize}};
    \node at (5.6,1.5) {\begin{footnotesize}$\gamma_1$\end{footnotesize}};
    \node at (7.4,1.5) {\begin{footnotesize}$\gamma_2$\end{footnotesize}};
    \node at (7.4,0.2) {\begin{footnotesize}$\partial^{0,h}\RR$\end{footnotesize}};
    \node at (9,3.4) {\begin{footnotesize}$\partial\wZ_i$\end{footnotesize}};
    \end{tikzpicture}
    \caption{An illustration of a rectangle submerging into a pseudo-bubble.}
    \label{fig:PH}
\end{figure}
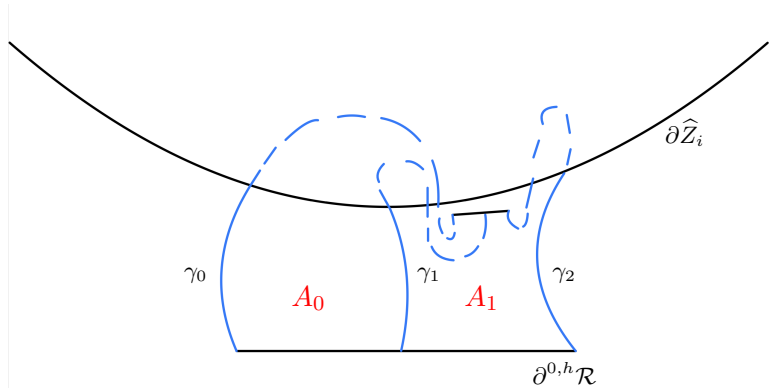

\begin{lem}
\label{lem:rect through PB}
Assume that a rectangle $\RR$ with $\Width(\RR)=K$ submerges into a pseudo-bubble $\wZ_i$. Then for every $\lambda>2$, there is 
\begin{itemize}
   \item  a grounded interval $J\subset \partial \wZ_i$ with $|J|<\frac{1}{\lambda^2}$, and
   \item sublamination $\widetilde {\mathcal Q}\subset \Fam(\RR)$ that overflows a lamination $\mathcal Q$ outside of $\wZ_i$ with  $\Width({\mathcal Q})\succeq K- O(\ln \lambda)$
\end{itemize}
such that
\begin{itemize}
    \item either $\mathcal Q$ is a lamination from $J$ to $\partial^{h,1} \RR$;
    \item or $\mathcal Q \subset \Fam^+ (J, \partial \wZ_i\setminus [\lambda J]^c)$; i.e., $\mathcal Q$ is a lamination outside of $\wZ_i$ from $J$ to $\partial \wZ_i\setminus (\lambda J)$.
\end{itemize}
\end{lem}
\begin{proof}
    Let $\widetilde{\gamma_0}, \widetilde{\gamma_1}, \widetilde{\gamma_2}$ be the leftmost, middle and rightmost vertical arcs of the rectangle $\RR$.
    We orient these arcs so that they connects the lower boundary $\partial^{h,0}\RR$ to $\partial^{h,1}\RR$.
    Since $\RR$ submerges into $\wZ_i$, $\widetilde{\gamma_j}$ intersects $\wZ_i$.
    Let $a_j$ be the first time $\widetilde{\gamma_j}$ enters $\wZ_i$, and let $\gamma_j \subseteq \widetilde{\gamma_j}$ be the sub arc connecting $\partial^{h,0}\RR$ and $a_j$.
    Let $A_0, A_1$ be the region bounded by $\gamma_0, \gamma_{1}, \partial^{h,0}\RR, \partial \wZ_i$ and $\gamma_1, \gamma_{2}, \partial^{h,0}\RR, \partial \wZ_i$ as illustrated in Figure \ref{fig:PH}).
    Since $\partial^{h,1}\RR$ is disjoint from $\XX(\wZ_i)$, at least one of the regions $A_0, A_1$ is disjoint from $\partial^{h,1}\RR$.
    Without loss of generality, we may assume $A_0$ is disjoint from $\partial^{h,1}\RR$.
    Consider the left rectangle $\RR' \subseteq \RR$ bounded by $\widetilde{\gamma_0}, \widetilde{\gamma_1}, \partial^{h,0}\RR, \partial^{h,1}\RR$, and let $I$ be the interval on $\partial \wZ_i$ bounded by $a_0, a_1$.
    Then for every vertical $\gamma$ arc connecting $\partial^{h,0}\RR'$ to $\partial^{h,1}\RR'$, 
    \begin{itemize}
    \item the first intersection of $\gamma$ with $D$ is in $I$; and
    \item the last intersection of $\gamma$ with $D$ is in $I^c=\partial D\setminus I$.
    \end{itemize}
    With this reduction, we can directly apply \cite[Lemma 6.9] {DL22}, and the lemma follows.
\end{proof}

\subsubsection{From full to outer families}
We need the following submergence results in our main application.
\begin{lem}
\label{lem:rect loc}
Let $\wZ$ be a pseudo-Siegel disk.
Let $\RR$ be a rectangle with $\Width(\RR)=K$ such that the $I:= \partial^{h,0} \RR$ is a grounded interval on $\wZ$, and $\partial^{h,1} \RR$ is disjoint from $\XX(\wZ)$.
Then for every $\lambda>2$, there is either 
\begin{itemize}
    \item a genuine subrectangle  $\RR_1$ of $\RR$ with $\Width(\RR_1) \succeq K $ such that $\RR_1$ is outside of $\intr \wZ$; or   
    \item a grounded interval $J\subset \partial \wZ$ such that $\Width^{+, per}_\lambda(J) \succeq K - O(\ln \lambda)$; in particular,  $|J| < \frac{1}{\lambda}$ if $K\gg  \ln \lambda$.
\end{itemize}
\end{lem}
\begin{proof}
Assume that there are no genuine subrectangle $\RR_1$ of $\RR$ with $\Width(\RR_1)\succeq \Width(\RR)$. Then a substantial part of $\Fam(\RR)$ submerges into $\intr \wZ^m$ and we have two cases:
\begin{itemize}
    \item either a substantial part of $\Fam(\RR)$ first submerges into $\intr \wZ^m$ in $(\lambda^3 I)^c $;
    \item or a substantial part of $\Fam(\RR)$ first submerges into $\intr \wZ^m$ in $\lambda^3 I$.
\end{itemize}
In the first case, we take $J:= I$. The second case follows from ~\cite[Corrolary 6.2]{DL22} applied to either the pair $I\cup L_-,\ (\lambda^3I)^c$ or to the pair  $I\cup L_+,\ (\lambda^3I)^c$, where $L_-, L_+$ are two intervals in $(\lambda^3I)\setminus I$.
\end{proof}

\subsection{Geodesic pseudo-Siegel disks} \label{ss:good:wZ} In this subsection, we summarize an explicit construction of \emph{geodesic} pseudo-Siegel disks $\wZ$ satisfying Theorem~\ref{thm:apbs} for an eventually-golden-mean $\psi^\bullet$-ql map $F$. We refer the readers to \cite[\S7 and Theorem 7.5]{DLL25} for more discussions.

Let $\bbm_F$ be the transitional level of $F$ defined in~\S\ref{subsec:aprioribounds} with respect to a sufficiently big but universal constant $\bK$. There exist a sufficiently big $\bM\gg 1$ and an increasing function $H\colon \R_{>0}\to \R_{\ge \bM}$ such that the following holds. We set
\begin{align}
    \label{eq:dfn:bM} 
    \begin{split}
    \bM_m &\coloneqq \begin{cases}
    \bM, \,\, \text{ if $m>\bbm_F$,}\\
     H(\length _m K_F) \,\, \text{ if $m = \bbm_F$,}\\
    \infty, \,\, \text{ if $m < \bbm_F$.}
\end{cases}      
    \end{split}
\end{align}

 We say that a level $m$ is \emph{near-parabolic} if $\length_{m}> \bM_m \length_{m+1}$; otherwise $m$ is \emph{non-parabolic}. Since $F$ is eventually-golden-mean, all sufficiently deep levels $m\gg_F\  1$ are non-parabolic.
 In short, $\bM_m$ will be a combinatorial threshold for regularization: if $\frac{\length_{m}}{\length_{m+1}}\ge \bM_m$, then $\wZ^{m+1}$ is regularized into $\wZ^m$ at depth $e^{\sqrt{ \ln \bM_{m}}}$, see~\S\ref{sss:dfn:regular}; otherwise $\wZ^m:=\wZ^{m+1}$. 
 We remark that by our definition of $\bM_m$, if $m<\bbm_F$, then we always set $\wZ^m := \wZ^{m+1}$.

\hide{In short, $\bM_m$ will be a combinatorial threshold for regularization: if $\frac{\length_{m}}{\length_{m+1}}\ge \bM$ and $m> \bbm_F$, where $\bbm_F$ is the transition level defined in~\S\ref{subsec:aprioribounds}, then $\wZ^{m+1}$ is regularized into $\wZ^m$ at depth $\bM_{\sqrt{\ln}}$, see~\S\ref{sss:dfn:regular}; otherwise $\wZ^m:=\wZ^{m+1}$. If $m<\bbm_F$, then we always set $\wZ^m := \wZ^{m+1}$. The case $m=\bbm_F$ is discussed in~\S\ref{sss:dfn:regular:m_G}.}

\subsubsection{Construction of parabolic fjord $\widehat{\mathfrak F}_I$ and $S^\inn _I$}\label{sss:dfn:regular} Consider a near-parabolic level $m$. Let $I=[a,b]\in \mathfrak{D}_m$ be an interval in the $m$th diffeo-tiling of $\partial Z$. Choose $a', b'\in I$ with $a<a'<b'<b$ such that \[\dist(a,a')=\dist(b',b) =\left\lfloor e^{\sqrt{ \ln \bM_{m}}}\right\rfloor \  \length_{m+1} \]
and set $\beta_I$ to be the hyperbolic geodesic of $V\setminus \overline Z$ connecting $a',b'$. This defines the parabolic fjord $\widehat{\mathfrak F}_I$, see Figure~\ref{fig:CD}.
We remark that since the level $m$ is near-parabolic, we have
$$
|I| = \dist(a,b) \asymp \length_m \geq \bM_m \length_{m+1} \gg \left\lfloor e^{\sqrt{ \ln \bM_{m}}}\right\rfloor \  \length_{m+1}.
$$

To construct $S^\inn _I$, we choose a sufficiently big $\bbv\gg 1$ with the understanding that $e^{\sqrt{ \ln \bM_{m}}} \gg \bbv$.
Choose $a''\in [a,a']$ and $b''\in [b',b]$ such that
\begin{equation}
\label{eq:dfn:S^inn}
    \dist(a,a'')=\dist(b'',b) =\left\lfloor \frac{e^{\sqrt{ \ln \bM_{m}}}}{ \bbv}\right\rfloor \length_{m+1} \gg \length_{m+1},
\end{equation} 
this defines $S^\inn _I$, see Figure~\ref{fig:CD}.

\subsubsection{Construction of $A_I$ and $\upbullet \XX_I$}
Let $\mathfrak G_I$ be the rectangle from $[a,a'']$ to $[b'', b]$ bounded by hyperbolic geodesics in $V\setminus \overline Z$; i.e., 
\[ \partial^{h,0} \mathfrak G_I =[a,a''],\sp\sp \partial^{h,1} \mathfrak G_I =[b'',b]\]
and $\partial^{v}\mathfrak G_I $ is the pair of hyperbolic geodesics. The condition that $e^{\sqrt{ \ln \bM_{m}}} \gg \bbv$ implies that (see \cite[\S 4]{DL22})\[ \Width (\mathfrak G_I) \ \asymp\  \ln \left(e^{\sqrt{ \ln \bM_{m}}}/\bbv \right) \ \gg \ 1.\]
We can now select a required rectangle $\XX_I$ conformally deep in $\mathfrak G_I$ (i.e., conformally close to $\beta_I$) so that its width is of the size $\Delta \ll  \sqrt{ \ln \bM_{m}}$. We can also select $A_I$ separating $\XX_I$ from $S^\inn_I$. In particular, we can assume that the interval \[[x_a, x_b]\coloneqq \partial \upbullet \XX_I \cap I\sp\sp\sp\text{ with }\sp x_a\in[a,a''], \sp x_b\in [b'', b]\]
is defined similar to \eqref{eq:dfn:S^inn}:
\[     \dist(a, x_a)=\dist(x_b,b) =\left\lfloor \frac{e^{\sqrt{ \ln \bM_{m}}}}{ \bbw}\right\rfloor \length_{m+1}, \]
where $\bbw \gg \bbv\gg 1$ with the understanding that we still have $e^{\sqrt{ \ln \bM_{m}}}\gg \bbw$.

\subsubsection{Stability of ${\wZ^m}$}\label{subsubsect:stabilityConstruction}
Since $e^{\sqrt{ \ln \bM_{m}}}\gg \bbw$, we see that the construction guarantees that $\dist(\partial^h \XX_I, \partial I) = \left\lfloor \frac{e^{\sqrt{ \ln \bM_{m}}}}{ \bbw}\right\rfloor \length_{m+1} \gg \length_{m+1}$.
Thus, by the discussion in \S~\ref{subsec:stability}, we see that ${\wZ^m}$ can be assumed to be $T$-stable for arbitrarily large $T$ (see Remark~\ref{rmk:stablitiy}).


\end{document}